\documentclass[12pt]{amsart}

\usepackage[utf8]{inputenc}
\usepackage{amsthm}
\usepackage{amsmath}
\usepackage{amscd}
\usepackage[mathscr]{euscript}
\usepackage{tikz-cd}
\usepackage{url}
\usepackage{hyperref}
\usepackage{amssymb}
\usepackage[margin=25truemm]{geometry}

\newtheorem{theorem}{Theorem}
\numberwithin{theorem}{section}
\newtheorem{proposition}[theorem]{Proposition}
\newtheorem{lemma}[theorem]{Lemma}
\newtheorem{corollary}[theorem]{Corollary}
\theoremstyle{definition} 
\newtheorem{definition}[theorem]{Definition}
\theoremstyle{remark}

\newtheorem{remark}[theorem]{Remark}

\numberwithin{theorem}{section}
\numberwithin{equation}{section}
\providecommand{\C}{\mathbb C} 

\providecommand{\Q}{\mathbb Q}
\providecommand{\RR}{\mathbb R}

\providecommand{\ZZ}{\mathbb Z}

\providecommand{\charac}{\operatorname{char}}

\providecommand{\supp}{\operatorname{Supp}}
\providecommand{\Val}{\operatorname{Val}}

\providecommand{\Ray}{\operatorname{Ray}}
\providecommand{\Spec}{\operatorname{Spec}}

\providecommand{\A}{\mathbb A}

\providecommand{\PP}{\mathbb P}

\providecommand{\pr}{\mathrm{pr}}
\providecommand{\Gr}{\mathrm{Gr}}

\providecommand{\id}{\mathrm{id}}
\providecommand{\VOL}{\mathrm{Vol}_{\mathrm{sb}}}
\providecommand{\Rt}{\mathscr{R}}
\providecommand{\Kt}{\mathscr{K}}
\providecommand{\spe}{\operatorname{sp}}
\providecommand{\bdd}{\operatorname{bdd}}

\providecommand{\SB}{\mathrm{SB}}

\providecommand{\OO}{\mathscr O}

\providecommand{\rank}{\operatorname{rank}}

\providecommand{\ker}{\operatorname{ker}}

\title{Stable rationality of hypersurfaces in sch\"{o}n affine varieties}
\author{Taro Yoshino}
\date{\today}
\address{Graduate School of Mathematical Sciences, The University of Tokyo, 3-8-1 Komaba,
Meguro-ku, Tokyo, 153-8914, Japan}
\email{yotaro@ms.u-tokyo.ac.jp}

\begin{document}

\begin{abstract}
  In recent years, there has been a development in approaching rationality problems through the motivic methods (cf. [Kontsevich--Tschinkel'19], [Nicaise--Shinder'19], [Nicaise--Ottem'21]).

  This method requires the explicit construction of degeneration families of curves with favorable properties.
  While the specific construction is generally difficult, [Nicaise--Ottem'22] combines combinatorial methods to construct degeneration families of hypersurfaces in toric varieties and shows the non-stable rationality of a very general hypersurface in projective spaces.

  In this paper, we extend the result of [Nicaise--Ottem'22] not only for hypersurfaces in algebraic tori but also to those in sch\"{o}n affine varieties.
  In application, we show the irrationality of certain hypersurfaces in $\Gr_\C(2, n)$ using the motivic method, which coincides with the result obtained by the same author in the previous research.
\end{abstract}
\maketitle
\section{Introduction}
 \subsection{The motivic method for the rationality problem}
    The rationality problem is one of the central problems in algebraic geometry. 
    The study of the rationality of hypersurfaces in projective spaces, in particular, is a hot topic in recent research. 
    Here, we will introduce one of the most important results in this area in recent years as follows: 
    \begin{proposition}\cite[Corollary 1.2]{S19}\label{prop: the rationality of hypersurfaces}
        Let $k$ be an uncountable field of $\charac(k)\neq 2$, and $n \geq 3$ and $d \geq 2 + \log_2(n)$ be integers. 
        Then a very general hypersurface $H_d\subset\PP^{n+1}_k$ of degree $d$ is not stably rational, i.e., $H_d\times\PP^m_k$ is not rational for any $m\in \ZZ_{\geq0}$.  
    \end{proposition}
    Note that recent years have seen improvements in the non-stably rational range of degrees and dimensions by \cite{Moe23} and \cite{LS24}.
    
    In the proof of Proposition \ref{prop: the rationality of hypersurfaces}, the diagonal decomposition and the unramified cohomology are key methods. 
    On the other hand, recently, there exists another approach -- the motivic approach -- for the rationality problem(cf. \cite{KT19}, \cite{NO21}, \cite{NS19}). 

    First, we introduce a summary of the motivic method for the rationality problem. 
    Let $K$ be a field.  
    Two $K$-schemes of finite type $X$ and $Y$ are stably birational if $X\times \PP^m_K$ is birational to $Y\times \PP^n_K$ for some non-negative integers $m$ and $n$. 
    If a $K$-variety $X$ is stably birational to $\Spec(K)$, then we call that $X$ is stably rational. 
    In particular, a rational $K$-variety $X$ is stably rational. 
    Let $\SB_K$ denote the set of stable birational equivalence classes $\{X\}_\mathrm{sb}$ of integral $K$-schemes $X$ of finite type, and $\ZZ[\SB_K]$ be the free abelian group on $\SB_K$.    
    In particular, $X$ is stably rational over $K$ if and only if $\{X\}_\mathrm{sb} = \{\Spec(K)\}_\mathrm{sb}$ for any variety $X$ over $K$ by the definition. 
    
    Let $k$ be an algebraically closed field of $\charac(k) = 0$. 
    Let $\Rt$ be a valuation ring defined as follows:
            \[
                \Rt = \bigcup_{n\in\ZZ_{>0}} k[[t^{\frac{1}{n}}]].
            \]
    Let $\Kt$ be the fraction field of $\Rt$. 
            We remark that $\Kt$ is written as follows:
            \[
                \Kt = \bigcup_{n\in\ZZ_{>0}} k((t^{\frac{1}{n}})).
            \]
    In \cite{NO21}, they constructed a ring morphism as follows: 
    \begin{proposition}\label{prop: NONO21}\cite[Lemma.3.3.5]{NO21}
        There exists a unique ring morphism $\VOL\colon \ZZ[\SB_{\Kt}]\rightarrow \ZZ[\SB_k]$ such that for every strictly toroidal proper $\Rt$-scheme $\mathscr{X}$ with the smooth generic fiber $X = \mathscr{X}_{\Kt}$, we have 
        \[
            \VOL(\{X\}_{\mathrm{sb}}) = \sum_{E\in \mathcal{S}(\mathscr{X})}(-1)^{\mathrm{codim}(E)}\{E\}_{\mathrm{sb}},
        \] 
        where $\{E\}$ is a stratification of $\mathscr{X}_k$. 
    \end{proposition}
    At this point, we omit the detailed definition of a strictly toroidal model(See Definition \ref{prop: NONO21}). 
    We remark that such ring morphisms appeared in these references(cf. \cite{KT19}, \cite{NO21}, \cite{NS19}), although the details of the forms differ. 
    In application, for any $\Kt$-variety $X$, if $\VOL(\{X\}_{\mathrm{sb}})\neq \{\Spec(k)\}_\mathrm{sb}$, then $X$ is not stably rational over $\Kt$. 
    In particular, this ring morphism $\VOL$ gives a birational invariant that detects irrationality, and this is called the motivic method.

    There are various applications of the motivic method for mentioning nonstable rationality, and we will excerpt and list several of these results as follows: 
    \begin{itemize}
        \item A very general quartic fivefold(\cite[Corollary 5.2]{NO22}). 
        \item A very general hypersurface in the product of two projective spaces of some degrees and dimensions(\cite[Proposition 6.2]{NO22}).  
        \item A very general complete intersections of some degrees and dimensions(\cite[Corollary 7.6]{NO22}). 
        \item A very general hypersurfaces in the quadric of some degrees and dimensions(\cite[Corollary 7.9]{NO22}). 
    \end{itemize} 
    As we can see, results for various varieties have been obtained through the motivic method. 
    This paper aims to discuss the rationality of hypersurfaces in the Grassmannian variety $\Gr_\C(2, n)$ using the motivic method as follows:
    \begin{theorem}[See Theorem \ref{thm: d in Grassmannian}]\label{thm: main theorem}
            If a very general hypersurface of degree $d$ in $\PP_\C^{2n-5}$ is not stably rational, then a very general hypersurface of degree $d$ in $\Gr_\C(2, n)$ is not stably rational.
    \end{theorem}
    In particular, the following corollary holds by Proposition \ref{prop: the rationality of hypersurfaces} immediately:  
    \begin{corollary}[See Theorem \ref{cor: log bound}]
        If $n\geq 5$ and $d \geq 3 + \log_2(n-3)$, then a very general hypersurface of degree $d$ in $\Gr_\C(2, n)$ is not stably rational.         
    \end{corollary}
    \subsection{The difficulty of the motivic method}
    While the motivic method looks quite simple, in practice, several problems could be solved.  
    For example, we have the following three problems : 
    \begin{itemize}
        \item [1.] To construct a ``good'' proper and flat model with a smooth generic fiber explicitly. 
        \item [2.] To enumerate strata of the closed fiber of this model. 
        \item [3.] To determine the stable birational equivalence classes (or birational equivalence classes) for each stratum.
    \end{itemize}
    Previous work \cite{NO22} successfully overcomes these three problems as follows(cf. \cite[Theorem. 3.14]{NO22}):
    \begin{itemize}
        \item [1.] They focused on a hypersurface $Y$ in an algebraic torus. 
        This hypersurface has a tropical compactification $\overline{Y}$ (cf.\cite{Tev}); in particular, they are compactified as a hypersurface in a proper toric variety $X$. 
        Moreover, if $Y$ is sch\"{o}n, then $\overline{Y}$ has a toroidal structure.
        This property was applied to the construction of a good model in \cite{NO22}.
        As an additional note, Bertini's theorem implies that a general hypersurface is sch\"{o}n.
        \item [2.] The ambient space is a toric variety, and the orbit decomposition of $X$ induces the stratification of $\overline{Y}$. 
        Moreover, each stratum is also a hypersurface in an algebraic torus, and we can calculate the defining Laurant polynomial of each stratum combinatorially.  
        \item [3.] They construct the strictly toroidal model $\mathscr{X}$ of a quartic 5-fold, which has a stratum $E\in\mathcal{S}(\mathscr{X})$, which is birational to a very general quartic double fourfold. In particular, $E$ is not stably rational (cf.\cite{HPT19}). 
        In addition to this, they showed that $\VOL(\{\mathscr{X_\Kt}\}_{\mathrm{sb}})\neq \{\Spec(k)\}_{\mathrm{sb}}$, and a very general quartic 5-fold is not stably rational. 
    \end{itemize}
\subsection{The significance of this paper}
    In the previous section, we observed that the construction of strictly toroidal models and the computation of stably birational volumes are relatively straightforward for hypersurfaces in algebraic tori.
    
    It is natural to consider applying the method of \cite{NO22} to study the rationality of hypersurfaces in other rational varieties. 
    For example, Grassmannian varieties are. 
    Grassmannian varieties have a Schubert decomposition; in particular, they have an algebraic torus as a dense open subspace. 
    Thus, it seems straightforward to construct a strictly toroidal model of the variety, which is birational to the hypersurface in the Grassmannian variety. 
    
    However, this does not work well because the defining Laurant polynomials of these hypersurfaces cannot be written as linear combinations of finite units. 
    In particular, we cannot apply Bertini's theorem. 
    Consequently, the construction of strictly toroidal models from hypersurfaces in algebraic tori faces certain limitations. 

    The tropical compactification of sch\"{o}n hypersurfaces in algebraic tori has been a significant focus in previous studies. 
    This motivates us to extend the discussion of compactification to hypersurfaces in not only algebraic tori but also sch\"{o}n affine varieties and explore whether the rationality problem of these hypersurfaces can be studied using motivic methods.
    This paper provides one possible answer to that question. 
    The following three points are particularly significant in the content of this paper:
    \begin{itemize}
        \item[I.] We constructed a strictly toroidal scheme using the tropical compactification of a sch\"{o}n affine variety and computed its stably birational volume (cf. Proposition \ref{prop: model I} and Proposition \ref{prop: model II}). 
        This work extends the results of previous studies. 
        A key theorem is Lemma \ref{lem: smooth to smooth}, which demonstrates that the tropical compactification of a sch\"{o}n affine variety has a toroidal structure. 
        The results naturally extend from this fact.
        \item[II.] Given a torus-invariant finite-dimensional linear system $\mathfrak{d}$ of an algebraic torus $T$ in which a schön affine variety $Z$ is embedded, we proved that the intersection of $Z$ with a general divisor in $\mathfrak{d}$ remains a sch\"{o}n affine variety (cf. Proposition \ref{prop: general condition of schon}). 
        This extends the classical result that a general hypersurface in an algebraic torus is a sch\"{o}n affine variety.
        The key to the proof lies in the set of valuations on $Z$ induced by $T$, which are used to construct the tropical compactification of hypersurfaces in $Z$ (cf. Section 4.1).
        \item[III.] Using the discussions so far, we proved Theorem \ref{thm: main theorem}. 
        It is known that there exists an open subset of $\Gr_\C(2, n)$ that is a sch\"{o}n affine variety. 
        By considering hypersurfaces of $\Gr_\C(2, n)$ as hypersurfaces of this open subset, it follows from (II) that a general hypersurface in $\Gr_\C(2, n)$ is a sch\"{o}n affine variety. 
        Then, using (I), we constructed a strictly toroidal scheme and computed its stably birational volume. 
        The explicit construction of a strictly toroidal scheme whose stably birational volume does not coincide $\{\Spec(\C)\}_{\mathrm{sb}}$ is the most significant contribution of this paper.
    \end{itemize}
    Moreover, we will provide a few additional remarks on Theorem \ref{thm: main theorem}. 
    \begin{itemize}
        \item The key point of the proof of Theorem \ref{thm: main theorem} is the result of the stable birational equivalence class of a very general hypersurface in a projective space(\cite[Corollary. 4.2]{NO22}). 
        \item We remark that for any integer $m\geq 3$, we cannot find these models of hypersurfaces in $\Gr_\Kt(m, n)$ related to sch\"{o}n affine varieties.
    \end{itemize} 
    \subsection{The rationality problem for hypersurfaces in Grassmannian varieties}
        We discuss previous works on the rationality of hypersurfaces in $\Gr_\C(2, n)$. 
        We fix a Pl\"{u}cker embedding $\Gr_\C(2, n)\hookrightarrow \PP^{n(n-1)/2 - 1}_\C$. 
        Let $X$ denote $\Gr_\C(2, n)$, and $X_{(r)}$ denote the intersection of $X$ and hypersurfaces of degree = $r$ in $\PP^{n(n-1)/2 - 1}_\C$. 
        Because $K_X = \OO_X(-n)$, a general $X_{(r)}$ is a Fano variety for $1\leq r\leq n - 1$. 
        We itemize the rationality of $X_{(r)}$ as follows:
        \begin{itemize}
            \item $X_{(1)}$ is rational(\cite[Theorem 2.2.1.]{Xu11}). 
            \item $X_{(2)}\subset \Gr_\C(2, 4)$ is complete intersection of 2 quadrics in $\PP^5_\C$, so it is a rational. 
            \item A very general $X_{(3)}\subset \Gr_\C(2, 4)$ is not stably rational(\cite[Theorem 1.1]{HT19})
            \item $X_{(2)}\subset \Gr_\C(2, 5)$ is a Gushel-Mukai 5-fold. 
            Moreover, it is rational(\cite[Proposition 4.2]{DK18}).
            \item Ottem showed that a very general $X_{(3)}\subset \Gr_\C(2, 5)$ is not stably rational in his unpublished paper(\cite{O23}). 
            \item We can show that a very general $X_{(4)}\subset \Gr_\C(2, 5)$ is not stably rational by \cite[Theorem 7.1.2]{Simen20} and \cite[Corollary 4.2]{NO22}. 
        \end{itemize}
    
    \subsection{Outline of the paper}
    This paper is organized as follows. 
    In Section 2, we organize the notation used in this article. 
    In Section 3, we recall the definition of tropical compactifications and sch\"{o}n compactifications, and we consider their basic properties. 
    In Section 4, we construct a strictly toroidal scheme with a smooth generic fiber from sch\"{o}n affine varieties and compute its stable birational volume. 
    In Section 5, we show that general hypersurfaces in a sch\"{o}n affine variety are also sch\"{o}n varieties. 
    In Section 6, we apply the result in the previous sections to the stable rationality problem of a very general hypersurface in $\Gr_\C(2, n)$. 
    In Section 7, we prove the lemmas needed in
this article. 
    \subsection{Relationship with mock toric varieties}
    The author of this paper previously obtained the same main result(\cite{Y24b}). 
    In the earlier work, the author used a variety called a mock toric variety(\cite{Y24a}), which generalizes toric varieties, to study good compactifications of hypersurfaces in mock toric varieties. 
    In fact, since a proper mock toric variety is a sch\"{o}n compactification(cf. \cite[Proposition 3.9]{Y24a}), this paper also extends the previous result \cite{Y24b}. 
    \subsection{Acknowledgment}
        The author is grateful to his supervisor, Yoshinori Gongyo, for his encouragement.
\section{Notation}
    In this paper, we use the following notation according to \cite{CLS11} and \cite{Ful93}:
    \begin{itemize}
    \item Let $k$ be an algebraically closed field of $\charac(k) = 0$. 
    \item Let $\Rt$ be a valuation ring defined as follows:
            \[
                \Rt = \bigcup_{n\in\ZZ_{>0}} k[[t^{\frac{1}{n}}]].
            \]
    \item Let $\Kt$ be the fraction field of $\Rt$. 
            We remark that $\Kt$ is written as follows:
            \[
                \Kt = \bigcup_{n\in\ZZ_{>0}} k((t^{\frac{1}{n}})).
            \] 
    \item Let $\SB_K$ denote the set of stable birational equivalence classes $\{X\}_\mathrm{sb}$ of integral $K$-schemes $X$ of finite type.
    In particular, $X$ is stably rational over $K$ if and only if $\{X\}_\mathrm{sb} = \{\Spec(K)\}_\mathrm{sb}$ for any variety $X$ over $K$.
    \item Let $\ZZ[\SB_K]$ be the free abelian group on $\SB_K$. 
    \item Let $S$ be a monoid. 
    If there exists a lattice $M$ of finite rank and a full and strongly convex rational polyhedral cone $\sigma$ in $M_\RR$ such that $S$ is isomorphic to  $\sigma\cap M$, then we call $S$ is a \textbf{toric monoid}.   
    \item Let $V$ be an $\RR$-linear vector space of finite dimension, let $W$ be the dual space of $V$, and let $B$ be a subset of $V$. 
    Then $B^\perp$ and $B^\vee$ denote as the following subsets of $W$: 
    \[
        B^\perp = \{w\in W\mid w(x) = 0\quad(\forall x\in B)\},
    \]
    \[
        B^\vee = \{w\in W\mid w(x) \geq 0\quad(\forall x\in B)\}.
    \]
    \item Let $N$ denote a lattice of finite rank and let $M$ denote the dual lattice of $N$. 
    \item Let $\sigma$ be a convex cone in $N_\RR$ and let $\tau$ be a face of $\sigma$.   
    Then we write $\tau\preceq \sigma$. 
    \item Let $\sigma$ be a convex cone in $N_\RR$. 
    Then $\sigma^\circ$ denotes the relative interior of $\sigma$. 
    \item Let $\Delta$ be a strongly convex rational polyhedral fan in $N_\RR$.  
    Then $X_k(\Delta)$ denotes a toric variety corresponding to $\Delta$ over $k$. 
    We sometimes write $X(\Delta)$ instead of $X_k(\Delta)$. 
    \item A ring $k[M]$ denotes the $k$-algebra associated with a monoid $M$. 
    For $\omega\in M$, $\chi^\omega$ denotes the monomial in $k[M]$ associated with $\omega$. 
    \item An affine algebraic group $T_N$ denotes $\Spec(k[M])$. 
    Note that $X(\{0_N\}) = T_N$. 
    \item Let $\Delta$ be a strongly convex rational polyhedral fan in $N_\RR$. 
    For $\sigma\in\Delta$, $O_\sigma$ denotes the orbit of the torus action of $X(\Delta)$ corresponding to $\sigma$. 
    \item Let $\sigma$ be a convex cone in an $\RR$-linear space $V$. 
    Let $\langle\sigma\rangle$ denote an $\RR$-linear subspace of $V$ that is spanned by $\sigma$. 
    \item Let $N$ and $N'$ be lattices of finite rank and $f\colon N'\rightarrow N$ be a homomorphism, and let $\Delta$ and $\Delta'$ be strongly convex rational polyhedral fans in $N_\RR$ and $N'_\RR$ respectively. 
    If for any $\tau\in\Delta'$, there exists  $\sigma\in\Delta$ such that $f_\RR(\tau)\subset\sigma$, then we call that $f$ is \textbf{compatible with} the fans $\Delta'$ and $\Delta$. 
    \item With the notation above, the map $f_*$ denotes the toric morphism from $X(\Delta')$ to $X(\Delta)$ induced by $f$. 
    \item Let $\Delta_!$ denote a convex fan $\{\{0\}, [0, \infty), (-\infty, 0]\}$ in $\RR$. 
    \item Let $\Delta_1, \Delta_2$ be fans in $\RR^n$ and $\RR^m$ respectively, and let $\Delta_1\times\Delta_2$ denote the following fan in $\RR^{n+m}$: 
    \[
        \{\sigma_1\times\sigma_2\mid\sigma_1\in \Delta_1, \sigma_2\in \Delta_2\}.
    \]
    \item Let $N$ be a lattice of finite rank, let $M$ be the dual lattice, and let $\langle\cdot,\cdot\rangle$ be a pairing of $N$ and $M$. 
    For $v\in N$, we can identify $v$ as a $T_N$-invariant valuation on $T_N$, which is trivial on $k$, and $v(f)$ denotes a value of $f\in k(M)$ by $v$, where $k(M)$ denotes the fraction field of $k[M]$. 
    On this identification, $v(\chi^\omega) = \langle v,\omega\rangle$ for any $v\in N$ and any $\omega\in M$. 
\end{itemize}
\section{Sch\"{o}n affine varieties}
To compute the stable birational volume, we need to construct a toroidal scheme explicitly. 
In this section, we examine the properties of tropical compactifications and sch\"{o}n compactifications of a closed subvariety of an algebraic torus. 
Following that, we will introduce some examples of sch\"{o}n compactification. 

Note that we do not need to assume that the field $k$ is algebraically closed of $\charac(k) = 0$ in this section. 
\subsection{Definition of tropical compactification}
In this subsection, we introduce the notion of tropical compactifications and sch\"{o}n compactifications (cf. \cite{Tev}). 
\begin{definition}\cite[Definition. 1.1, 1.3]{Tev}\label{def: tropical compactification}
  Let $k$ be a field, $N$ be a lattice of rank $n$, and $Z$ be a closed subscheme of $\mathbb{G}^n_{m, k} = T_N$.
  Let $\Delta$ be a strongly convex rational polyhedral fan in $N_\RR$ and $\overline{Z^{X(\Delta)}}$ be the scheme theoretic closure in $X(\Delta)$.
  \begin{itemize}
    \item[1.] If $\overline{Z^{X(\Delta)}}$ is proper over $k$ and the multiplication morphism $m\colon T_N\times \overline{Z^{X(\Delta)}}\rightarrow X(\Delta)$ is faithfully flat, we call that $\overline{Z^{X(\Delta)}}$ is a tropical compactification of $Z$.
    \item[2.] In addition to this assumption, if $m$ is a smooth morphism, we call that $\overline{Z^{X(\Delta)}}$ is a sch\"{o}n compactification of $Z$.
    \item[3.] If $Z$ has a sch\"{o}n compactification, we call that $Z$ is a sch\"{o}n affine variety. 
  \end{itemize}
\end{definition}
In this paper, we use the following facts, which are referenced in \cite{Tev} and \cite{MS15}, respectively:
\begin{proposition}\cite[Theorem. 1.2]{Tev}\cite[Proposition. 6.4.17]{MS15}\label{prop: proper and schon}
  We keep the notation in Definition \ref{def: tropical compactification}, and we assume $Z$ is integral.
  Then the following statements hold:
  \begin{itemize}
    \item[(a)] There exists a tropical compactification of $Z$.
    \item[(b)] Let $\Delta$ be a fan in $N_\RR$, which provides a tropical compactification of $Z$.
    Then it holds that $\supp(\Delta) = \mathrm{Trop}(Z)$.
  \end{itemize}
\end{proposition}
\subsection{Properties of tropical compactification}
In this subsection, we explore the characterization of tropical compactifications and sch\"{o}n compactifications. 
They have properties similar to those of ambient toric varieties in Definition \ref{def: tropical compactification}. 
\begin{proposition}\label{prop: property1}
  Let $N$ be a lattice of finite rank, let $\Delta$ be a strongly convex rational polyhedral fan in $N_\RR$, and let $Z$ be a closed subscheme of $\mathbb{G}^n_{m, k} = T_N$.
  We assume that the multiplication morphism $m\colon T_N\times \overline{Z^{X(\Delta)}}\rightarrow X(\Delta)$ is flat.
  Then the following statements hold:
  \begin{itemize}
    \item[(a)] A subset $\{\sigma\in\Delta\mid \overline{Z^{X(\Delta)}}\cap O_\sigma \neq \emptyset\}$ of $\Delta$ is a subfan of $\Delta$.
    \item[(b)] Let $d$ be a nonnegative integer. 
    If the dimension of all irreducible components of $Z$ is $d$, then the dimensional of all irreducible components of $\overline{Z^{X(\Delta)}}\cap O_\sigma$ is $d - \dim(\sigma)$ for any $\sigma\in\Delta$ such that $\overline{Z^{X(\Delta)}}\cap O_\sigma\neq \emptyset$.
    \item[(c)] The multiplication morphism $m$ is smooth if and only if $\overline{Z^{X(\Delta)}}\cap O_\sigma$ is smooth for any $\sigma\in\Delta$.
    \item[(d)] We assume that $m$ is smooth.
      Then $\overline{Z^{X(\Delta)}}$ is a normal scheme and a Cohen-Macaulay scheme.
      Moreover, let $W_1, \ldots, W_r$ denote irreducible components of $Z$, then $\overline{Z^{X(\Delta)}} = \coprod_{1\leq i\leq r}\overline{W_i^{X(\Delta)}}$ is a connected and irreducible decomposition of $\overline{Z^{X(\Delta)}}$. 
      In particular, $T_N\times \overline{W_i^{X(\Delta)}}\rightarrow X(\Delta)$ is smooth for any $1\leq i\leq r$.
  \end{itemize}
\end{proposition}
\begin{proof}
    We prove the statements from (a) to (d) in order. 
    \begin{itemize}
        \item[(a)] Because $m$ is flat and of finite type, $m$ is an open morphism. 
        Moreover, $m$ is a restriction of the action morphism of $X(\Delta)$, and hence, the image of $m$ is an open toric subvariety of $X(\Delta)$. 
        Let $\Delta'$ denote the subfan of $\Delta$ associated with this open subvariety of $X(\Delta)$. 
        By the definition of $m$, $\Delta'$ coincides with $\{\sigma\in\Delta\mid \overline{Z^{X(\Delta)}}\cap O_\sigma \neq \emptyset\}$, and thus, the statement holds. 
        \item[(b)] By the assumption, the dimension of all irreducible components of $T_N\times \overline{Z^{X(\Delta)}}$ is $n + d$. 
        Thus, the dimension of all irreducible components of $T_N\times (\overline{Z^{X(\Delta)}}\cap O_\sigma) = (T_N\times\overline{Z^{X(\Delta)}})\times_{X(\Delta)} O_\sigma$ is $(n + d - \dim(X(\Delta))) + \dim(O_\sigma)$. 
        Therefore, by the argument of toric varieties, the dimension of all irreducible components of $\overline{Z^{X(\Delta)}}\cap O_\sigma$ is $d - \dim(\sigma)$. 
        \item[(c)] One direction is held from the fact that the smoothness is preserved under the base change. 
        Thus, we assume that $\overline{Z^{X(\Delta)}}\cap O_\sigma$ is smooth for any $\sigma\in\Delta$. 
        For each $\sigma\in\Delta$, there exists the following Cartesian diagram by Lemma \ref{lem: action}(a): 
        \begin{equation*}
            \begin{tikzcd}
            T_N\times (\overline{Z^{X(\Delta)}}\cap O_\sigma)\ar[r]\ar[d]&T_N\times O_\sigma\ar[r]\ar[d]& O_\sigma \ar[d]\\
            T_N\times \overline{Z^{X(\Delta)}} \ar[r] & T_N\times X(\Delta)\ar[r] & X(\Delta),
            \end{tikzcd}
        \end{equation*}
        where the left horizontal morphisms are closed immersions, the right horizontal ones are action morphisms, and the composition of the lower morphisms is $m$. 
        The composition of the upper morphisms in the diagram above can be identified with the composition of the projection $T_N\times (\overline{Z^{X(\Delta)}}\cap O_\sigma)\rightarrow T_{N/\langle\sigma\rangle\cap N}\times (\overline{Z^{X(\Delta)}}\cap O_\sigma)$ and the multiplication morphism $T_{N/\langle\sigma\rangle\cap N}\times (\overline{Z^{X(\Delta)}}\cap O_\sigma)\rightarrow O_\sigma$.  
        Hence, the composition of the upper morphisms is smooth by Lemma \ref{lem: action smooth}. 
        Thus, the fibers of $m$ at all points in $X(\Delta)$ are smooth. 
        Note that $m$ is flat, and hence, $m$ is smooth. 
        \item[(d)] Because all fibers of $m$ are regular and $X(\Delta)$ is normal and Cohen-Macaulay, 
        $\overline{Z^{X(\Delta)}}$ is also normal and Cohen-Macaulay by \cite[Corollaries of Theorem 23.3 and 23.9]{Mat86}. 
        By the definition of $\{W_i\}_{1\leq i\leq r}$, $\{\overline{{W_i}^{X(\Delta)}}\}_{1\leq i\leq r}$ are irreducible components of $\overline{Z^{X(\Delta)}}$. 
        Because $\overline{Z^{X(\Delta)}}$ is normal, ${\overline{{W_i}^{X(\Delta)}}}_{1\leq i\leq r}$ are disjoint, and hence, $\overline{Z^{X(\Delta)}} = \coprod_{1\leq i\leq r}\overline{W_i^{X(\Delta)}}$ is a connected and irreducible decomposition of $\overline{Z^{X(\Delta)}}$. 
    \end{itemize}
\end{proof}
It is well known that a closure of any torus orbit of toric varieties has a toric structure. 
Now we show that this property also holds for tropical compactification. 
\begin{proposition}\label{prop: stratification}
  Let $N$ be a lattice of finite rank, let $\Delta$ be a strongly convex rational polyhedral fan in $N_\RR$, and let $Z$ be a closed subscheme of $T_N = \mathbb{G}^n_{m, k}$.
  We assume that the multiplication morphism $m\colon T_N\times \overline{Z^{X(\Delta)}}\rightarrow X(\Delta)$ is flat.
  Then the following statements hold:
  \begin{itemize}
    \item[(a)] The multiplication morphism $m_\sigma\colon T_{N/\langle\sigma\rangle\cap N}\times (\overline{Z^{X(\Delta)}}\cap \overline{O_\sigma})\rightarrow \overline{O_\sigma}$ is flat for any $\sigma\in\Delta$.
    \item[(b)] The scheme theoretic closure of $\overline{Z^{X(\Delta)}}\cap O_\sigma$ in $\overline{O_\sigma}$ is $\overline{Z^{X(\Delta)}}\cap \overline{O_\sigma}$ for any $\sigma\in\Delta$.
    \item[(c)] We keep the notation of (a). 
    If $m$ is smooth, then $m_\sigma$ is smooth for any $\sigma\in\Delta$. 
    \item[(d)] Let $\sigma\in\Delta$ be a cone such that $\overline{Z^{X(\Delta)}}\cap O_\sigma\neq \emptyset$.
      We assume that $m$ is smooth. 
      Let $E_1, \ldots, E_r$ denote irreducible components of $\overline{Z^{X(\Delta)}}\cap O_\sigma$.
      Then $\overline{Z^{X(\Delta)}}\cap \overline{O_\sigma} = \coprod_{1\leq i\leq r} \overline{E_i}$ is a connected and irreducible decomposition of $\overline{Z^{X(\Delta)}}\cap \overline{O_\sigma}$.
  \end{itemize}
\end{proposition}
\begin{proof}
    We prove the statement from (a) to (d) in order. 
    \begin{itemize}
        \item[(a)] By Lemma \ref{lem: action}(a), 
        there exists the following Cartesian square: 
        \begin{equation*}
            \begin{tikzcd}
            T_N\times (\overline{Z^{X(\Delta)}}\cap \overline{O_\sigma})\ar[r, "q_*\times\id"]\ar[d]&T_{N/\langle\sigma\rangle\cap N}\times (\overline{Z^{X(\Delta)}}\cap \overline{O_\sigma})\ar[r, "m_\sigma"]& \overline{O_\sigma} \ar[d]\\
            T_N\times \overline{Z^{X(\Delta)}} \ar[rr, "m"] & & X(\Delta),
            \end{tikzcd}
        \end{equation*}
        where $q_*$ is a natural quotient morphism $T_N\rightarrow T_{N/\langle\sigma\rangle\cap N}$ induced by the quotient map $q\colon N\rightarrow N/\langle\sigma\rangle\cap N$. 
        Because $m$ is flat, $m_\sigma\circ(q_*\times\id)$ is also flat. 
        Moreover, $q_*\times\id$ is faithfully flat, and hence, $m_\sigma$ is flat. 
        \item[(b)] By (a) and Lemma \ref{lem: flat-closure lemma}, the scheme theoretic closure of $T_{N/\langle\sigma\rangle\cap N}\times (\overline{Z^{X(\Delta)}}\cap O_\sigma)$ in $T_{N/\langle\sigma\rangle\cap N}\times \overline{O_\sigma}$ is $T_{N/\langle\sigma\rangle\cap N}\times (\overline{Z^{X(\Delta)}}\cap \overline{O_\sigma})$. 
        Thus, the scheme theoretic closure of $ \overline{Z^{X(\Delta)}}\cap O_\sigma$ in $\overline{O_\sigma}$ is $\overline{Z^{X(\Delta)}}\cap \overline{O_\sigma}$ by Lemma \ref{lem: base change and closure}. 
        \item[(c)] We keep the notation in the proof of (a). 
        Then the composition $m_\sigma\circ(q_*\times\id)$ is smooth. 
        Because $q_*\times\id$ is smooth and subjective, $m_\sigma$ is also smooth. 
        \item[(d)] By (b) and (c), we can apply Proposition \ref{prop: property1}(d) for the torus $O_\sigma = T_{N/{\langle\sigma\rangle\cap N}}$, the toric variety $\overline{O_\sigma}$, the closed subschemes $\overline{Z^{X(\Delta)}}\cap O_\sigma\subset O_\sigma$ and $\overline{Z^{X(\Delta)}}\cap \overline{O_\sigma}\subset \overline{O_\sigma}$, and the multiplication morphism $m_\sigma\colon T_{N/\langle\sigma\rangle\cap N}\times (\overline{Z^{X(\Delta)}}\cap \overline{O_\sigma})\rightarrow \overline{O_\sigma}$. 
    \end{itemize}
\end{proof}
From a morphism of lattices and compatible fans, we can construct a morphism of toric varieties. 
A key question arises as to whether tropical compactification is preserved under the pullback of toric morphisms. 
The following proposition showed in \cite[Proposition 2.5]{Tev} indicates that such preservation holds in the case of dominant toric morphisms.
\begin{proposition}\label{prop: surjection}\cite[Proposition 2.5]{Tev}
  Let $\pi\colon N'\rightarrow N$ be a surjective morphism of lattices of finite rank, let $Z$ be a closed subscheme of $T_N$, and let $\Delta$ be a strongly convex rational polyhedral fan in $N_\RR$. 
  We may assume that the multiplication morphism $m\colon T_N\times \overline{Z^{X(\Delta)}}\rightarrow X(\Delta)$ is flat. 
  Let $\Delta'$ be a strongly convex rational polyhedral fan in $N'_\RR$ such that $\pi$ is compatible with $\Delta'$ and $\Delta$, let $Z'$ denote $Z\times_{T_N} T_{N'}$, and let $m'\colon T_{N'}\times \overline{Z'^{X(\Delta')}}\rightarrow X(\Delta')$ be the multiplication morphism. 
  Then the following statements hold: 
  \begin{itemize}
    \item[(a)] It foolws that $\overline{Z'^{X(\Delta')}} = \overline{Z^{X(\Delta)}}\times_{X(\Delta)}X(\Delta')$ in $X(\Delta')$.
    Moreover, $m'$ is flat. 
    \item[(b)] Let $\sigma\in\Delta$ and let $\tau\in\Delta'$ be cones such that $\pi_\RR(\tau^\circ)\subset \sigma^\circ$. 
    Then ($\overline{Z^{X(\Delta)}}\cap O_\sigma)\times\mathbb{G}^{s}_{m, k}\cong \overline{Z'^{X(\Delta')}}\cap O_\tau$, where $s = \rank(\ker(\pi))+\dim(\sigma)-\dim(\tau)$.   
    \item[(c)] If $m$ is smooth, then $m'$ is also smooth. 
    \item[(d)] If $Z$ has a sch\"{o}n compactification, $Z'$ also has a sch\"{o}n compactification.
  \end{itemize} 
\end{proposition}
\begin{proof}
    We prove the statements from (a) to (d) in order. 
    \begin{itemize}
        \item[(a)] By Lemma \ref{lem: action}(b), there exists the following Cartesian square: 
        \begin{equation*}
            \begin{tikzcd}
            T_{N'}\times (\overline{Z^{X(\Delta)}}\times_{X(\Delta)}X(\Delta')) \ar[rr, "m''"]\ar[d, "\id\times\pi_*"]& & X(\Delta') \ar[d, "\pi_*"]\\
            T_{N'}\times \overline{Z^{X(\Delta)}} \ar[r, "\pi_*\times\id"] & T_N\times \overline{Z^{X(\Delta)}}\ar[r, "m"]  & X(\Delta),
            \end{tikzcd}
        \end{equation*}
        where $m''$ is the multiplication morphism of $\overline{Z^{X(\Delta)}}\times_{X(\Delta)}X(\Delta')$. 
        By the assumption, $m\circ (\pi_*\times\id)$ is flat, and hence, $m''$ is also flat. 
        By the flatness of $m''$ and Lemma \ref{lem: flat-closure lemma}, the scheme theoretic closure of $T_{N'}\times Z'$ in $T_{N'}\times X(\Delta')$ is $T_{N'}\times (\overline{Z^{X(\Delta)}}\times_{X(\Delta)}X(\Delta'))$. 
        Thus, the scheme theoretic closure of $Z'$ in $X(\Delta')$ is $\overline{Z^{X(\Delta)}}\times_{X(\Delta)}X(\Delta')$ by Lemma \ref{lem: base change and closure}. 
        In particular, $m'$ is flat. 
        \item[(b)] By (a), there exists the following Cartesian square: 
        \begin{equation*}
            \begin{tikzcd}
                  \overline{Z'^{X(\Delta')}}\cap O_\tau\ar[r, hook]\ar[d, "\pi_*"] & O_\tau\ar[d, "\pi_*"]\\
                  \overline{Z^{X(\Delta)}}\cap O_\sigma\ar[r, hook]& O_\sigma.
            \end{tikzcd}
        \end{equation*}
        Because $\pi$ is surjective, $\pi_*\colon O_\tau\rightarrow O_\sigma$ is a trivial torus fibration of relative dimension $\rank(\ker(\pi))+\dim(\sigma)-\dim(\tau)$. 
        Thus, the statement holds. 
        \item[(c)] By the proof in (a), $m'$ is also smooth. 
        \item[(d)] We assume $\overline{Z^{X(\Delta)}}$ is a sch\"{o}n compactification of $Z$ and $\supp(\Delta') = \pi^{-1}_\RR(\supp(\Delta))$. 
        Then $m'$ is smooth and faithfully flat by the proof in (a). 
        Because $\pi_*\colon X(\Delta')\rightarrow X(\Delta)$ is proper, $\overline{Z'^{X(\Delta')}} = \overline{Z^{X(\Delta)}} \times_{X(\Delta)}X(\Delta')$ is also proper over $k$. 
        Thus, $\overline{Z'^{X(\Delta')}}$ is a sch\"{o}n compactification of $Z'$. 
    \end{itemize}
\end{proof}
Proposition \ref{prop: surjection} claims that a toric resolution of an embedded toric variety gives another tropical compactification of a closed subvariety of an algebraic torus. 
In particular, its tropical compactification is not unique. 
Thus, we introduce the notion of a \textbf{good fan} for it in the following definition. 
This fan is a kind of ``minimal" fan that obtains its tropical compactification. 
We remark that this is not in the strict sense and there is no need to require the fan to satisfy strong convexity. 
\begin{definition}\label{def: good fan}
  Let $N$ be a lattice of finite rank, let $Z\subset T_N$ be a closed subscheme of $T_N$, and let $\Delta$ be a rational polyhedral convex fan in $N_\RR$. 
  We call that $\Delta$ is a \textbf{good fan} for $Z$ if the multiplication morphism $T_N\times \overline{Z^{X(\Delta')}}\rightarrow X(\Delta')$ is flat for any strongly convex rational polyhedral fan in $N_\RR$ such that $\Delta'$ is a refinement of $\Delta$. 
\end{definition}
\begin{proposition}\label{prop: good fan example}
  We keep the notation in Definition \ref{def: good fan}. 
  Then the following statements hold:
  \begin{itemize}
    \item[(a)] If $\Delta$ is strongly convex and the multiplication morphism $T_N\times \overline{Z^{X(\Delta)}}\rightarrow X(\Delta)$ is flat, then $\Delta$ is a good fan for $Z$. 
    \item[(b)] We keep the assumption in (a). 
    Let $\pi\colon N'\rightarrow N$ be a surjective morphism of lattices of finite rank, let $Z'$ denote $Z\times_{T_N} T_{N'}$, and let $\Delta' = \{(\pi_\RR)^{-1}(\sigma)\mid \sigma\in\Delta\}$ be a fan in $N'_\RR$. 
    Then $\Delta'$ is a good fan for $Z'$.
  \end{itemize}
\end{proposition}
\begin{proof}
    By Proposition \ref{prop: surjection} (a), the statements hold. 
\end{proof}
\subsection{Example of sch\"{o}n affine varieties}
In this subsection, we give some examples of sch\"{o}n compactifications. 
These examples are given by ``base point free'' hyperplane arrangements.  
In this section, We use the following notation: 
\begin{itemize}
  \item Let $n$ and $d$ be positive integers, and let $k$ be a field. 
  \item Let $\mathcal{B} = \{f_0, f_1, \ldots, f_d\}$ denote a finite subset of $\Gamma(\PP^n_k, \OO_{\PP^n_k}(1))\setminus\{0\}$. 
  We assume that $\mathcal{B}$ generates $\Gamma(\PP^n_k, \OO_{\PP^n_k}(1))$ as a $k-$vector space. 
  We regard $f_i$ as a homogeneous polynomial of $\deg(f_i) = 1$. 
  \item Let $\iota_0$ be a rational map $\PP^n_k\dashrightarrow \PP^d_k$ defined by $\PP^n_k\ni a\mapsto [f_0(a):f_1(a):\cdots:f_d(a)]\in\PP^d_k$. 
  Then $\iota_0$ is a closed embedding by the assumption of $\mathcal{B}$. 
  \item Let $\mathcal{V}$ denote the following set: 
  \[
    \mathcal{V} = \{V\subset \Gamma(\PP^n_k, \OO_{\PP^n_k}(1))\mid V = \sum_{i; f_i\in V} k\cdot f_i\}.
  \]
  \item For $V\in\mathcal{V}$, let $\rho(V)$ denote a set $\{0\leq i\leq d\mid f_i\in V\}$. 
  \item Let $\{e^0, e^1, \ldots, e^d\}$ denote a canonical basis of $\ZZ^{d+1}$, and $\mathbf{1}\in \ZZ^{d+1}$ denote $\sum_{0\leq i\leq d}e^i$. 
  \item Let $N$ denote $\ZZ^{d+1}/\ZZ\mathbf{1}$, let $p$ denote the quotient morphism $\ZZ^{d+1}\rightarrow N$, and let $e_i\in N$ denote $p(e^i)$ for each $0\leq i\leq d$.  
  \item For $V\in \mathcal{V}$, let $e_{V}$ denote $\sum_{i\in \rho(V)}e_i\in N$. 
  \item Let $\mathcal{C}$ denote the following set: 
  \[
  \mathcal{C} = \{(V_1, V_2, \ldots, V_s= \Gamma(\PP^n_k, \OO_{\PP^n_k}(1)))\mid s\in\ZZ_{>0}, 0\neq V_i\in\mathcal{V}, V_i\subsetneq V_{i+1},1\leq\forall i\leq s\}.
  \]    
  \item For $c = (V_1, V_2, \ldots, V_s)\in\mathcal{C}$, $\sigma_c$ denote a strongly convex rational polyhedral cone in $N_\RR$ generated by $\{e_{V_i}\}_{1\leq i\leq s}$. 
  Let $\Delta(\mathcal{B})$ denote $\{\sigma_c\mid c\in\mathcal{C}\}$. 
  \item Let $Z = \PP^n_k\times_{\PP^d_k}\mathbb{G}^d_{m, k}$ be a closed subscheme of $\mathbb{G}^d_{m, k} = T_N$. 
\end{itemize}
The following fact shows the relation of the complement $Z$ of the union of hyperplanes and the set $\Delta(\mathcal{B})$. 
The details of the proof can be found in \cite{MS15}. 

\begin{proposition}\label{prop: matroid}\cite[Theorem. 4.2.6, Theorem. 4.1.11]{MS15}
  Let $E$ denote a set $\{0, 1, \ldots, d\}$ and $2^E$ denote the power set of $E$. 
  We define a map $\delta\colon 2^E\rightarrow \ZZ$ as follows: 
  \[
    2^E\ni A\mapsto \dim_k(\sum_{i\in A}k\cdot f_i)\in\ZZ.
  \]
  Then the following statements hold: 
  \begin{itemize}
    \item[(a)] $(E, \delta)$ is a matroid in the context of  \cite[Definition 4.2.3]{MS15}. 
    \item[(b)] Let $A$ be a subset of $E$. 
    If $\delta(A)<\delta(A\cup\{i\})$ holds for any $i\notin A$, we say that $A$ is flat. 
    Let $\mathcal{F}$ denote the set which consists of all flat subsets of $E$. 
    Then there exists a one-to-one correspondence of elements in $\mathcal{F}$ and linear subspaces in $\mathcal{V}$ as follows:
    \[
    	\mathcal{F}\ni A\mapsto \sum_{i\in A}k\cdot f_i \in\mathcal{V}.
    \] 
    \item[(c)] The map $\mathcal{C}\rightarrow \Delta(\mathcal{B})$ defined as $c\mapsto \sigma_c$ for each $c\in\mathcal{C}$ is bijective. 
    \item[(d)] The set $\Delta(\mathcal{B})$ is a strongly convex rational simplicial fan in $N_\RR$. 
    \item[(e)] The equation $\supp(\Delta(\mathcal{B})) = \mathrm{Trop}(Z)$ holds. 
  \end{itemize}
\end{proposition}
\begin{proof}
  We prove these statements from (a) to (e) in order. 
  \begin{itemize}
    \item[(a)] Let $A, B\in 2^E$ be subsets of $E$, and let $U$ denote $\sum_{i\in A}k\cdot f_i$ and let $W$ denote $\sum_{i\in B}k\cdot f_i$ which are linear subspaces of $\Gamma(\PP^n_k, \OO_{\PP^n_k}(1))$. 
    Because $\dim(U)\leq |A|$, it follows that $\delta(A)\leq |A|$. 
    Moreover, if $A\subset B$, then $U\subset W$, and hence, $\delta(A)\leq \delta(B)$.     				
    Finally, because $\dim_k(U) + \dim_k(W) = \dim_k(U + W) + \dim_k(U\cap W)$ and $\sum_{i\in A\cap B}k\cdot f_i\subset U\cap W$, if holds that $\delta(A\cap B)+ \delta(A\cup B)\leq \delta(A) + \delta(B)$. 
    \item[(b)] First, we will show $A = \rho(\sum_{i\in A}k\cdot f_i)$ for any $A\in \mathcal{F}$. 
    Indeed, it is clear that $A\subset \rho(\sum_{i\in A}k\cdot f_i)$, so that we will show $\rho(\sum_{i\in A}k\cdot f_i)\subset A$. 
    Let $j\in \rho(\sum_{i\in A}k\cdot f_i)$. 
    If $j\notin A$, then $\delta(A)<\delta(A\cup\{j\})$ because $A\in\mathcal{F}$. 
    However, $\delta(A\cup\{j\})\leq \dim_k(\sum_{i\in A}k\cdot f_i)$ because $f_j\in \sum_{i\in A}k\cdot f_i$, hence it is contradiction. 
    Thus, $A = \rho(\sum_{i\in A}k\cdot f_i)$. 
    				
    Second, we will show that $\rho(V)\in\mathcal{F}$ for any $V\in\mathcal{V}$. 
    Let $j\in E\setminus \rho(V)$. 
    Then $f_j\notin V$. 
    Because $V\in\mathcal{V}$, it follows that $V = \sum_{i\in \rho(V)}k\cdot f_i$. 
    Thus, $V \subsetneq \sum_{i\in \rho(V)\cup\{j\}}k\cdot f_i$, in particular, $\delta(\rho(V))<\delta(\rho(V)\cup \{j\})$. 
    				
    Therefore, the map in the statement (b) has the inverse map $\rho$, and this is isomorphic.  
    \item[(c)] 
    Let $c\in\mathcal{C}$ and $x = (x_0, \ldots, x_d) \in (\ZZ^{d+1})_\RR$. 
    Then the following statements are equivalent by the definition of $\sigma_c$: 
        \begin{itemize}
  	\item[(i)] It follows that 
    $x \in p^{-1}_\RR(\sigma_c)$. 
        \item[(ii)] There exists $a_1, \ldots, a_s\in\RR$ such that $a_1\geq a_2\geq\cdots\geq a_s$ and $x_l = a_j$ for any $l\in E$, where $j$ is a unique integer such that $l\in\rho(V_j)\setminus\rho(V_{j-1})$.
        \end{itemize}
    If $\sigma_{c_1} = \sigma_{c_2}$ for $c_1, c_2\in\mathcal{C}$, then $p^{-1}_\RR(\sigma_{c_1}) = p^{-1}_\RR(\sigma_{c_2})$. 
    Thus, $c_1 = c_2$ by the equivalent statements above. 
    \item[(d)] By (a), (b), and \cite[Theorem. 4.2.6]{MS15}, the statement holds. 
    \item[(e)] This is a result of \cite[Theorem. 4.1.11]{MS15}. 
  \end{itemize}
\end{proof}
  The following statement shows that $Z$ has a sch\"{o}n compactification. 
  This result may be a well-known fact, but the author could not find the literature so we show the proof. 
\begin{proposition}\label{prop: arrangement}
  With the notation above, $\overline{Z^{X(\Delta(\mathcal{B}))}}$ is a sch\"{o}n compactification of $Z$.  
\end{proposition}
\begin{proof}
    By \cite[Theorem. 1.5]{LQ11} and Proposition \ref{prop: matroid}(d) and (e), it is enough to check that $Z$ has a sch\"{o}n compactification. 
    By the proof of \cite[Theorem. 1.7]{Tev}, the tropical compactification of $Z$ is obtained by the universal projective space bundle on the Grassmanniann varieties $\Gr_k(n+1, d+1)$ in this case. 
    This bundle is smooth over $\Gr_k(n+1, d+1)$, and hence, the multiplication morphism is also smooth. 
\end{proof}

\section{Stable birational volume of a sch\"{o}n variety}
In this section, we construct the strictly toroidal model of a smooth variety by the sch\"{o}n compactification and compute its stable birational volume. 

Through this section, we use the following notation:
\begin{itemize}
  \item Let $k$ be an algebraically closed field of $\charac(k) = 0$. 
  \item Let $N$ be a lattice of finite rank. 
  \item Let $\pr_2\colon N\oplus\ZZ\rightarrow \ZZ$ be the second projection of $N\oplus\ZZ$. 
  \item Let $Y$ be an equidimensional closed subscheme of $T_N\times \mathbb{G}^1_{m, k} = T_{N\oplus\ZZ}$.  
  \item Let $\Delta$ be a strongly convex rational polyhedral fan in $(N\oplus\ZZ)_\RR$. 
  We assume that $\supp(\Delta)\subset N_\RR\times \RR_{\geq 0}$. 
  By this assumption, there exists a toric morphism $(\pr_2)_*\colon X(\Delta)$
  $\rightarrow \A^1_k$ which is induced by $\pr_2$. 
  \item Let $\Delta_0$ denote the subfan of $\Delta$ which consists of all cones $\sigma\in\Delta$ such that $\sigma\subset N_\RR\times\{0\}$. 
  \item Let $\Delta_Y$ denote the subset of $\Delta$ which consists of all cones $\sigma\in\Delta$ such that $\overline{Y^{X(\Delta)}}\cap O_\sigma\neq \emptyset$.
  \item Let $\Rt$ be a valuation ring defined as follows:
  \[
      \Rt = \bigcup_{n\in\ZZ_{>0}} k[[t^{\frac{1}{n}}]].
  \]
  \item Let $\Kt$ be the fraction field of $\Rt$. 
  We remark that $\Kt$ is written explicitly as follows:
  \[
      \Kt = \bigcup_{n\in\ZZ_{>0}} k((t^{\frac{1}{n}})).
  \]
  \item Let $\Spec(\Rt)\rightarrow \Spec(k[t])$ be a morphism of affine schemes induced by a $k$-morphism $k[t]\hookrightarrow \Rt$ whose image of $t$ is $t$.  
  \item Let $\mathcal{Y}$ denote the scheme $\overline{Y^{X(\Delta)}}\times_{\A^1_k}\Spec(\Rt)$ and let $\mathcal{Y}^\circ$ denote the scheme $Y\times_{\mathbb{G}^1_{m, k}}\Spec(\Kt)$. 
  We remark that $\mathcal{Y}^\circ$ is an open subscheme of $\mathcal{Y}_\Kt \cong \overline{Y^{X(\Delta_0)}}\times_{\mathbb{G}^1_{m, k}}\Spec(\Kt)$. 
  Moreover, we remark that We can identify with $\mathcal{Y}_k$ and $\overline{Y^{X(\Delta)}}\cap (\pr_2)^{-1}_*(0)$. 
\end{itemize}
\subsection{The definition of some properties of fans}
In this subsection, we introduce the notion of some properties of $\Delta$ before constructing the strictly toroidal model of $\mathcal{Y}^\circ$ in Proposition \ref{prop: model I}. 
\begin{definition}\label{def: type of polytope}
  We define some properties of $\Delta$ as follows:   
  \begin{itemize}
      \item Let $\Delta_{\spe}$ denote a subset of $\Delta$ defined as follows:
      \[
          \Delta_{\spe} = \{\sigma\in\Delta\mid\sigma\cap(N_\RR\times\{1\})\neq\emptyset\}.
      \]
      We remark that the following equation holds:
      \[
          \Delta_{\spe} = \{\sigma\in\Delta\mid(\pr_2)_\RR(\sigma)\neq \{0\}\}
          = \Delta\setminus\Delta_0.
      \]
      \item Let $\Delta_{\bdd}$ denote a subset of $\Delta_{\spe}$ defined as follows:
      \[
          \Delta_{\bdd} = \{\sigma\in\Delta_{\spe}\mid\sigma\cap(N_\RR\times\{1\}) \mathrm{\ \ is\ bounded.}\}.
      \]
      \item We call that $\Delta$ is \textbf{compactly\ arranged} if for every $\sigma_1, \sigma_2\in\Delta_{\bdd}$, and $\tau\in\Delta$ such that $\sigma_1\cup \sigma_2\subset \tau$, there exists $\sigma_3\in\Delta_{\bdd}$ such that $\sigma_1\cup \sigma_2\subset \sigma_3$. 
      \item We call that $\Delta$ is \textbf{generically\ unimodular} if every $\sigma\in\Delta_0$ is unimodular. 
      \item We call that $\Delta$ is \textbf{specifically\ reduced} if for every ray $\gamma\in\Delta_{\bdd}$, we have $\gamma\cap (N\times\{1\}) \neq \emptyset$. 
    \end{itemize}
\end{definition}
The following proposition gives sufficient conditions for some properties in Definition \ref{def: type of polytope}.  
\begin{proposition}\label{prop: example of type of polytope}
  We keep the notation in Definition \ref{def: type of polytope}.  
  Then the following statements follow:
  \begin{enumerate}
      \item[(a)] If $\Delta$ is a simplicial fan, then $\Delta$ is compactly arranged. 
      \item[(b)] If $\Delta$ is a unimodular fan, then $\Delta$ is compactly arranged and generically unimodular. 
  \end{enumerate}
\end{proposition}
\begin{proof}
  We show these statements from (a) to (b). 
  \begin{enumerate}
      \item[(a)] Let $\sigma_1, \sigma_2\in\Delta_{\bdd}$, and $\tau\in\Delta$ be cones such that $\sigma_1\cup \sigma_2\subset \tau$. 
      Then $\sigma_1$ and $\sigma_2$ are faces of $\tau$. 
      Because $\tau$ is a simplicial cone, there exists a face $\sigma_3$ of $\tau$ such that $\sigma_3$ is generated by $\sigma_1$ and $\sigma_2$. 
      By the definition of $\sigma_3$, we have $\sigma_3\in\Delta_{\bdd}$ and $\sigma_1\cup \sigma_2\subset \sigma_3$. 
      \item[(b)] The subfan $\Delta_{0}$ of $\Delta$ is unimodular too. 
      Thus, $\Delta$ is generically unimodular. 
      Moreover, $\Delta$ is compactly arranged by (a). 
  \end{enumerate}
\end{proof}
\subsection{Constructing the strictly toroidal model}
In this subsection, we construct the strictly toroidal model of $\mathcal{Y}^\circ$. 
\begin{proposition}\label{prop: model I}
  With the notation above, we assume the following conditions: 
  \begin{itemize}
    \item[(1.)] The morphism $(\pr_2)_*|_{\overline{Y^{X(\Delta)}}}\colon \overline{Y^{X(\Delta)}}\rightarrow \A^1_k$ is proper. 
    \item[(2.)] The multiplication morphism $m\colon  T_{N\oplus\ZZ}\times\overline{Y^{X(\Delta)}}\rightarrow X(\Delta)$ is smooth. 
    Then $\Delta_Y$ is a subfan of $\Delta$ by Proposition \ref{prop: property1}(a). 
    \item[(3.)] The fan $\Delta$ is generically unimodular and specifically reduced.   
  \end{itemize}
  Then the following statements hold:
  \begin{itemize}
    \item[(a)] The scheme $\mathcal{Y}$ is flat, proper, and of finite presentation over $\Rt$. 
    \item[(b)] The scheme $\mathcal{Y}$ is strictly toroidal. 
    \item[(c)] The generic fiber $\mathcal{Y}$ is smooth over $\Kt$. 
    \item[(d)] The scheme theoretic closure of $\mathcal{Y}^\circ$ in $\mathcal{Y}_\Kt$ is $\mathcal{Y}_\Kt$.
    Moreover, the scheme theoretic closure of $\mathcal{Y}_\Kt$ in $\mathcal{Y}$ is $\mathcal{Y}$.  
    \item[(e)] Let $\{W^\circ_j\}_{1\leq j\leq s}$ be irreducible components of $\mathcal{Y}^\circ$ and let $W_j$ denote the scheme theoretic closure of $W^\circ_j$ in $\mathcal{Y}_\Kt$ for each $1\leq j\leq s$.  
    Then $W^\circ_j$ is a dense open subscheme of $W_j$ for any $1\leq j\leq s$, and $\{W_j\}_{1\leq j\leq s}$ are disjoint and all irreducible components of $\mathcal{Y}_\Kt$. 
    \item[(f)] We keep the notation in (e). Let $\mathcal{W}_j$ be the scheme theoretic closure of $W_j$ in $\mathcal{Y}$ for $1\leq j\leq s$. 
    Then each $\mathcal{W}_j$ is a flat, proper, and strictly toroidal scheme of finite presentation over $\Rt$ with smooth generic fiber $(\mathcal{W}_j)_\Kt = W_j$. 
    Moreover, $\{\mathcal{W}_j\}_{1\leq j\leq s}$ are disjoint and irreducible components of $\mathcal{Y}$. 
  \end{itemize}
\end{proposition}
\begin{proof}
    We prove the statements from (a) to (f) in order. 
    \begin{itemize}
        \item[(a)] By the assumption, $(\pr_2)_*|_{\overline{Y^{X(\Delta)}}}\colon \overline{Y^{X(\Delta)}}\rightarrow \A^1_k$ is proper and of finite presentation, and hence, $\mathcal{Y}$ is also proper and of finite presentation over $\Spec(\Rt)$. 
        For checking the flatness of $\mathcal{Y}$, it is enough to show that $(\pr_2)_*|_{\overline{Y^{X(\Delta)}}}\colon \overline{Y^{X(\Delta)}}\rightarrow \A^1_k$ is flat at any point in $(\pr_2)^{-1}_*(0)$. 
        Let  $\sigma\in\Delta_{\spe}$ and let $I_Y$ be the ideal of $k[T_{N\oplus\ZZ}]$ associated with $Y\subset T_{N\oplus\ZZ}$. 
        Then $I_Y\cap k[X(\sigma)]$ is the ideal of $k[X(\sigma)]$ associated with the closed subscheme $\overline{Y^{X(\Delta)}}\cap X(\sigma)$ of $X(\sigma)$, and it is enough to show that $t\in k[X(\sigma)]/(I_Y\cap k[X(\sigma)])$ is not a zero divisor because $k[t]_{(t)}$ is a DVR. 
        Let $f\in k[X(\sigma)]$ be an element such that $t\cdot f\in I_Y\cap k[X(\sigma)]$. 
        Because $t$ is a unit in $k[T_{N\oplus\ZZ}]$, $f = t^{-1}(tf)\in I_Y$. 
        In particular, it follows that $f\in I_Y\cap k[X(\sigma)]$, so that $t$ is not a zero divisor in $k[X(\sigma)]/(I_Y\cap k[X(\sigma)])$. 
        \item[(b)] Let $y\in \mathcal{Y}_k$. 
        Then there uniquely exists $\sigma\in \Delta_{\spe}$ such that $y\in O_\sigma$. 
        Because $\Delta$ is specifically reduced, there exists a sublattice $N_0$ of $N$ such that the conditions in Lemma \ref{lem: relative toroidal}(a) hold. 
        Let $p\colon N\oplus \ZZ\rightarrow (N/N_0)\oplus\ZZ$ be the natural quotient map and let $\sigma_0$ denote the strongly convex rational polyhedral cone $p_\RR(\sigma)$ in $((N/N_0)\oplus\ZZ)_\RR$. 
        Then $(\overline{Y^{X(\Delta)}}\cap X(\sigma))\times_{\A^1_k} \Spec(\Rt)$ is an open subscheme of $\mathcal{Y}$, and there exists the following Cartesian diagram: 
        \begin{equation*}
            \begin{tikzcd}
            (\overline{Y^{X(\Delta)}}\cap X(\sigma))\times_{\A^1_k} \Spec(\Rt)\ar[r]\ar[d]&X(\sigma_0)\times_{\A^1_k}\Spec(\Rt)\ar[r]\ar[d]& \Spec(\Rt)\ar[d]\\
            \overline{Y^{X(\Delta)}}\cap X(\sigma) \ar[r, "p_*|_{\overline{Y^{X(\Delta)}}\cap X(\sigma)}"] & X(\sigma_0)\ar[r] & \A^1_k,
            \end{tikzcd}
        \end{equation*}
        where the composition of the lower morphisms are $(\pr_2)_*|_{\overline{Y^{X(\Delta)}}\cap X(\sigma)}$. 
        
        Thus $p_*|_{\overline{Y^{X(\Delta)}}\cap X(\sigma)}\colon \overline{Y^{X(\Delta)}}\cap X(\sigma)\rightarrow X(\sigma_0)$ is smooth at $y\in \overline{Y^{X(\Delta)}}\cap O_{\sigma}$ by Lemma \ref{lem: smooth to smooth}. 
        Therefore, $\mathcal{Y}$ is strictly toroidal at $y$ by Lemma \ref{lem: relative toroidal}(d). 
        \item[(c)] Let $\sigma\in\Delta_0$. 
        Because $(\pr_2)_\RR(\sigma) = \{0\}$, there exists a sublattice $N_0$ of $N$ such that $(N_0\oplus\ZZ)\oplus(\langle\sigma\rangle\cap (N\oplus\ZZ)) = N\oplus\ZZ$. 
        Let $p\colon N\oplus \ZZ\rightarrow N/N_0$ be the natural quotient map and let $\sigma_0$ be the strongly convex rational polyhedral cone $p_\RR(\sigma)$ in $(N/N_0)_\RR$. 
        Then $p_*|_{\overline{Y^{X(\Delta)}}\cap X(\sigma)}\colon \overline{Y^{X(\Delta)}}\cap X(\sigma)\rightarrow X(\sigma_0)$ is smooth at any point $x\in \overline{Y^{X(\Delta)}}\cap O_{\sigma}$ by Lemma \ref{lem: smooth to smooth}. 
        Moreover, $X(\sigma_0)$ is smooth because $\Delta$ is generically unimodular. 
        Thus, $\overline{Y^{X(\Delta)}}$ is smooth at any point $x\in \overline{Y^{X(\Delta)}}\cap O_{\sigma}$, and hence, $\overline{Y^{X(\Delta)}}\cap X(\Delta_0) = \overline{Y^{X(\Delta_0)}}$ is smooth over $k$. 
        This indicates that $\mathcal{Y}_\Kt$ is smooth over $\Kt$ by the generic smoothness and the following Cartesian product:
        \begin{equation*}
            \begin{tikzcd}
            \mathcal{Y}_\Kt\ar[r]\ar[d]&\Spec(\Kt)\ar[d]\\
            \overline{Y^{X(\Delta)}}\cap X(\Delta_0) \ar[r, "(\pr_2)_*"] & \mathbb{G}^1_{m, k}.
            \end{tikzcd}
        \end{equation*}        
        \item[(d)] There exists the following Cartesian diagram: 
        \begin{equation*}
            \begin{tikzcd}
            \mathcal{Y}^\circ\ar[r]\ar[d]&\mathcal{Y}\ar[r]\ar[d]& \Spec(\Rt)\ar[d]\\
            Y \ar[r] & \overline{Y^{X(\Delta)}}\ar[r] & \A^1_k.
            \end{tikzcd}
        \end{equation*}
        Moreover, the scheme theoretic closure of $Y$ in $\overline{Y^{X(\Delta)}}$ is $\overline{Y^{X(\Delta)}}$ because $\overline{Y^{X(\Delta)}}$ is the scheme theoretic closure of $Y$ in $X(\Delta)$. 
        Thus, the scheme theoretic closure of $\mathcal{Y}^\circ$ in $\mathcal{Y}$ is $\mathcal{Y}$ by Lemma \ref{lem: base change and closure}. 
        Therefore, scheme theoretic image of $\mathcal{Y}^\circ$ in $\mathcal{Y}_\Kt$ is $\mathcal{Y}_\Kt$ by Lemma \ref{lem: open-image}. 

        We can show the latter claim in the same way, but we show it in another way. 
        By (a), $\mathcal{Y}$ is flat over $\Rt$. 
        Thus, $\mathcal{Y}$ is the scheme theoretic closure of $\mathcal{Y}_\Kt$ in $\mathcal{Y}$ by Lemma \ref{lem:Gubler}. 
        \item[(e)] For each $1\leq j\leq s$, scheme theoretic closure of $W^\circ_j$ in $\mathcal{Y}^\circ$ is $W^\circ_j$, so that $W^\circ_j = \mathcal{Y}^\circ\cap W_j$ by Lemma \ref{lem: open-image}, in particular, $W^\circ_j$ is an open subscheme of $W_j$. 
        Moreover, the underlying space of $W_j$ is a topological closure of $W^\circ_j$ in $\mathcal{Y}_\Kt$, and hence, $W^\circ_j$ is dense in $W_j$. 

        By (d), the underlying space of $\mathcal{Y}_K$ is the union of the underlying spaces $\{W_j\}_{1\leq j\leq s}$. 
        In addition to this, $\mathcal{Y}_{\Kt}$ is smooth over $\Kt$ by (c). 
        Thus, $\{W_j\}_{1\leq j\leq s}$ are disjoint and all irreducible components of $\mathcal{Y}_\Kt$. 
        \item[(f)] By (d) and (e), we remark that the number of connected components of $\mathcal{Y}$ is less than and equal to $s$.  
        In addition to this, we remark that the generic fiber of a connected component of $\mathcal{Y}$ is non-empty and can be decomposed into a disjoint union of some $W_j$. 
        Let $E$ be a connected component of $\mathcal{Y}$. 
        By the first remark, $E$ is an open and closed subscheme of $\mathcal{Y}$. 
        In particular, $E$ is flat and locally of finite presentation over $\Spec(\Rt)$ by (a). 
        Moreover, by (b), $\mathcal{Y}_k$ is reduced, so $E_k$ is reduced too.  
        Thus, by Lemma \ref{lem: stacks exchange}(b), $E_\Kt$ is connected. 
        This shows that there exists a unique $1\leq j\leq s$ such that $E_\Kt = W_j$. 
        We recall that $E$ and $\mathcal{Y}$ are flat over $\Spec(\Rt)$ and $E$ is a closed subscheme of $\mathcal{Y}$. 
        Thus, by Lemma \ref{lem:Gubler}, $E = \mathcal{W}_j$. 
        On the other hand, let $E_j$ denote the connected component of $\mathcal{Y}$ such that $W_j\subset E_j$ for each $1\leq j\leq s$. 
        By the same argument, we can show $\mathcal{W}_j = E_j$. 
        Hence, for any $1\leq j\leq s$, $\mathcal{W}_j$ is flat over $\Spec(\Rt)$, $(\mathcal{W}_j)_\Kt = W_j$, and $\{\mathcal{W}_j\}_{1\leq j\leq s}$ are disjoint and irreducible components of $\mathcal{Y}$. 
        
        By (a), (b), and (c), $\mathcal{W}_j$ is a proper strictly toroidal model of a smooth $\Kt$-variety $W_j$ for any $1\leq j\leq s$. 
    \end{itemize}
\end{proof}
The following proposition shows that the stable birational volume of $\mathcal{Y}^\circ$ can be computed combinatorially.  
\begin{proposition}\label{prop: model II}
  We keep the notation and the assumption in Proposition \ref{prop: model I}.
  In addition to this, we assume that $\Delta$ is compactly arranged. 
  Let $\mathfrak{I}$ denote a set which consists of all irreducible components of $\mathcal{Y}_k$.
  For each $\tau\in \Delta_{\spe}\cap \Delta_Y$, $\{E^{(1)}_{\tau}, E^{(2)}_{\tau}, \ldots, E^{(r_\tau)}_{\tau}\}$ denote all connected components of $\overline{Y^{X(\Delta)}}\cap O_\tau$.  
  We remark that these are also irreducible components of $\overline{Y^{X(\Delta)}}\cap O_\tau$ by Proposition \ref{prop: property1}(c). 
  Then the following statements hold:
  \begin{itemize}
    \item[(a)]  The following equation holds:
    \[
        \mathfrak{I} = \bigcup_{\substack{\gamma\in \Delta_{\spe}\cap \Delta_Y\\ 
        \dim(\gamma) = 1}} \biggl\{\overline{E^{(i)}_{\gamma}}\biggr\}_{1\leq i\leq r_\gamma}.
    \]
    \item[(b)]  
    For any integer $1\leq j\leq s$, and any stratum $E\in \mathcal{S}(\mathcal{W}_j)$, there uniquely exists $\tau\in \Delta_{\bdd}\cap\Delta_Y$ and $1\leq i\leq r_{\tau}$ such that $E = \overline{E^{(i)}_{\tau}}$. 
    \item[(c)] Conversely, for any $\tau\in \Delta_{\bdd}\cap \Delta_Y$ and any integer $1\leq i\leq r_\tau$, there uniquely exists integer $1\leq j\leq s$ and $E\in \mathcal{S}(\mathcal{W}_j)$ such that $E = \overline{E^{(i)}_{\tau}}$. 
    \item[(d)] The following equation holds:
    \[
        \sum_{1\leq j\leq s}\VOL\biggl(\bigl\{W^\circ_j\bigr\}_{\mathrm{sb}}\biggr) = 
        \sum_{\tau\in\Delta_{\bdd}\cap\Delta_Y}(-1)^{\dim(\tau)-1}\biggl(\sum_{1\leq i\leq r_\tau} \bigl\{E^{(i)}_{\tau}\bigr\}_{\mathrm{sb}}\biggr).
    \]
    \item[(e)] Let $\Sigma$ be a strongly convex rational polyhedral fan such that $\Delta$ is a refinement of $\Sigma$ and satisfies the conditions (1.) and (2.) in Proposition \ref{prop: model I}. 
    Let $\Sigma_Y$ denote $\{\sigma\in\Sigma\mid \overline{Y^{X(\Sigma)}}\cap O_\sigma\neq \emptyset \}$ and let $\{E^{(1)}_{\sigma}, E^{(2)}_{\sigma}, \ldots, E^{(r_\sigma)}_{\sigma}\}$ denote all connected components of $\overline{Y^{X(\Sigma)}}\cap O_\sigma$ for each $\sigma\in \Sigma_{\spe}\cap \Sigma_Y$. 
    Then the following equation holds:
    \[
        \sum_{1\leq j\leq s}\VOL\biggl(\bigl\{W^\circ_j\bigr\}_{\mathrm{sb}}\biggr) = 
        \sum_{\sigma\in\Sigma_{\bdd}\cap\Sigma_Y}(-1)^{\dim(\sigma)-1}\biggl(\sum_{1\leq i\leq r_\sigma} \bigl\{E^{(i)}_{\sigma}\bigr\}_{\mathrm{sb}}\biggr).
    \]
  \end{itemize}
\end{proposition}
\begin{proof}
    We prove the statement from (a) to (e) in order. 
    In this proof, We identify with $\mathcal{Y}_k$ and $\overline{Y^{X(\Delta)}}\times_{\A^1} \{0\}$.
    \begin{itemize}
        \item[(a)]  
        We recall that $(\overline{Y^{X(\Delta)}}\times_{\A^1}\{0\})\cap O_\tau \neq \emptyset$ if and only if $\tau\in\Delta_{\spe}\cap \Delta_Y$. 
        Therefore, $\mathcal{Y}_k$ has the following stratification:
        \begin{align*}
             \mathcal{Y}_k &= \coprod_{\tau\in\Delta_{\spe}\cap \Delta_Y}(\overline{Y^{X(\Delta)}}\cap O_\tau)\\
             &= \coprod_{\tau\in\Delta_{\spe}\cap \Delta_Y}(\coprod_{1\leq i\leq r_{\tau}}E^{(i)}_{\tau}).
       \end{align*}
       For any $\tau\in \Delta_{\spe}\cap \Delta_Y$, there exists a ray $\gamma\in\Delta_{\spe}$ such that $\gamma\preceq\tau$. 
       Because $\Delta_Y$ is a subfan of $\Delta$, $\gamma\in\Delta_Y$. 
       Moreover, for any $\tau_1$ and $\tau_2\in\Delta_{\spe}\cap \Delta_Y$, $\overline{Y^{X(\Delta)}}\cap O_{\tau_1} \subset \overline{Y^{X(\Delta)}}\cap \overline{O_{\tau_2}}$ if and only if $\tau_2\subset \tau_1$ by the argument of toric varieties. 
       In conclusion, the following equation holds by Proposition \ref{prop: stratification}(d):
       \[
             \mathcal{Y}_k = \bigcup_{\substack{\gamma\in\Delta_{\spe}\cap \Delta_Y\\\dim(\gamma) = 1}}(\bigcup_{1\leq i\leq r_{\gamma}}\overline{E^{(i)}_{\gamma}}).
        \]        
        Furthermore, we remark that for any $\tau_1$ and $\tau_2\in\Delta_{\spe}\cap \Delta_Y$, any $1\leq i_1\leq r_{\tau_1}$, and any $1\leq i_2\leq r_{\tau_2}$, if we have $\overline{E^{(i_1)}_{\tau_1}} = \overline{E^{(i_2)}_{\tau_2}}$, then $\tau_1 = \tau_2$ and $i_1 = i_2$. 
        Indeed, because $E^{(i_1)}_{\tau_1} \cap \overline{E^{(i_2)}_{\tau_2}} \neq\emptyset$, we have $O_{\tau_1}\cap\overline{O_{\tau_2}}\neq\emptyset$. 
        Thus, $\tau_2\subset \tau_1$ by the argument of toric varieties. 
        Similarly, $\tau_1\subset \tau_2$ too. 
        Moreover, we have $i_1 = i_2$. 
                
        Thus, $\overline{E^{(i)}_{\gamma}}$ is an irreducible component of $\mathcal{Y}_k$ for any ray $\gamma\in \Delta_{\spe}\cap \Delta_Y$ and $1\leq i\leq r_\gamma$, and hence, the following set is equal to $\mathfrak{I}$ by the remark above:
        \[
            \bigcup_{\substack{\gamma\in\Delta_{\spe}\cap \Delta_Y\\ \dim(\gamma) = 1}}\biggl\{\overline{E^{(j)}_{\gamma}}\biggr\}_{1\leq j\leq r_{\gamma}}.
        \]
        \item[(b)]  Let $1\leq j'\leq s$ be an integer and let $E\in\mathcal{S}(\mathcal{W}_{j'})$. 
        By Proposition \ref{prop: model I}(f), $\mathcal{Y}_k = \coprod_{1\leq j\leq s} (\mathcal{W}_{j})_k$. 
        Thus, all irreducible components of $(\mathcal{W}_{j'})_k$ are in $\mathfrak{I}$. 
        Hence, by (a), there exist rays $\gamma_1, \ldots, \gamma_m\in \Delta_{\spe}\cap \Delta_Y$ and integers $\{i_l\}_{1\leq l\leq m}$ such that $1\leq i_l\leq r_{\gamma_l}$ for any $1\leq l\leq m$, and $E$ is a connected component of the following closed subset:
        \[
            \bigcap_{1\leq l\leq m} \overline{E^{(i_l)}_{\gamma_l}}.
        \]
        In particular, by Proposition \ref{prop: stratification}(d), $\cap_{1\leq l\leq m} \overline{O_{\gamma_l}} \neq \emptyset$. 
        Thus, there exists $\sigma\in\Delta$ such that $\gamma_l\preceq\sigma$ for any $1\leq l\leq m$. 
        Hence, by the assumption, there exists $\tau\in\Delta_{\bdd}$ such that $\gamma_l\preceq\tau$ for any $1\leq l\leq m$. 
        We can take $\tau$ minimal because $\Delta$ is a fan. 
        Then the following equation holds by the assumption of $\tau$:
        \[
            \bigcap_{1\leq l\leq m} \overline{O_{\gamma_l}} = \overline{O_{\tau}}.
        \]
        Thus, $\overline{Y^{X(\Delta)}}\cap \overline{O_{\tau}}\neq \emptyset$ by the following inclusion:
        \[
        \emptyset \subsetneq E\subset \bigcap_{1\leq l\leq m} \overline{E^{(i_l)}_{\gamma_l}}\subset \overline{Y^{X(\Delta)}}\cap\bigcap_{1\leq l\leq m}\overline{O_{\gamma_l}}.
        \]
        In particular, 
        $\overline{Y^{X(\Delta)}}\cap O_{\tau}\neq \emptyset$ by Proposition \ref{prop: stratification}(b), and 
        $\tau\in\Delta_Y$. 
        Moreover, the following equation holds:
        \begin{align*}
            \overline{Y^{X(\Delta)}}\cap \overline{O_{\tau}}&= \bigcap_{1\leq l\leq m} (\overline{Y^{X(\Delta)}}\cap\overline{O_{\gamma_l}})\\
            &= \coprod_{\{\{h_l\}_{1\leq l\leq m}\mid1\leq h_l\leq r_{\gamma_l}\}} (\bigcap_{1\leq l\leq m} \overline{E^{(h_l)}_{\gamma_l}}).
        \end{align*}
        Thus, $E$ is also a connected component of $\overline{Y^{X(\Delta)}}\cap \overline{O_{\tau}}$. 
        Therefore, by Proposition \ref{prop: stratification}(d), there exists $1\leq i\leq r_{\tau}$ such that $E = \overline{E^{(i)}_{\tau}}$. 
        We already have shown the uniqueness in the proof of (a).
        \item[(c)] Let $\tau\in\Delta_{\bdd}\cap\Delta_Y$ and let $1\leq i\leq r_{\tau}$ be an integer. 
        Let $\gamma_1, \ldots, \gamma_m$ be all rays of $\tau$. 
        Because $\tau\in\Delta_{\bdd}$, we have $\gamma_1, \ldots, \gamma_m\in\Delta_{\bdd}$. 
        Moreover, $\gamma_1, \ldots, \gamma_m\in\Delta_Y$ because $\Delta_Y$ is a subfan of $\Delta$. 
        Because $\gamma_1, \ldots, \gamma_m$ are all rays of $\tau$, the following equation holds:
        \[
            \bigcap_{1\leq l\leq m} \overline{O_{\gamma_l}} = \overline{O_{\tau}}.
        \]
        Thus, like as the argument in proof of (b), the following equation holds:
        \begin{align*}
            \overline{Y^{X(\Delta)}}\cap \overline{O_{\tau}}&= \bigcap_{1\leq l\leq m} (\overline{Y^{X(\Delta)}}\cap\overline{O_{\gamma_l}})\\
            &= \coprod_{\{\{h_l\}_{1\leq l\leq m}\mid 1\leq h_l\leq r_{\gamma_l}\}} (\bigcap_{1\leq l\leq m}\overline{E^{(h_l)}_{\gamma_l}}).
        \end{align*}
        Hence, there exists integers $(i_l)_{1\leq l\leq m}$ such that $1\leq i_l\leq r_{\gamma_l}$ for any  $1\leq l\leq m$ and $\overline{E^{(i)}_{\tau}}$ is a connected component of the following closed subset:
        \[
            \bigcap_{1\leq l\leq m} \overline{E^{(i_l)}_{\gamma_l}}.
        \]
        By (a), for any $1\leq l\leq m$, $\overline{E^{(i_l)}_{\gamma_l}}$ is an irreducible component of $\mathcal{Y}_k$. 
        Now, $\overline{E^{(i)}_{\tau}}$ is non-empty, so there uniquely exists $1\leq j\leq s$ such that each $\overline{E^{(i_l)}_{\gamma_l}}$ is an irreducible component of $(\mathcal{W}_j)_k$ by Proposition \ref{prop: model I}(f). 
        This shows that $\overline{E^{(i)}_{\tau}}\in\mathcal{S}(\mathcal{W}_j)$. 
        \item[(d)] By (b), (c), and the remark in the proof of (a), we can check that the following equation holds:
        \[
            \coprod_{1\leq j\leq s}\mathcal{S}(\mathcal{W}_j) = \coprod_{\tau\in\Delta_{\bdd}\cap \Delta_Y}\{\overline{E^{(i)}_{\tau}}\}_{1\leq i\leq r_{\tau}}.
        \]
        We remark that $\{(\mathcal{W}_j)_{\Kt}\}_{\mathrm{sb}} = \{W^\circ_j\}_{\mathrm{sb}}$ for any $1\leq j\leq s$ by Proposition \ref{prop: model I}(e) and (f). 
        Thus, by Proposition \ref{prop: property1}(b), and Proposition \ref{prop: NONO21}, the following equation holds:
        \[
            \sum_{1\leq j\leq s}\VOL\biggl(\bigl\{W^\circ_j\bigr\}_{\mathrm{sb}}\biggr) = \sum_{\tau\in\Delta_{\bdd}\cap \Delta_Y}(-1)^{\dim(\tau)-1}\biggl(\sum_{1\leq i\leq r_{\tau}} \bigl\{E^{(i)}_{\tau}\bigr\}_{\mathrm{sb}}\biggr).
        \]
        \item[(e)]
        Let $\tau\in\Delta_{\spe}\cap\Delta_Y$.  
        Then there exists unique $\sigma\in\Sigma$ such that $\tau^\circ\subset\sigma^\circ$ because $\Delta$ is a refinement of $\Sigma$. 
        Because $\tau\in\Delta_{\spe}$, we have $\sigma\in\Sigma_{\spe}$. 
        Moreover, by Proposition \ref{prop: surjection}(b), $\overline{Y^{X(\Delta)}}\cap O_{\tau}$ is a trivial algebraic torus fibration over $\overline{Y^{X(\Sigma)}}\cap O_\sigma$ whose dimension of the fiber is $\dim(\sigma) -\dim(\tau)$. 
        In particular, $\sigma\in\Sigma_Y$. 
        Moreover, $r_{\tau} = r_\sigma$ and if it is necessary, we can replace the index such that $E^{(i)}_{\tau}$ is a trivial algebraic torus fibration of $E^{(i)}_{\sigma}$ for any $1\leq i\leq r_\sigma$. 

        Conversely, for any $\sigma\in\Sigma_{\spe}\cap\Sigma_Y$ and for $\tau\in\Delta$ such that $\tau^\circ\subset\sigma^\circ$, we have $\tau\in\Delta_{\spe}\cap\Delta_Y$ by the same argument. 

        Thus, the following equation holds from (d):
        \[
            \sum_{1\leq j\leq s}\VOL\biggl(\bigl\{W^\circ_j\bigr\}_
            {\mathrm{sb}}\biggr) = \sum_{\sigma\in\Sigma_{\spe}\cap \Sigma_Y}\biggl(\sum_{\substack{\tau\in\Delta_{\bdd}\\
            \tau^\circ\subset\sigma^\circ}}(-1)^{\dim(\tau)-1}\biggr)\biggl(\sum_{1\leq i\leq r_{\sigma}} \bigl\{E^{(i)}_{\sigma}\bigr\}_{\mathrm{sb}}\biggr).
        \]
                
        Now, we use the notation in \cite[$\S3$]{NPS16}. 
        Let $\sigma\in\Sigma_{\spe}$. 
        We can identify with $N_\Q$ and $N_\Q\times\{1\}$, so we can regard $\sigma\cap (N_\Q\times\{1\})$ as a $\Q$-rational polyhedron in $N_\Q$. 
        Similarly, for any $\tau\in\Delta_{\spe}$, we can regard $\tau\cap (N_\Q\times\{1\})$ as a $\Q$-rational polyhedron in $N_\Q$. 
        Moreover, there exists the following disjoint decomposition of $\Q$-rational polyhedrons because $\Delta$ is a refinement of $\Sigma$:
        \[
            \sigma^\circ\cap(N_\Q\times\{1\}) = \coprod_{\substack{\tau\in\Delta_{\spe}\\
            \tau^\circ\subset\sigma^\circ}}\tau^\circ\cap(N_\Q\times\{1\}).
        \]
        There exists the Euler characteristic $\chi'$ of definable subsets in $N_\Q$ by \cite[Lemma 9.6]{HK06}. 
        By \cite[Lemma 9.6]{HK06}, $\chi'$ has the additivity, so the following equation holds:
        \[
            \chi'(\sigma^\circ\cap(N_\Q\times\{1\})) = \sum_{\substack{\tau\in\Delta_{\spe}\\
            \tau^\circ\subset\sigma^\circ}}\chi'(\tau^\circ\cap(N_\Q\times\{1\})).
        \]
        Furthermore, for any $\tau\in\Delta_{\spe}$, the following equation holds by \cite[Proposition 3.7]{NPS16}: 
        \begin{equation*}
            \chi'(\tau^\circ\cap(N_\Q\times\{1\}))=
            \begin{cases}
                (-1)^{\dim(\tau) -1}& \tau\in\Delta_{\bdd}, \\
                0                   & \tau\notin\Delta_{\bdd}.\\
            \end{cases}
        \end{equation*}
        Thus, for any $\sigma\in\Sigma_{\spe}\cap\Sigma_Y$, 
        there exists the following equation:
        \begin{align*}
            \sum_{\substack{\tau\in\Delta_{\bdd}\\
            \tau^\circ\subset\sigma^\circ}}(-1)^{\dim(\tau)-1}&= \sum_{\substack{\tau\in\Delta_{\bdd}\\
            \tau^\circ\subset\sigma^\circ}}\chi'(\tau^\circ\cap(N_\Q\times\{1\}))\\
            &= \sum_{\substack{\tau\in\Delta_{\spe}\\
            \tau^\circ\subset\sigma^\circ}}\chi'(\tau^\circ\cap(N_\Q\times\{1\}))\\
            &= \chi'(\sigma^\circ\cap(N_\Q\times\{1\})).
        \end{align*}
        Similarly, for any $\sigma\in\Sigma_{\spe}$, the following equation holds by \cite[Proposition 3.7]{NPS16}:
        \begin{equation*}
            \chi'(\sigma^\circ\cap(N_\Q\times\{1\}))=
            \begin{cases}
                (-1)^{\dim(\sigma) -1} & \sigma\in\Sigma_{\bdd}, \\
                0                     & \sigma\notin\Sigma_{\bdd}.\\
            \end{cases}
        \end{equation*}
        Thus, the following equation holds:
        \[
            \sum_{1\leq j\leq s}\VOL\biggl(\bigl\{W^\circ_j\bigr\}_{\mathrm{sb}}\biggr) = \sum_{\sigma\in\Sigma_{\bdd}\cap\Sigma_Y}(-1)^{\dim(\sigma)-1}\biggl(\sum_{1\leq i\leq r_{\sigma}} \bigl\{E^{(i)}_{\sigma}\bigr\}_{\mathrm{sb}}\biggr).
        \]
    \end{itemize}
\end{proof}
\section{General hypersurfaces in sch\"{o}n varieties}
In this section, we will show that a general hypersurface in a sch\"{o}n affine variety is also a sch\"{o}n affine variety, and the fan that obtains its compactification can be computed combinatorially. 
\subsection{Valuations on affine sch\"{o}n varieties}
Toric varieties have valuations associated with lattice points.  
In this subsection, we examine the valuations on the sch\"{o}n affine varieties associated with lattice points. 
In the case of general sch\"{o}n affine varieties, there exists no one-to-one relation with lattice points and valuations, but this relation can be applied to the construction of the sch\"{o}n hypersurfaces. 

In this section, we use the following notation:
\begin{itemize}
  \item Let $N$ be a lattice of finite rank. 
  \item Let $M$ be the dual lattice of $N$. 
  \item Let $Z$ be a closed subvariety of $T_N$ and $\iota$ denote the closed immersion $Z\hookrightarrow T_N$. 
  \item Let $m$ denote the multiplication morphism $T_N\times Z\rightarrow T_N$. 
  \item Let $\Delta$ be a good fan for $Z$. 
  We may assume there exists a strongly convex refinement $\Delta_1$ of $\Delta$ such that $\overline{Z^{X(\Delta_1)}}$ is a sch\"{o}n compactification. 
  We remark that $\supp(\Delta) = \mathrm{Trop}(Z)$ by Proposition \ref{prop: proper and schon}(b). 
  \item Let $\Ray_Z$ denote the set which consists of all rays $\gamma$ in $N_\RR$ such that $\gamma\subset \supp(\Delta) = \mathrm{Trop}(Z)$. 
  \item For $\gamma\in \Ray_Z$, let $\Val_{Z, \gamma}$ denote the set of all integer-valued divisorial valuations on $Z$ whose valuation rings coincide with local rings at the generic points of irreducible components of $\overline{Z^{X(\gamma)}}\cap O_\gamma$. 
  We remark that all irreducible components of $\overline{Z^{X(\gamma)}}\cap O_\gamma$ is a prime divisor of a normal variety $\overline{Z^{X(\gamma)}}$ by Proposition \ref{prop: property1}(b) and (d), and Proposition \ref{prop: surjection}(c). 
  \item Let $\Val_{Z}$ denote $\cup_{\gamma\in\Ray_Z} \Val_{Z, \gamma}$. 
  We remark that every $v\in \Val_{Z}$ is trivial on $k$. 
  \item For $v\in \Val_{Z}$, let $\tilde{v}$ denote the integer-valued divisorial valuation on $T_N\times Z$ such that $\tilde{v}(\sum_{\omega\in M}a_\omega\chi^\omega) = \min_{\{\omega\in M; a_\omega \neq 0\}}v(a_\omega)$ for any $\sum_{\omega\in M}a_\omega\chi^\omega\in k[T_N\times Z] = k[Z][M]$, where $a_\omega\in k[Z]$ for any $\omega\in M$.    
\end{itemize}
\begin{proposition}\label{prop: valuation on schon}
  With the notation above, let $v\in \Val_Z$ be a valuation on $Z$. 
  Let $\theta(v)$ denote the map $\mathrm{K}(T_N)^*\rightarrow \ZZ$ such that $\theta(v)(f) = \tilde{v}(m^*(f))$ for any $f\in \mathrm{K}(T_N)^*$. 
  Then the following statements hold:
  \begin{itemize}
    \item[(a)] The map $\theta(v)$ is a torus invariant valuation on $T_N$. 
    In particular, we can regard $\theta(v)$ as an element in $N$.   
    \item[(b)] Let $\gamma\in\Ray_Z$ be a ray such that $v\in \Val_{Z, \gamma}$. 
    Then $\gamma$ is generated by $\theta(v)$.  
    \item[(c)] The equation $v(\iota^*(\chi^\omega)) = \theta(v)(\chi^\omega) = \langle \theta(v), \omega\rangle$ holds for any $\omega\in M$. 
    \item[(d)] For any $w'\in \mathrm{Trop}(Z)\cap N\setminus\{0_N\}$, the number of elements of the set $\{v'\in \Val_Z\mid \theta(v') = w'\}$ is the number of irreducible components of $\overline{Z^{X(\gamma')}}\cap O_{\gamma'}$, where $\gamma'$ is a ray generated by $w'$.   
  \end{itemize}
  By this proposition, let $\theta$ denote the map $\Val_Z\rightarrow N$ as above. 
\end{proposition}
\begin{proof}
    We prove the statements from (a) to (d) in order. 
    \begin{itemize}
        \item[(a)] Because $\tilde{v}$ is a divisorial valuation and $m$ is dominant, $\theta(v)$ is also a divisorial valuation on $T_N$, which is trivial on $k$. 
        By the definition of $v$, there exists $\gamma\in\Ray_Z$ such that $v\in\Val_{Z, \gamma}$. 
        Let $E$ denote the irreducible component of $\overline{Z^{X(\gamma)}}\cap O_\gamma$ which $v$ is a valuation of $\OO_{\overline{Z^{X(\gamma)}}, E}$. 
        Then $\tilde{v}$ is a valuation of $\OO_{T_N\times \overline{Z^{X(\gamma)}}, T_N\times E}$ by the definition of $\tilde{v}$. 
        Because $m$ extends to the multiplication morphism $T_N\times \overline{Z^{X(\gamma)}}\rightarrow X(\gamma)$, $\theta(v)$ is a valuation of $\OO_{X(\gamma), O_\gamma}$. 
        Thus, $\theta(v)$ is a torus invariant divisorial valuation on $T_N$. 
        \item[(b)] By the proof of (a), $\theta(v)$ is an integer-valued valuation of $\OO_{X(\gamma), O_\gamma}$. 
        Thus, there uniquely exists $w\in \gamma\cap N\setminus\{0_N\}$ such that $\theta(v) = w$. 
        In particular, $\gamma$ is generated by $\theta(v)$. 
        \item[(c)] Let $\omega\in M$. 
        By the definition of $\tilde{v}$, it follows that $\tilde{v}(\iota^*(\chi^\omega)\chi^\omega) = v(\iota^*(\chi^\omega))$. 
        Moreover, we can check that $m^*(\chi^\omega) = \iota^*(\chi^\omega)\chi^\omega$, and hence, $\theta(v)(\chi^\omega) = \tilde{v}(\iota^*(\chi^\omega)\chi^\omega) = v(\iota^*(\chi^\omega))$. 
        \item[(d)] Let $u\in \gamma'\cap N\setminus\{0_N\}$ be a primitive element, let $\omega\in M$ be an element such that $\langle u, \omega\rangle = 1$, and let $m'$ denote the multiplication morphism $T_N\times \overline{Z^{X(\gamma')}}\rightarrow X(\gamma')$. 
        Then $\chi^\omega$ generates the ideal associated with the closed subscheme $O_{\gamma'}\subset X(\gamma')$, and there exists the following Cartesian diagram by Lemma \ref{lem: action}(a): 
        \begin{equation*}
            \begin{tikzcd}
            T_N\times (\overline{Z^{X(\gamma')}}\cap O_{\gamma'})\ar[r]\ar[d]& O_{\gamma'}\ar[d]\\
            T_N\times \overline{Z^{X(\gamma')}} \ar[r, "m'"] & X(\gamma').
            \end{tikzcd}
        \end{equation*}
        By the diagram above and Proposition \ref{prop: property1}(c), $m^*(\chi^\omega) = \iota^*(\chi^\omega)\chi^\omega$ is a uniformizer of $\OO_{T_N\times \overline{Z^{X(\gamma')}}, T_N\times E'}$ for any irreducible component $E'$ of $\overline{Z^{X(\gamma')}}\cap O_{\gamma'}$. 
        In particular, $\iota^*(\chi^\omega)$ is a uniformizer of $\OO_{\overline{Z^{X(\gamma')}}, E'}$ for any irreducible component $E'$ of $\overline{Z^{X(\gamma')}}\cap O_{\gamma'}$. 

        Let $r\in\ZZ_{>0}$ and $E'$ be an irreducible component of $\overline{Z^{X(\gamma')}}\cap O_{\gamma'}$. 
        Then by the argument above, there exists a valuation of $v'$ of $\OO_{\overline{Z^{X(\gamma')}}, E'}$ such that $v'(\iota^*(\chi^\omega)) = r$. 
        Then $\langle\theta(v'), \omega\rangle = v'(\iota^*(\chi^\omega)) = r$ by (c). 
        Thus, $\theta(v') = ru$ by (b) and the equation $\langle u, \omega\rangle = 1$. 

        Conversely, if $\theta(v'_1) = \theta(v'_2)$ and $v'_1$ and $v'_2$ is an integer-valued valuation of $\OO_{\overline{Z^{X(\gamma')}}, E'}$, then $v'_1(\iota^*(\chi^\omega)) = \theta(v'_1)(\chi^\omega) = \theta(v'_2)(\chi^\omega) = v'_2(\iota^*(\chi^\omega))$. 
        Because $\iota^*(\chi^\omega)$ is a uniformizer of $\OO_{\overline{Z^{X(\gamma')}}, E'}$, it follows that $v'_1 = v'_2$. 

        Therefore, the statement holds. 
    \end{itemize}
\end{proof}
It is well known fact that the torus invariant valuations can characterize the global section ring of affine toric varieties. 
The following proposition claims that this fact also holds for sch\"{o}n affine varieties. 
\begin{proposition}\label{prop: criterion of affine open}
  Let $\sigma\in\Delta_1$. 
  Because the affine scheme $Z$ is an open subscheme of the affine scheme $\overline{Z^{X(\Delta_1)}}\cap X(\sigma) = \overline{Z^{X(\sigma)}}$, we can regard $k[\overline{Z^{X(\Delta_1)}}\cap X(\sigma)]$ as a subring of $k[Z]$. 
  Then the following statements hold:
  \begin{itemize}
      \item[(a)] Let $w\in\sigma^\circ\cap N$ and $v\in\theta^{-1}(w)$ be a valuation. 
      Then there uniquely exists an irreducible component $E$ of $\overline{Z^{X(\Delta_1)}}\cap O_\sigma$ such that the valuation ring of $v$ dominates $\OO_{\overline{Z^{X(\Delta_1)}}, E}$. 
      \item[(b)] $k[\overline{Z^{X(\Delta_1)}}\cap X(\sigma)] = \{f\in k[Z]\mid v(f)\geq 0\quad (\forall v\in \theta^{-1}(\sigma\cap N))\}$.
  \end{itemize}
\end{proposition}
\begin{proof}
    We will prove the statements from (a) to (b).
    \begin{itemize}
        \item[(a)] Let $\gamma$ be a ray generated by $w$ in $M_\RR$. 
        Then $v\in \Val_{Z, \gamma}$ by Proposition \ref{prop: valuation on schon}(b). 
        Let $F$ denote the irreducible components of $\overline{Z^{X(\gamma)}}\cap O_\gamma$ such that $v$ is a valuation of $\OO_{\overline{Z^{X(\gamma)}}, F}$ and let $\mu\colon \overline{Z^{X(\gamma)}}\rightarrow \overline{Z^{X(\Delta_1)}}$ be a restriction morphism of $X(\gamma)\rightarrow X(\Delta_1)$ to $\overline{Z^{X(\gamma)}}$. 
        Then there uniquely exists an irreducible component $E$ of $\overline{Z^{X(\Delta_1)}}\cap O_\sigma$ such that $\mu(F) = E$ by Proposition \ref{prop: surjection}(b). 
        In particular, $\OO_{\overline{Z^{X(\gamma)}}, F}$ dominates $\OO_{\overline{Z^{X(\Delta_1)}}, E}$.
        \item[(b)] First, we show that the right-hand side of the equation above contains $k[\overline{Z^{X(\Delta_1)}}\cap X(\sigma)]$. 
        Let $v\in\theta^{-1}(\sigma\cap N)$, let $\gamma\in\Ray_Z$ be a cone such that $v\in\Val_{Z, \gamma}$, $E$ be an irreducible component of $\overline{Z^{X(\gamma)}}\cap O_\gamma$ such that $v$ is a valuation of $\OO_{\overline{Z^{X(\gamma)}}, E}$, and let $\tau\preceq\sigma$ be a face such that $\theta(v)\in\tau^\circ$. 
        By Proposition \ref{prop: valuation on schon}(b), $\theta(v)$ generates $\gamma$, and hence, $\gamma^\circ\subset \tau^\circ$. 
        By (a), there exists an irreducible component $F$ of $\overline{Z^{X(\Delta_1)}}\cap O_\tau$ such that $\OO_{\overline{Z^{X(\gamma)}}, E}$ dominates $\OO_{\overline{Z^{X(\Delta_1)}}, F}$. 
        Thus, $v(f)\geq 0$ for any $f\in k[\overline{Z^{X(\Delta_1)}}\cap X(\sigma)]$. 

        Next, we show that $k[\overline{Z^{X(\Delta_1)}}\cap X(\sigma)]$ contains the right-hand side of the equation above. 
        Let $f\in k[Z]$ be an element such that $f$ is contained in the right-hand side of the equation above, let $W$ denote $\overline{Z^{X(\sigma)}}\setminus Z$, and let $\{E^{(i)}_\gamma\}_{1\leq i\leq r_\gamma}$ be irreducible components of $\overline{Z^{X(\sigma)}}\cap O_\gamma$ for each ray $\gamma\preceq\sigma$. 
        For showing $f\in k[\overline{Z^{X(\Delta_1)}}\cap X(\sigma)]$, it is enough to show that $f\in \OO_{\overline{Z^{X(\sigma)}}, D}$ for all prime Weil divisor $D$ of  $\overline{Z^{X(\sigma)}}$ in $W$ because $\overline{Z^{X(\sigma)}}$ is a normal affine scheme of finite type over $k$ by Proposition \ref{prop: property1}(d). 
        By the argument of the toric varieties and Proposition \ref{prop: property1}(b) and Proposition \ref{prop: stratification}(d), the set of all prime Weil divisors in $W$ is the following set:
        \[
            \bigcup_{\gamma\preceq\sigma, \dim(\gamma) = 1}\bigcup_{1\leq i\leq r_\gamma} \{\overline{E^{(i)}_\gamma}\}.
        \]
        Thus, for each prime Weil divisor $D$ of $\overline{Z^{X(\sigma)}}$ in $W$ and each integer-valued valuation $v$ whose valuation ring is $\OO_{\overline{Z^{X(\sigma)}}, D}$, there exists a ray $\gamma\preceq\sigma$ such that $v\in\Val_{Z, \gamma}$. 
        Then $\theta(v)\in\sigma\cap N$ by Proposition \ref{prop: valuation on schon} (b). 
        In particular, $v\in\theta^{-1}(\sigma\cap N)$, and hence, $f\in \OO_{\overline{Z^{X(\sigma)}}, D}$ by the assumption of $f$. 
    \end{itemize}
\end{proof}
\subsection{Linear system on a sch\"{o}n variety}
In this subsection, we show that a general hypersurface in the linear system, which is generated by units on sch\"{o}n affine varieties, has a sch\"{o}n compactification. 
Moreover, we construct its fan combinatorially. 

In this section, we use the following notation:
\begin{itemize}
  \item Let $k$ be an algebraically closed field of characteristic 0. 
  \item Let $N$ be a lattice of finite rank. 
  \item Let $M$ be the dual lattice of $N$. 
  \item Let $Z$ be a closed subvariety of $T_N$ and let $\iota$ denote the closed immersion $Z\hookrightarrow T_N$. 
  \item Let $S$ be a non-empty finite set. 
  \item Let $M^S$ denote the set of all maps from $S$ to $M$.  
  \item For $u, v\in M^S$, if there exists $\omega\in M$ such that $u(i) - v(i) = \omega$ for any $i\in S$, we note $u\sim v$. 
  This relation $\sim$ is an equivalence relation of $M^S$. 
  \item For $u\in M^S$, $P(u)$ denote the convex closure of $u(S)$ in $M_\RR$, $\Sigma(u)$ denote the normal fan in $N_\RR$ of $P(u)$. 
  We remark that $\Sigma(u)$ is a rational polyhedral convex fan in $N_\RR$. 
  \item Let $\Delta$ denote a fan in $N_\RR$ and $u\in M^S$. 
  Let $\Sigma(\Delta, u)$ denote the set $\{\sigma_1\cap \sigma_2\mid \sigma_1\in \Delta, \sigma_2\in\Sigma(u)\}$.  
  We remark that $\Sigma(\Delta, u)$ is a rational polyhedral convex fan in $N_\RR$ and the refinement of $\Delta$ by Lemma \ref{lem: compute normal fan}(c). 
\end{itemize}
The following proposition is related to the convex geometry, and it can be applied to later propositions. 
\begin{proposition}\label{prop: normal fan I}
  Let $u, v\in M^S$ and let $\Delta$ be a fan in $N_\RR$. 
  We assume that $u\sim v$. 
  Then the following statements hold:
  \begin{itemize}
    \item[(a)] The equations $\Sigma(u) = \Sigma(v)$ and $\Sigma(\Delta, u) = \Sigma(\Delta, v)$ hold. 
    \item[(b)] For any $\sigma\in\Sigma(\Delta, u)$, there exists $\omega_\sigma\in M$ such that the following conditions hold: 
      \begin{itemize}
        \item[(i)] $u(i) - \omega_\sigma\in \sigma^\vee$ for any $i\in S$. 
        \item[(ii)] There exists $i\in S$ such that $u(i) - \omega_\sigma\in \sigma^\perp\cap M$. 
      \end{itemize}
    \item[(c)] We keep the notation in (b). 
    Let $\omega'_\sigma\in M$ be another element that satisfies two conditions in (b). 
    Then $\omega_\sigma - \omega'_\sigma\in \sigma^\perp\cap M$. 
    \item[(d)] We keep the notation in (b). 
    A set $S^\sigma = \{i\in S\mid u(i) - \omega_\sigma\in \sigma^\perp\cap M\}$ is non-empty and independent of the choice of $\omega_\sigma\in M$ such that $\omega_\sigma$ satisfies the two conditions in (b).
  \end{itemize}
\end{proposition}
\begin{remark}
    In fact, $\sigma$ does not need to be a cone in  $\Delta$, but this assumption is added as it becomes necessary in later applications.
\end{remark}
\begin{proof}
    We prove the statements from (a) to (d) in order: 
    \begin{itemize}
        \item[(a)] There exists $\omega\in M$ such that $u(i) -v(i) = \omega$ for any $i\in S$. 
        Then $P(u) = P(v) + \omega$, and hence, $\Sigma(u) = \Sigma(v)$. 
        Moreover, $\Sigma(\Delta, u) = \Sigma(\Delta, v)$ by the definition of $\Sigma(\Delta, u)$ and $\Sigma(\Delta, v)$. 
        \item[(b)] By the definition of $\Sigma(\Delta, u)$, there exists $\tau\in\Sigma(u)$ and a face $Q\preceq P(u)$ such that $\sigma\subset\tau$ and $\tau = \{v\in N_\RR\mid \langle v, \omega_1-\omega_2\rangle\geq 0, \forall \omega_1\in P(u), \forall \omega_2\in Q\}$. 
        Thus, there exists $i_0\in S$ such that $u(i_0)\in Q$, and hence, $\langle v', u(i)-u(i_0)\rangle\geq 0$ for any $v'\in \sigma$ and $i\in S$. 
        Therefore, $u(i) - u(i_0)\in \sigma^\vee$ for any $i\in S$, and we take $u(i_0)$ as $\omega_\sigma$. 
        \item[(c)] Let $i\in S$ be an element such that $u(i)-\omega_\sigma\in \sigma^\perp\cap M$. 
        Then $\omega_\sigma - \omega_{\sigma'} = (\omega_\sigma - u(i)) + (u(i)-\omega_{\sigma'})\in \sigma^\vee\cap M$. 
        Similarly, $\omega_{\sigma'} - \omega_\sigma\in \sigma^\vee\cap M$, and hence, $\omega_\sigma - \omega'_\sigma\in \sigma^\perp\cap M$. 
        \item[(d)] This is obvious by (c). 
    \end{itemize}
\end{proof}
Later, we shall construct sch\"{o}n hypersurfaces in a sch\"{o}n affine variety in a general condition. 
The following definition is important for the computation of the definition polynomials of the stratification of sch\"{o}n hypersurfaces. 
The detail can be found in Proposition \ref{prop: generic condition of flatness}(d). 
\begin{definition}\label{def: orbit-polytope}
    We keep the notation in Proposition \ref{prop: normal fan I}. 
    Let $\sigma\in\Sigma(\Delta, u)$ and let $\omega_\sigma\in M$ be an element that satisfies the conditions in Proposition \ref{prop: normal fan I}(b). 
    Then we define the map $u^\sigma\colon S^\sigma\rightarrow \sigma^\perp\cap M$ as follows: 
    \[
        S^\sigma\ni i\mapsto u(i)-\omega_\sigma\in \sigma^\perp\cap M.
    \]
    We remark that $u^\sigma$ is well-defined up to the equivalence relation on $(\sigma^\perp\cap M)^{S^\sigma}$ by Proposition \ref{prop: normal fan I}(c) and (d). 
    Moreover, for any $u, v\in M^S$ such that $u\sim v$, we can check that $u^\sigma\sim v^\sigma$. 
\end{definition}
The following proposition shows that a general member in the liner system generated by $u\in M^S$  has a tropical compactification. 
\begin{proposition}\label{prop: generic condition of flatness}
  With the notation above, let $u\in M^S$ be a map. 
  We assume that $\Delta$ is a good fan for $Z$, $\Sigma = \Sigma(\Delta, u)$ is strongly convex, and $\overline{Z^{X(\Sigma)}}$ is a sch\"{o}n compactification of $Z$. 
  
  For $a = (a_i)_{i\in S}\in k^S$, let $F(a)$ denote $\sum_{i\in S}a_i\iota^*(\chi^{u(i)})\in k[Z]$, let $J(a)\subset k[Z]$ denote the ideal generated by $F(a)$, let $H(a)$ denote a closed subscheme in $Z$ defined by $J(a)$, let $\overline{H(a)^{X(\Sigma)}}$ denote the scheme theoretic closure of $H(a)$ in $X(\Sigma)$, and let $(\omega_\sigma)_{\sigma\in\Sigma}\in M^{\Sigma}$ denote a family which each $\omega_\sigma$ satisfies two conditions in Proposition \ref{prop: normal fan I} (b) for $\sigma\in\Sigma$. 
  We remark that $\overline{H(a)^{X(\Sigma)}}$ is also the scheme theoretic closure of $H(a)$ in $\overline{Z^{X(\Sigma)}}$. 
  
  Then the following statements hold:
  \begin{itemize}
    \item[(a)] It follows that $\iota^*(\chi^{-\omega_\sigma})F(a)\in k[\overline{Z^{X(\Sigma)}}\cap X(\sigma)]$ for any $\sigma\in \Sigma(\Delta, u)$ and $a\in k^S$. 
    \item[(b)] Let $\sigma\in \Sigma(\Delta, u)$. 
    There exists a dense open subset $U_\sigma\subset \A^{|S|}_k$ such that it holds that $v(\iota^*(\chi^{-\omega_\sigma})F(a))$
    $ = 0$ for any $v\in \theta^{-1}(\sigma^\circ\cap N)$ and any $a\in U_\sigma(k)$. 
    Let $U$ denote $\bigcap_{\sigma\in\Sigma(\Delta, u)}U_\sigma$.  
    \item[(c)] We keep the notation in (b). 
    Let $\sigma\in \Sigma(\Delta, u)$ and $a\in U(k)$.  
    Then $J(a)\cap k[\overline{Z^{X(\Sigma)}}\cap X(\sigma)]$ is generated by $\iota^*(\chi^{-\omega_\sigma})F(a)$.
    \item[(d)] We keep the notation in (b). 
    Let $p^\sigma$ denote the quotient morphism $k[\overline{Z^{X(\Sigma)}}\cap X(\sigma)]\rightarrow k[\overline{Z^{X(\Sigma)}}\cap O_\sigma]$ and let $\iota^\sigma$ denote the closed immersion $\overline{Z^{X(\Sigma)}}\cap O_\sigma\hookrightarrow O_\sigma$ for $\sigma\in \Sigma(\Delta, u)$. 
    
    Then $p^\sigma(\iota^*(\chi^{-\omega_\sigma})F(a))|_E\neq 0$ and the following equation holds for any $\sigma\in \Sigma(\Delta, u)$, any irreducible component $E$ of $\overline{Z^{X(\Sigma)}}\cap O_\sigma$, and any $a\in U(k)$:
    \[
      p^\sigma(\iota^*(\chi^{-\omega_\sigma}F(a)) = 
      \sum_{i\in S^\sigma} 
      a_i p^{\sigma}(\iota^*(\chi^{u(i)-\omega_\sigma}))
       = \sum_{i\in S^\sigma} 
      a_i (\iota^{\sigma})^*(\chi^{u^\sigma(i)}).
    \]
    \item[(e)] We keep the notation in (b). 
    The multiplication morphism $T_N\times\overline{H(a)^{X(\Sigma)}}\rightarrow X(\Sigma)$ is flat for any $a\in U(k)$. 
  \end{itemize}
\end{proposition}
\begin{proof}
    We prove the statements from (a) to (e) in order:
    \begin{itemize}
        \item[(a)] Because $\iota^*(\chi^{-\omega_\sigma})F(a) = \sum_{i\in S}a_i\iota^*(\chi^{u(i)-\omega_\sigma})$, it is enough to show that $\iota^*(\chi^{u(i)-\omega_\sigma})\in k[\overline{Z^{X(\Sigma)}}\cap X(\sigma)]$ for any $i\in S$. 
        For any $v\in \theta^{-1}(\sigma\cap N)$, $v(\iota^*(\chi^{u(i)-\omega_\sigma})) = \langle \theta(v), u(i) - \omega_\sigma\rangle \geq 0$ by Proposition \ref{prop: valuation on schon}(c) and the definition of $\omega_\sigma$. 
        Thus, $\iota^*(\chi^{u(i)-\omega_\sigma})\in k[\overline{Z^{X(\Sigma)}}\cap X(\sigma)]$ for any $i\in S$ by Proposition \ref{prop: criterion of affine open}(b). 
        \item[(b)] Let $E^{(1)}_\sigma, E^{(2)}_\sigma, \ldots, E^{(r)}_\sigma$ denote irreducible components of $\overline{Z^{X(\Sigma)}}\cap O_\sigma$, let $w\in \sigma^\circ\cap N$, and let $\gamma$ denote a ray in $M_\RR$ generated by $w$. 
        Then the number of the irreducible components of $\overline{Z^{X(\gamma)}}\cap O_\gamma$ is $r$ by Proposition \ref{prop: surjection}(b). 
        By Proposition \ref{prop: valuation on schon}(d), $|\theta^{-1}(w)| = r$, and let $\{v_1, \ldots, v_r\}$ denote $\theta^{-1}(w)$. 
        By Proposition \ref{prop: criterion of affine open}(a), there is one-to-one correspondence of $\theta^{-1}(w)$ and irreducible components of $\overline{Z^{X(\Sigma)}}\cap O_\sigma$.
        Thus,  if it is necessary, we can replace the index of $\{v_l\}_{1\leq l\leq r}$ such that the valuation ring of $v_l$ dominates $\OO_{\overline{Z^{X(\Sigma)}}, E^{(l)}_\sigma}$ for each $1\leq l\leq r$. 
        For $q\in\ZZ$ and $1\leq l\leq r$, let $V_{l, q}$ denote the following subset of $k^{|S|}$: 
        \[
            V_{l, q} := \{a\in k^{S}\mid v_l(\iota^*(\chi^{-\omega_\sigma})F(a)) \geq q \}.
        \]
        Because each $v_l$ is a valuation that is trivial on $k$, $V_{l, q}$ is a linear subspace of $k^{|S|}$. 
        Moreover, $V_{l, 0} = k^{|S|}$ by (a), and there exists $i\in S$ such that $v_l(\iota^*(\chi^{u(i) - \omega_\sigma})) = 0$ by Proposition \ref{prop: valuation on schon}(c). 
        Thus, we can regard $V_{l, 0}\setminus V_{l, 1}\subset k^{|S|} = \A^{|S|}_k(k)$ as a dense open subset of $\A^{|S|}_k$. 
        Let $U_\sigma$ denote the open subset $\cap_{1\leq l\leq r}(V_{l, 0}\setminus V_{l, 1})$ of $\A^{|S|}_k$. 
        Then $\iota^*(\chi^{-\omega_\sigma}F(a)) \in \OO^*_{\overline{Z^{X(\Sigma)}}, E^{(l)}_\sigma}$ for any $1\leq l\leq r$ and $a\in U_\sigma(k)$ by Proposition \ref{prop: criterion of affine open}(a). 
        Let $w'\in \sigma^\circ\cap N$, and let $v'\in \theta^{-1}(w')$.
        Then there exists $1\leq l'\leq r$ such that valuation ring of $v'$ dominates $\OO_{\overline{Z^{X(\Sigma)}}, E^{(l')}_\sigma}$  by Proposition \ref{prop: criterion of affine open}(a), and hence, $v'(\iota^*(\chi^{-\omega_\sigma})F(a)) = 0$ for any $a\in U_\sigma(k)$. 
        \item[(c)] Let $\sigma\in\Sigma$. 
        First, we will show that $v(\iota^*(\chi^{-\omega_\sigma})F(a)) = 0$ for any $v\in\theta^{-1}(\sigma\cap N)$ and $a\in U(k)$.
        Let $v\in \theta^{-1}(\sigma\cap N)$ be a valuation. 
        Then there exists a face $\tau\preceq\sigma$ such that $\theta(v)\in\tau^\circ$. 
        Because $\sigma^\vee\subset\tau^\vee$ and $\sigma^\perp\subset\tau^\perp$, $\omega_\sigma - \omega_\tau\in\tau^\perp\cap M$ by Proposition \ref{prop: normal fan I}(c). 
        In particular, $v(\iota^*(\chi^{\omega_\tau-\omega_\sigma})) = \langle\theta(v), \omega_\tau-\omega_\sigma\rangle = 0$ by Proposition \ref{prop: valuation on schon}(c). 
        Thus, $v(\iota^*(\chi^{-\omega_\sigma})F(a)) = v(\iota^*(\chi^{\omega_\tau-\omega_\sigma})) + v(\iota^*(\chi^{-\omega_\tau})F(a)) = 0$ for any $a\in U(k)$ by (b). 

        Second, We will show that the ideal $J(a)\cap k[\overline{Z^{X(\Sigma)}}\cap X(\sigma)]$ is generated by the element $\iota^*(\chi^{-\omega_\sigma})F(a)$ for any $a\in U(k)$. 
        By (a), $\iota^*(\chi^{-\omega_\sigma})F(a)\in J(a)\cap k[\overline{Z^{X(\Sigma)}}\cap X(\sigma)]$. 
        Conversely, let $g\in J(a)\cap k[\overline{Z^{X(\Sigma)}}\cap X(\sigma)]$. 
        Then there exists $h\in k[Z]$ such that $g = h\cdot\iota^*(\chi^{-\omega_\sigma})F(a)$ in $k[Z]$. 
        For any $v\in\theta^{-1}(\sigma\cap N)$, $v(g)\geq 0$ and $v(\iota^*(\chi^{-\omega_\sigma})F(a)) = 0$ by the argument above and Proposition \ref{prop: criterion of affine open}(b).
        Thus, $v(h)\geq 0$, and hence $h\in k[\overline{Z^{X(\Sigma)}}\cap X(\sigma)]$ by Proposition \ref{prop: criterion of affine open}(b).
        \item[(d)] By the proof of (b), we showed that $\iota^*(\chi^{-\omega_\sigma})F(a)\in\OO^*_{\overline{Z^{X(\Sigma)}}, E}$, and hence, it follows that $p^\sigma(\iota^*(\chi^{-\omega_\sigma})F(a))|_E\neq 0$ for any irreducible component $E$ of $\overline{Z^{X(\Sigma)}}\cap O_\sigma$. 
        For any $i\notin S^\sigma$, $v(\iota^*(\chi^{u(i) - \omega_\sigma})) > 0$ for any $v\in \theta^{-1}(\sigma^\circ\cap N)$ by Proposition \ref{prop: valuation on schon}(c). 
        Thus, $\iota^*(\chi^{u(i) - \omega_\sigma})$ is contained in the maximal ideal of $\OO_{\overline{Z^{X(\Sigma)}}, E}$ by Proposition \ref{prop: criterion of affine open}(a), and hence, $p^\sigma(\iota^*(\chi^{u(i)-\omega_\sigma}))|_E = 0$ for any irreducible component $E$ of $\overline{Z^{X(\Sigma)}}\cap O_\sigma$. 
        Therefore, the statement holds. 
        \item[(e)] Let $x\in T_N\times\overline{H(a)^{X(\Sigma)}}$. 
        Then there exists $\sigma\in\Sigma$ such that $x\in T_N\times (\overline{H(a)^{X(\Sigma)}}\cap O_\sigma)$. 
        Let $N_0$ denote a sublattice of $N$ such that $N_0\oplus(\langle\sigma\rangle\cap N) = N$, let $p$ denote the quotient morphism $N\rightarrow N/N_0$, and let $\sigma_0$ denote the strongly convex rational polyhedral cone $p_\RR(\sigma)$ in $(N/N_0)_\RR$. 
        Then there exists the following Cartesian product by Lemma \ref{lem: action}(c). 
        \begin{equation*}
          \begin{tikzcd}
            T_{N}\times (\overline{H(a)^{X(\Sigma)}}\cap X(\sigma))\ar[rr]\ar[d, "p_*\times\id"] & & X(\sigma)\ar[d, "p_*"]\\
            T_{N/N_0}\times (\overline{H(a)^{X(\Sigma)}}\cap X(\sigma)) \ar[r, "\id\times p_*|_{\overline{H}}"] & T_{N/N_0}\times X(\sigma_0)\ar[r, "m_0"] & X(\sigma_0),
          \end{tikzcd}
        \end{equation*}
        where the upper morphism is the multiplication morphism, and $m_0$ is the action morphism of $X(\sigma_0)$. 
        
        Let $y$ denote $(p_*\times\id)(x)\in T_{N/N_0}\times X(\sigma)$, let $\alpha$ denote $m_0\circ(\id\times p_*)|_{T_{N/N_0}\times (\overline{Z^{X(\Sigma)}}\cap X(\sigma))}$, and let $\beta$ denote $m_0\circ(\id\times p_*)|_{T_{N/N_0}\times (\overline{H(a)^{X(\Sigma)}}\cap X(\sigma))}$. 
        We remark that $\alpha^{-1}(O_{\sigma_0}) = T_{N/N_0}\times (\overline{Z^{X(\Sigma)}}\cap O_{\sigma})$. 
        By Lemma \ref{lem: action}(c), there exists the following Cartesian diagram:
        \begin{equation*}
          \begin{tikzcd}
            T_{N}\times (\overline{Z^{X(\Sigma)}}\cap X(\sigma))\ar[r]\ar[d, "p_*\times\id"] & X(\sigma)\ar[d, "p_*"]\\
            T_{N/N_0}\times (\overline{Z^{X(\Sigma)}}\cap X(\sigma)) \ar[r, "\alpha"] & X(\sigma_0).
          \end{tikzcd}
        \end{equation*}
        Because $\overline{Z^{X(\Sigma)}}$ is a tropical compactification of $Z$, $\alpha$ is flat. 
        Let $E$ be an irreducible component of $\overline{Z^{X(\Sigma)}}\cap O_\sigma$ such that $y\in T_{N/N_0}\times E$.  
        By the first Cartesian diagram, if $\beta$ is flat at $y$, then the multiplication morphism $T_N\times \overline{H(a)^{X(\Sigma)}}\rightarrow X(\Sigma)$ is flat at $x$. 

        Thus, we will show that $\beta$ is flat at $y$. 
        Let $g(a)$ denote $\iota^*(\chi^{-\omega_\sigma})F(a)$. 
        Then $1\otimes p^\sigma(g(a))\in \Gamma(T_{N/N_0}\times (\overline{Z^{X(\Sigma)}}\cap O_\sigma), \OO_{T_{N/N_0}\times (\overline{Z^{X(\Sigma)}}\cap O_\sigma)})$. 
        By (d), $1\otimes p^\sigma(g(a))|_{T_{N/N_0}\times E}$ is not a zero divisor in $\Gamma(T_{N/N_0}\times E, \OO_{T_{N/N_0}\times E})$. 
        Thus, $1\otimes p^\sigma(g(a))$ is not a zero divisor in $\OO_{\alpha^{-1}(O_{\sigma_0}), y}$. 
        On the other hand, $1\otimes g(a)$ generates the ideal of $k[T_{N/N_0}\times (\overline{Z^{X(\Sigma)}}\cap X(\sigma))]$ associated with the closed subscheme $T_{N/N_0}\times (\overline{H(a)^{X(\Sigma)}}\cap X(\sigma))$ of $T_{N/N_0}\times (\overline{Z^{X(\Sigma)}}\cap X(\sigma))$ by (c). 
        Then $\beta$ is flat at $y$ by Lemma \ref{lem: criterion of flatness}. 
      \end{itemize}
\end{proof}
In Proposition \ref{prop: generic condition of flatness} (e), 
the multiplication morphism of a general hypersurface is flat, but it is not faithfully flat in general. 
The following proposition shows how to compute the subfan of $\Sigma$, which obtains the tropical compactification of a general hypersurface in sch\"{o}n affine varieties. 
\begin{proposition}\label{prop: general condition of fan}
  We keep the notation in Proposition \ref{prop: generic condition of flatness}. 
  Let $\sigma\in \Sigma = \Sigma(\Delta, u)$ and let $\{E^{(1)}_{\sigma}, E^{(2)}_{\sigma}, \ldots, E^{(r_\sigma)}_{\sigma}\}$ be all irreducible components of $\overline{Z^{X(\Sigma)}}\cap O_\sigma$. 
  Let $V(\sigma, l, u, \omega_\sigma)$ denote a $k$-linear subspace of $\Gamma(E^{(l)}_\sigma, \OO_{\overline{Z^{X(\Sigma)}}\cap O_\sigma})$ generated by the following set:
  \[
    \{p^{\sigma}(\iota^*(\chi^{u(i)-\omega_\sigma}))|_{E^{(l)}_\sigma}= (\iota^\sigma)^*(\chi^{u(i)-\omega_\sigma})|_{E^{(l)}_\sigma}\}_{i\in S^\sigma}.
  \] 
  
  Then the following statements hold:
  \begin{itemize}
    \item[(a)] The integer $\dim_k(V(\sigma, l, u, \omega_\sigma))$ is independent of the choice of $\omega_\sigma\in M$ which satisfies the conditions in Proposition \ref{prop: normal fan I}(b). 
    Let $d_{Z}(\sigma, l, u)$ denote $\dim_k(V(\sigma, l, u,  \omega_\sigma))$.
    \item[(b)] For the notation in (a), $d_{Z}(\sigma, l, u)\geq 1$ for any $\sigma\in \Sigma$ and any $1\leq l\leq r_\sigma$.
    \item[(c)] For the notation in (a), there exists a dense open subset $W_{\sigma, l}\subset \A^{|S|}_k$ such that if $d_{Z}(\sigma, l, u)\geq 2$, then $\overline{H(a)^{X(\Sigma)}}\cap E^{(l)}_\sigma \neq \emptyset$ for any $a\in W_{\sigma, l}(k)$. 
    
    Let $W$ denote $\cap_{\sigma\in\Sigma}(\cap_{1\leq l\leq r_\sigma}W_{l, \sigma})$.  
    \item[(d)] Let $\Sigma(\Delta, u, Z)$ denote the subset of $\Sigma(\Delta, u)$ as follows:
    \[
      \Sigma(\Delta, u, Z) = \{\sigma\in \Sigma(\Delta, u)\mid 
      \max_{1\leq l\leq r_\sigma}\{d_{Z}(\sigma, l, u)\}\geq 2\}. 
    \]
    Then $\Sigma(\Delta, u, Z)$ is a subfan of $\Sigma(\Delta, u)$ and equal to the following subset of $\Sigma(\Delta, u)$ for a genera $a\in k^{|S|}$: 
    \[
        \{\sigma\in \Sigma(\Delta, u)\mid \overline{H(a)^{X(\Sigma)}}\cap O_\sigma\neq \emptyset\}.
    \]
  \end{itemize}
\end{proposition}
\begin{proof}
  We will prove the statement from (a) to (d) in order. 
  \begin{itemize}
    \item[(a)] Let $\omega'_\sigma\in M$ be an element that satisfies two conditions in Proposition \ref{prop: normal fan I}(b). 
    Then $\omega_\sigma - \omega'_\sigma\in \sigma^\perp\cap M$ by Proposition \ref{prop: normal fan I}(c). 
    In particular, $\chi^{\omega_\sigma - \omega'_\sigma}\in k[O_\sigma]^*$, and hence, $p^\sigma(\iota^*(\chi^{\omega_\sigma - \omega'_\sigma)})|_{E^{(l)}_\sigma}$ is a unit in $\Gamma(E^{(l)}_\sigma, \OO_{\overline{Z^{X(\Sigma)}}\cap O_\sigma})$. 
    Thus, $V(\sigma, l, u, \omega_\sigma)$ and $V(\sigma, l, u, \omega'_\sigma)$ is isomorphic by the multiplication by this unit. 
    \item[(b)] By Proposition \ref{prop: generic condition of flatness}(d), $p^\sigma(\iota^*(\chi^{-\omega_\sigma})F(a))|_{E^{(l)}_\sigma}\neq 0$ for any $a\in U(k)$, and hence,  $V(\sigma, l, u, \omega_\sigma) \neq 0$. 
    \item[(c)] If $d_Z(\sigma, l, u) = 1$, then we take $W_{\sigma, l}$ as $U$. 
    We assume $d_Z(\sigma, l, u) \geq 2$. 
    By Proposition \ref{prop: generic condition of flatness}(c) and (d), $\sum_{i\in S^\sigma}a_i p^\sigma(\iota^*(\chi^{u(i)-\omega_\sigma}))|_{E^{(l)}_\sigma}$ generates the ideal of $k[E^{(l)}_\sigma]$ associated with a closed subscheme $\overline{H(a)^{X(\Sigma)}}\cap E^{(l)}_\sigma$ of $E^{(l)}_\sigma$ for any $a\in U(k)$. 
    Moreover, there exists an open subset $W'_{\sigma, l}$ of $\A^{|S|}_k$ such that $\sum_{i\in S^\sigma}b_i p^\sigma(\iota^*(\chi^{u(i)-\omega_\sigma}))|_{E^{(l)}_\sigma}$ is not a unit in $k[E^{(l)}_\sigma]$ for any $b = (b_i)_{i\in S}\in W'_{\sigma, l}(k)$ by Lemma \ref{lem: dim 2 lemma}(b). 
    Therefore, we take $W_{\sigma, l}$ as $U\cap W'_{\sigma, l}$.
    \item[(d)] Let $a\in (U\cap W)(k)$  and $\sigma\in\Sigma$. 
    If $\max_{1\leq l\leq r_\sigma}\{d_{Z}(\sigma, l, u)\} = 1$, then $\overline{H(a)^{X(\Sigma)}}\cap O_{\sigma} = \emptyset$ by Proposition \ref{prop: generic condition of flatness}(d). 
    On the other hand, if $\max_{1\leq l\leq r_\sigma}\{d_{Z}(\sigma, l, u)\} \geq 2$, then $\overline{H(a)^{X(\Sigma)}}\cap E^{(l)}_{\sigma} \neq \emptyset$ by (c). 
    Moreover, the multiplication morphism $T_N\times \overline{H(a)^{X(\Sigma)}}\rightarrow X(\Sigma)$ is open by Proposition \ref{prop: generic condition of flatness}(e). 
    Thus, $\Sigma(\Delta, u, Z)$ is a subfan of $\Sigma(\Delta, u)$.  
  \end{itemize}
\end{proof}
By Proposition \ref{prop: generic condition of flatness} and Proposition \ref{prop: general condition of fan}, we obtained the tropical compactification of a general hypersurface in sch\"{o}n affine varieties. 
By considering a more general case, it can be shown to be sch\"{o}n by the following proposition: 
\begin{proposition}\label{prop: general condition of schon}
  We keep the notation in Proposition \ref{prop: generic condition of flatness} and Proposition \ref{prop: general condition of fan}. 
  Let $\Sigma_Z$ denote $\Sigma(\Delta, u, Z)$, let $\overline{H(a)^{X(\Sigma_Z)}}$ denote the scheme theoretic closure of $H(a)$ in $X(\Sigma_Z)$ for $a\in k^S$. 
  Let $m$ denote the multiplication morphism $T_N\times\overline{H(a)^{X(\Sigma_Z)}}\rightarrow X(\Sigma_Z)$. 
  Then there exists a dense open subset $V$ of $\A^{|S|}_k$ such that the following conditions hold for any $a\in V(k)$: 
  \begin{itemize}
    \item[(i)] The scheme $H(a)$ is equidimensional. 
    \item[(ii)] The scheme $\overline{H(a)^{X(\Sigma_Z)}}$ is proper over $k$.  
    \item[(iii)] The morphism $m$ is smooth and faithfully flat. 
  \end{itemize} 
\end{proposition}
\begin{proof}
  For $\sigma\in\Sigma_Z$, all irreducible components of $\overline{Z^{X(\Sigma_Z)}}\cap O_\sigma$ are smooth by \ref{prop: property1}(c). 
  Then there exists a dense open subset $Q_{\sigma, l}$ of $\A^{|S|}_k$ such that the hypersurface in $E^{(l)}_\sigma$ defined by $\sum_{i\in S^\sigma}c_i p^\sigma((\iota^{\sigma})^*(\chi^{u(i)-\omega_\sigma}))|_{E^{(l)}_\sigma}$ is smooth over $k$ for any $c = (c_i)_{i\in S}\in Q_{\sigma, l}(k)$ by Bertini's Theorem for each $1\leq l\leq r_\sigma$. 
  Let $Q$ denote $\cap_{\sigma\in\Sigma_Z}(\cap_{1\leq l\leq r_\sigma}Q_{\sigma, l})$, let $V$ denote $U\cap W\cap Q$, and let $a\in V(k)$. 
  Because $a\in W(k)\cap Q(k)$, $H(a)$ is a smooth hypersurface in $Z$. 
  In particular, $H(a)$ is equidimensional. 
  Because $\overline{Z^{X(\Sigma)}}$ is a tropical compactification of $Z$, $\overline{H(a)^{X(\Sigma)}}$ is also proper. 
  Thus, $\overline{H(a)^{X(\Sigma_Z)}}$ is proper, and the multiplication morphism $T_N\times \overline{H(a)^{X(\Sigma_Z)}}\rightarrow X(\Sigma_Z)$ is faithfully flat by Proposition \ref{prop: generic condition of flatness}(e), the proof of Proposition \ref{prop: general condition of fan}(d), and the fact that $a\in U(k)\cap W(k)$. 
  Finally, for any $\sigma\in\Sigma_Z$, $\overline{H(a)^{X(\Sigma_Z)}}\cap O_\sigma$ is smooth over $k$ by Proposition \ref{prop: generic condition of flatness}(c) and (d), and the fact $a\in U(k)\cap Q(k)$. 
  Thus, the multiplication morphism $T_N\times \overline{H(a)^{X(\Sigma_Z)}}\rightarrow X(\Sigma_Z)$ is smooth by Proposition \ref{prop: property1}(c). 
\end{proof}
In Proposition \ref{prop: model I} and Proposition \ref{prop: model II}, we need a fan that satisfies three conditions, i.e. generically unimodular, specifically reduced, and compactly arranged. 
In general, the fan that obtains the tropical compactification does not satisfy these conditions.
By using the following proposition, we change the ambient toric variety, which satisfies these conditions. 
The geometrical interpretation is the semistable reduction. 
\begin{proposition}\label{prop: modify by multiplication}
    Let $N$ be a lattice of finite rank, and let $M$ be the dual lattice. 
    We identify with $M\oplus\ZZ$ and the dual lattice of $N\oplus\ZZ$ naturally. 
    Let $S$ be a finite set, let $u\in M^S, \kappa\in \ZZ^S$ be maps, let $(u, \kappa)$ denote a map $S\rightarrow M\oplus\ZZ$ defined by the product of $u$ and $\kappa$, and let $\psi_l$ denote the endomorphism of $N\oplus\ZZ$ defined as follows: 
    \[
        N\oplus\ZZ\ni (v, n)\mapsto (lv, n)\in N\oplus\ZZ.
    \]
    Then the following statements hold:
    \begin{itemize}
        \item[(a)] The following equation holds for any positive integer $l$:
        \[
            \Sigma((u, l\kappa)) = \{(\psi_l)_\RR(\sigma)\mid \sigma\in \Sigma((u, \kappa))\}. 
        \]
        \item[(b)] Let $\Delta$ be a fan in $N_\RR$. 
        Then the following equation holds for any positive integer $l$:
        \[
            \Sigma(\Delta\times\Delta_!, (u, l\kappa)) = \{(\psi_l)_\RR(\sigma)\mid \sigma\in \Sigma(\Delta\times\Delta_!, (u, \kappa))\}. 
        \]
        \item[(c)] We keep the notation in (b). 
        Let $\Sigma_l$ denote $\Sigma(\Delta\times\Delta_!, (u, l\kappa))$ and let $\Sigma_{l, +}$ denote $\{\sigma\in \Sigma_l\mid \sigma\subset N_\RR\times \RR_{\geq 0}\}$ for positive integer $l$. 
        Then there exists a positive integer $l_0$ such that for any positive integer $n$, there exists a strongly convex, generically unimodular, specifically reduced, compactly arranged refinement $\Delta'_n$ of $\Sigma_{nl_0, +}$. 
    \end{itemize}
\end{proposition}
\begin{proof}
  We will prove the statements from (a) to (c) in order. 
  \begin{itemize}
    \item[(a)] By the definition of $\psi_l$, 
    $(\psi^*_l)_\RR(\omega, m) = (l\omega, m)$ for any $(\omega, m)\in (M\oplus\ZZ)_\RR$. 
    In particular, $(\psi^*_l)_\RR(P((u, l\kappa))) = l\cdot P((u, \kappa))$, and hence, the statement holds.
    \item[(b)] For any $\tau\in\Delta\times\Delta_!$ and $\gamma\in\Sigma((u, \kappa))$, $(\psi_l)_\RR(\tau\cap\gamma) = \tau\cap (\psi_l)_\RR(\gamma)$ because $(\psi_l)_\RR(\tau) = \tau$. 
    Thus, the statement holds by (a).
    \item[(c)] Let $m\in\ZZ_{>0}$ and let $\Delta'$ be a strongly convex rational polyhedral fan in $(N\oplus\ZZ)_\RR$ which is a refinement of $\Sigma_{m, +}$. 
    We assume that $\Delta'$ is generically unimodular, specifically reduced, and compactly arranged. 
    For $l\in\ZZ_{>0}$, let $\Delta'_l$ denote the following set:
    \[
        \Delta'_l = \{(\psi_l)_\RR(\tau)\mid\tau\in\Delta'\}.
    \]
    By (b), $\Delta'_l$ is a strongly convex rational polyhedral fan in $N'_\RR$ and a refinement of $\Sigma_{lm, +}$. 
    Now, we show that $\Delta'_l$ is also generically unimodular, specifically reduced, and compactly arranged for any $l\in\ZZ_{>0}$. 
    
    By the definition of $\psi_l$, we can check that $\Delta'_l$ is also generically unimodular. 
    
    Let $\tau\in\Delta'$. 
    Then by the definition of $\psi_l$, we can check that $\tau\in\Delta'_{\spe}$ if and only if $(\psi_l)_\RR(\tau)\in(\Delta'_l)_{\spe}$. 
    Similarly, we can check that $\tau\in\Delta'_{\bdd}$ if and only if $(\psi_l)_\RR(\tau)\in(\Delta'_l)_{\bdd}$. 
    Therefore, $\Delta'_l$ is also compactly arranged. 

    We remark that $(\psi_l)_\RR$ induces a one-to-one correspondence with rays $\gamma\in\Delta'_{\spe}$ and those in $(\Delta'_l)_{\spe}$ by the definition of $\psi_l$.
    Let $\gamma\in\Delta'_{\spe}$ be a ray and let $(v, c)\in N\oplus\ZZ$ be a minimal generator of $\gamma$. 
    By the assumption, $c = 1$. 
    Thus, $(lv, 1)$ is a generator of $(\psi_l)_\RR(\gamma)$. 
    Hence, $\Delta'_l$ is specifically reduced. 

    By the argument above, it is enough to show that the following statement (**) holds:
    \begin{enumerate}
        \item[(**)] There exists $m\in\ZZ_{>0}$ and $\Delta'$ be a strongly convex rational polyhedral fan in $(N\oplus\ZZ)_\RR$ such that $\Delta'$ is a refinement of $\Sigma_{m, +}$, generically unimodular, specifically reduced, and compactly arranged. 
    \end{enumerate}
    Let $\Delta''$ be a strongly convex unimodular refinement of $\Sigma_{1, +}$. 
    For a positive integer $m$, let $\Delta''_m$ denote the following set:
    \[
        \Delta''_m = \{(\psi_m)_\RR(\tau)\mid\tau\in\Delta''\}.
    \]
    By (b), the argument above, and Proposition \ref{prop: example of type of polytope}(b), we can check that $\Delta''_m$ is a strongly convex rational polyhedral fan in $(N\oplus\ZZ)_\RR$, a refinement of $\Sigma_{m, +}$, generically unimodular, and compactly arranged for any $m\in\ZZ_{>0}$. 
    Let $\Gamma$ denote the following subset of $\Delta''$:
    \[
        \Gamma = \{\gamma\in\Delta''_{\spe}\mid \dim(\gamma) = 1\}.
    \]
    For $\gamma\in\Gamma$, let $v_\gamma$ denote an element in $N_\Q$ such that $(v_\gamma, 1)\in\gamma$. 
    Because $|\Gamma|$ is finite, there exists $m_0\in\ZZ_{>0}$ such that $m_0v_\gamma\in N$ for any $\gamma \in\Gamma$. 
    Therefore, we can check that $\Delta''_{m_0}$ is specifically reduced.
  \end{itemize}
\end{proof}
In general, finding an explicit description of a stratification is complicated.  
This is also different from toric geometry. 
The following proposition indicates that, under favorable conditions on the cone, computing the defining ideals of such varieties is relatively straightforward.
\begin{proposition}\label{prop: easy-computation}
    Let $N$ and $N_0$ be lattices of finite rank, let $\pr_1\colon N_0\oplus N\rightarrow N_0$ be the first projection, let $\pr_2\colon N_0\oplus N\rightarrow N$ be the second projection, let $\sigma$ be a strongly convex rational polyhedral cone in $(N_0\oplus N)_\RR$, let $\sigma_0$ denote the rational polyhedral cone $(\pr_1)_\RR(\sigma)$ in $(N_0)_\RR$, let $Z$ be a closed subscheme of $T_N$, and let $Z' = T_{N_0\oplus N}\times_{T_N} Z$ denote the closed subscheme $T_{N_0\oplus N}$. 
    We assume that $(\pr_2)_\RR(\sigma) = \{0\}$. 
    By this assumption, $\sigma_0$ is strongly convex. 
    Then the following statements hold:
    \begin{itemize}
        \item[(a)] By the assumption, there exists an isomorphism $X(\sigma)\rightarrow X(\sigma_0)\times T_N$. 
        Then on the isomorphism above, $\overline{Z'^{X(\sigma)}}$ and $X(\sigma_0)\times Z$ are isomorphic. 
        \item[(b)] By the assumption, there exists an isomorphism $O_\sigma\rightarrow O_{\sigma_0}\times T_N$. 
        Then on the isomorphism above, $\overline{Z'^{X(\sigma)}}\cap O_{\sigma}$ and $O_{\sigma_0}\times Z$ are isomorphic.
    \end{itemize}
\end{proposition}
\begin{proof}
  We prove the statements from (a) to (b). 
  \begin{itemize}
    \item[(a)] 
    We remark that $Z' = T_{N_0}\times Z$ on the isomorphism. 
    Moreover, there exists the following Cartesian diagram:
    \begin{equation*}
      \begin{tikzcd}
        T_{N_0}\times Z\ar[r]\ar[d]& T_{N_0}\times T_N\ar[r, "\pr_1"]\ar[d]& T_{N_0}\ar[d]\\
        X(\sigma_0)\times Z\ar[r]& X(\sigma_0)\times T_N\ar[r, "\pr_1"]& X(\sigma_0),
      \end{tikzcd}
    \end{equation*}
    where the left horizontal morphisms are closed immersions, and vertical morphisms are open immersions. 
    In particular, the composition of the lower morphisms is flat, and hence, the scheme theoretic closure of $T_{N_0}\times Z$ in $X(\sigma_0)\times T_N$ is $X(\sigma_0)\times Z$ by Lemma \ref{lem: flat-closure lemma}. 
    \item[(b)] It is obvious by (a).  
  \end{itemize}
\end{proof}
\section{Application: Rationality of hypersurfaces in Gr(2, n)}
In this section, we apply the results in the previous section for the stable rationality of a very general hypersurface in $\Gr_\C(2, n)$. 
    
    In this section, we use the following notation. 
    \begin{itemize}
        \item Let $n$ be a positive integer greater than 3. 
        Let $k$ be an uncountable algebraically closed field of $\charac(k) = 0$. 
        \item Let $I$ denote the following set:
        \[
            I = \{(i, j)\in\ZZ^2\mid 0\leq i\leq j\leq n-3\}.
        \]
        \item Let $\{e^{i,j}\}_{(i, j)\in I}$ denote a canonical basis of $\ZZ^{|I|}$ and let $\mathbf{1}\in \ZZ^{|I|}$ be $\sum_{(i, j)\in I}e^{i,j}$. 
        \item Let $N$ denote $\ZZ^{|I|}/\ZZ\mathbf{1}$, let $\Pi$ denote the quotient morphism $\ZZ^{|I|}\rightarrow N$, and let $e_{i,j}\in N$ denote $\Pi(e^{i,j})$ for $(i, j)\in I$. 
        \item Let $(\ZZ^{|I|})^\vee$ denote the dual lattice of $\ZZ^{|I|}$ and let $\{\omega_{i, j}\}_{(i, j)\in I}$ denote the dual basis of $\{e^{i,j}\}_{(i, j)\in I}$. 
        \item Let $M$ denote a sublattice of $(\ZZ^{|I|})^\vee$ as follows: 
        \[
            M = \{\sum_{(i, j)\in I}a_{i, j}\omega_{i, j}\in(\ZZ^{|I|})^\vee\mid \sum_{(i, j)\in I} a_{i, j} = 0\}.
        \]
        We remark that we can regard $M$ as the dual lattice of $N$. 
        \item Let $N^\dagger$ be a lattice of rank $n - 1$ and  let $\{e^\dagger_j\}_{-1\leq j\leq n-3}$ be a basis of $N^\dagger$. 
        Let $M^\dagger$ be the dual lattice of $N^\dagger$ and let $\{\eta_j\}_{-1\leq j\leq n-3}$ be the dual basis of $\{e^\dagger_j\}_{-1\leq j\leq n-3}$. 
        \item Let $\{Y_{i}\}_{0\leq i\leq n-3}$ be a homogeneous coordinate function of $\PP^{n-3}_k$. 
        
        Let $L_{i, j}\in\Gamma(\PP^{n-3}_k, \OO_{\PP^{n-3}_k}(1))$ be homogeneous functions as follows for $(i, j)\in I$: 
        \[
            L_{i,j} = \left\{
                \begin{array}{ll}
                    Y_i, & i = j,\\
                    Y_i - Y_j, & i < j.
                \end{array}
            \right.
        \]
        \item Let $\iota$ denote a closed immersion from $\PP^{n-3}_k$ to $\PP^{|I| - 1}_k$ defined by $\{L_{i, j}\}_{(i, j)\in I}$. 
        \item Let $\{X_{i,j}\}_{(i, j)\in I}$ be a homogeneous coordinate function of $\PP^{|I| - 1}_k$.
        \item Let $\mathcal{B}$ denote $\{L_{i, j}\}_{(i, j)\in I}$. 
        We remark that $\mathcal{B}$ generates $\Gamma(\PP^{n-3}_k, \OO(1))$. 
        Let $\mathcal{V}$, $\mathcal{C}$, $\Delta(\mathcal{B})$, $Z$ be data defined as those in section 3.2. 
        \item Let $H\in k[y_1, \ldots, y_{n-3}]$ be a polynomial defined as follows: 
        \[
            H = \prod_{1\leq j\leq n-3}y_j \prod_{1\leq j\leq n-3}(1- y_j)\prod_{1\leq i<j\leq n-3}(y_j - y_i).
        \]
        \item Let $o$ denote a ring morphism from $k[y_1, \ldots, y_{n-3}]_H$ to $k[Z]$ defined as $o(y_j) = \frac{Y_j}{Y_0}$ for any $0\leq j\leq n-3$. 
        We can check that $o$ is isomorphic. 
        From now, we regard $k[Z]$ as $k[y_1, \ldots, y_{n-3}]_H$. 
        We remark that for any $(i, j)\in I$, the following equation holds:
        \[
            \iota^*(\frac{X_{i, j}}{X_{0. 0}}) = \left\{
                \begin{array}{ll}
                    y_i, & i = j,\\
                    y_i - y_j, & i < j,
                \end{array}
            \right.
        \]
        where $y_0 = 1$. 
        \item Let $e_t$ be a generator of $\ZZ$ and $\delta$ be the dual basis of $\ZZ^\vee$. 
        \item Let $\frac{X_{i, j}}{X_{l, m}}\in k[M]$ denote the torus invariant mononomial associated with $\omega_{i, j} - \omega_{l, m}\in M$ for $(i, j), (l, m)\in I$, let $x_j\in k[M^\dagger]$ denote the torus invariant mononomial associated with $\eta_j\in M^\dagger$, and let $t\in k[\ZZ^\vee]$ be the torus invariant mononomial associated with $\delta\in\ZZ^\vee$. 
        \item Let $N'$ denote $N\oplus N^\dagger$ and let $\pi$ denote the first projection $N'\rightarrow N$. 
        \item Let $Z'$ denote the closed subscheme $Z\times_{T_N}T_{N'}$ of $T_{N'}$, let $\Delta$ denote the fan in $N'_\RR$ defined as $\Delta = \{\pi_\RR^{-1}(\sigma)\mid \sigma\in\Delta(\mathcal{B})\}$, and let $\iota'$ denote the closed immersion $Z'\hookrightarrow T_{N'}$. 
        \item Let $Z''$ denote $Z'\times \mathbb{G}^1_{m, k}$ and let $\iota''$ denote the closed immersion $Z''\hookrightarrow T_{N'\oplus\ZZ}$ defined as the product of $\iota'$ and $\id_{\mathbb{G}^1_{m, k}}$. 
        \item Let $J$ denote $\{(i, j)\in\ZZ^2\mid 0\leq i<j\leq n-1\}$. 
        
        \item Let $M_k(n, 2)$ denote a set of all $n\times 2$-matrices over $k$. 
        A set $M_k(n, 2)$ has a natural $\mathrm{GL}_k(2)$-action, and we can identify with $M_k(n, 2)/\mathrm{GL}_k(2)$ and $\Gr_k(2, n)(k)$. 
        \item For $A = (a_{i_1, i_2})_{0\leq i_1\leq n-1, 0\leq i_2\leq 1}\in M_k(n, 2)$, and integers $(i, j)\in J$, let $d_{i, j}(A)$ denote $a_{i, 0}a_{j, 1} - a_{i, 1}a_{j, 0}$. 
        \item Let $\{W_{i, j}\}_{(i, j)\in J}$ be a homogeneous coordinate function of $\PP^{|J|-1}_k$. 
        \item Let $\mathrm{Pl}$ denote the following Pl\"{u}cker embedding:
            \begin{align*}
                &\mathrm{Pl}\colon \Gr_k(2, n)\ni [A = (a_{i_1, i_2})_{0\leq i_1\leq n-1, 0\leq i_2\leq 1}]\mapsto[d_{i, j}(A)]\in \PP^{|J| - 1}_k.
            \end{align*}
        \item Let $U^{i, j}$ denote the affine open subset of $\PP^{|J|-1}_k$ defined by $W_{i, j}\neq 0$ for $(i, j)\in J$, let $U$ denote open subset $\Gr_k(2, n)\cap U^{0, 1}$ of $\Gr_k(2, n)$, and let $\Gr^\circ_k(2, n)$ denote  open subset $\Gr_k(2, n)\cap \bigcap_{(i, j)\in J}U^{i, j}$ of $\Gr_k(2, n)$. 
        We remark that we can identify with $U$ and $\A^{2(n-2)}_k$ as follows: 
        \begin{multline*}
            M_k(n, 2)/\mathrm{GL}_k(2)\supset U(k) \ni
            \biggl[\begin{pmatrix}
                1 & 0 \\
                0 & 1 \\
                u_0 & v_0 \\
                u_1 & v_1 \\
                \vdots & \vdots\\
                u_{n-3} & v_{n-3} \\
            \end{pmatrix}
            \biggr]\\\longleftrightarrow
            (u_0, v_0, u_1, v_1, \ldots, u_{n-3}, v_{n-3})\in \A^{2(n-2)}_k(k).
        \end{multline*}  
        Let $\xi$ denote the closed immersion $\A^{2(n-2)}_k\hookrightarrow U^{0, 1}$ defined by the identification above with $U$ and $\A^{2(n-2)}_k$. 
        We regard $\{u_0, v_0, \ldots, u_{n-3}, v_{n-3}\}$ as coordinate functions of $\A^{2(n-2)}_k$. 
        Let $\xi^*$ denote a surjective ring morphism induced by $\xi$ from $k[\{\frac{W_{i, j}}{W_{0, 1}}\}_{(i, j)\in J}]$ to $k[u_0, v_0, \ldots, u_{n-3}, v_{n-3}]$. 
        \item We define $\{f_{i, j}\}_{(i, j)\in J}\subset k[u_0, v_0, \ldots, u_{n-3}, v_{n-3}]$ as follows: 
        \[
            f_{i,j} = \left\{
                \begin{array}{ll}
                    1, & (i,j) = (0, 1),\\
                    v_{j - 2}, & i = 0, j > 1,\\
                    -u_{j - 2}, & i = 1,\\
                    u_{i-2}v_{j-2} - u_{j-2}v_{i-2}, & i > 1.
                \end{array}
            \right.
        \]
        We remark that $\xi^*(\frac{W_{i, j}}{W_{0, 1}}) = f_{i, j}$ for any $(i, j)\in J$ by the definition of $\mathrm{Pl}$. 
        \item Let $F\in k[u_0, v_0, \ldots, u_{n-3}, v_{n-3}]$ be a polynomial defined as $F = \prod_{(i, j)\in J}f_{i, j}$.
        We remark that the inclusion morphism $k[u_0, v_0, \ldots, u_{n-3}, v_{n-3}]\rightarrow k[u_0, v_0, \ldots, u_{n-3}, v_{n-3}]_F$ induces an open immersion $\Gr^\circ_k(2, n)\hookrightarrow U$. 
        \item Let $\zeta$ denote a ring morphism of $k$-algebras $k[Z]\otimes_k k[M^\dagger]\rightarrow k[u_0, v_0, \ldots, u_{n-3}, v_{n-3}]_F$ defined as follows:
        \begin{align*}
            \zeta(x_{-1}) &= v_0u^{-1}_0,\\
            \zeta(x_j) &= u_j, &0\leq j\leq n-3,\\
            \zeta(y_j) &= u_0v^{-1}_0u^{-1}_jv_j, &1\leq j\leq n-3.
        \end{align*}
        We can check that $\zeta$ is well-defined and isomorphic. 
        \item For $(i, j)\in J$, let denote $\varpi_{i, j}\in M' = M\oplus M^\dagger$ as follows: 
                \[
                \varpi_{i, j} = \left\{
                    \begin{array}{ll}
                        0, & (i, j) = (0, 1),\\
                        \eta_{-1} + \eta_0, & (i, j) = (0, 2),\\
                        \eta_{-1} + \eta_{j - 2}  + (\omega_{j-2, j-2} - \omega_{0, 0}), & i = 0, j > 2,\\
                        \eta_{j - 2}, & i = 1,\\
                        \eta_{-1} + \eta_{i - 2} + \eta_{j - 2} + (\omega_{i - 2, j-2} - \omega_{0, 0}), & i > 1.
                    \end{array}
                \right.
             \]
        \item For $(i, j)\in J$, let $s_{i, j}\in k[Z]\otimes_k k[M^\dagger]$ be elements as follows: 
        \[
            s_{i,j} = \left\{
                \begin{array}{ll}
                    1, & (i,j) = (0, 1),\\
                    x_{-1}x_0, & (i, j) = (0, 2),\\
                    x_{-1}x_{j-2}y_{j-2}, & i = 0, j > 2,\\
                    -x_{j-2}, & i = 1,\\
                    x_{-1}x_{0}x_{j - 2}(y_{j - 2}- 1), &i = 2,\\
                    x_{-1}x_{i - 2}x_{j - 2}(y_{j - 2}- y_{i - 2}), &i > 2.\\
                \end{array}
            \right.
        \]
        We can check that $\zeta(s_{i, j}) = f_{i, j}$, $s_{i, j} = \pm{\iota'}^*(\chi^{\varpi_{i, j}})$, and $s_{i, j}$ is a unit of $k[Z]\otimes_k k[M^\dagger]$ for any $(i, j)\in J$. 
        \item Let $d\geq 2$ be a positive integer. 
        Let $S_{J, d}$ denote the following set: 
        \[
            S_{J, d} = \{\alpha = (\alpha_{i, j})\in\ZZ^{|J|}_{\geq 0}\mid \sum_{(i, j)\in J}\alpha_{i, j} = d\}.
        \]
        \item For $\alpha\in S_{J, d}$, let $s_\alpha$ denote $\prod_{(i, j)\in J} s^{\alpha_{i, j}}_{i, j}\in k[Z]\otimes_k k[M^\dagger]$ and 
        let $\varpi_\alpha\in M\oplus M^\dagger$ denote $\sum_{(i, j)\in J} \alpha_{i, j}\varpi_{i, j}$. 
        We remark that $s_\alpha = \pm\iota'^*(\chi^{\varpi_\alpha})$ for any $\alpha\in S_{J, d}$. 
        Let $\mathrm{sign}(\alpha)$ denote a sign such that $s_\alpha = \mathrm{sign}(\alpha)\iota'^*(\chi^{\varpi_{\alpha}})$ for $\alpha\in S_{J, d}$. 
        \item Let $X$ be a smooth closed subvariety of $\PP^{|J| - 1}_k$. 
        \item For $\alpha\in S_{J, d}$, let $\mathbb{W}^\alpha$ denote $\prod_{(i, j)\in J}W_{i, j}^{\alpha_{i, j}}$. 
        \item Let $\{t_\alpha\}_{\alpha\in S_{J, d}}$ denote coordinate functions of $\A^{|S_{J, d}|}_k$, let $\mathscr{H}^n_d$ denote the closed subvariety of $\PP^{|J| - 1}_k\times \A^{|S_{J, d}|}_k$ defined by $\sum_{\alpha\in S_{J, d}}\mathrm{sign}(\alpha)t_\alpha\mathbb{W}^\alpha = 0$, and let $\mathscr{X}_d$ denote the closed subscheme $X\times \A^{|S_{J, d}|}_k\cap \mathscr{H}^n_d$ of $\PP^{|J| - 1}_k\times \A^{|S_{J, d}|}_k$. 
        \item Let $\vartheta$ be the composition of the closed immersion $\mathscr{X}_d\hookrightarrow \PP^{|J| - 1}_k\times \A^{|S_{J, d}|}_k$ and the second projection $\PP^{|J| - 1}_k\times \A^{|S_{J, d}|}_k\rightarrow \A^{|S_{J, d}|}_k$. 
        \item Let $u\colon S_{J, d}\rightarrow M\oplus M^\dagger$ be a map such that $u(\alpha) = \varpi_\alpha$ for $\alpha\in S_{J, d}$. 
        \item Let $J_0, J_1$, and $J_2$ denote the following subset of $J$: 
        \begin{align*}
            J_0 &= \{(0, 1)\},\\
            J_1 &= \{(i, j)\in J\mid i < 2, j > 1\},\\
            J_2 &= \{(i, j)\in J\mid i > 1\}.
        \end{align*}
        We remark that $J = \coprod_{0\leq i\leq 2} J_i$. 
        \item For $\alpha\in S_{J, d}$, we define integers $c_0(\alpha), c_1(\alpha)$, and $c_2(\alpha)$ as follows:
            \begin{align*}
                c_0(\alpha) &= \alpha_{0, 1},\\
                c_1(\alpha) &= \sum_{(i, j)\in J_1}\alpha_{i, j},\\
                c_2(\alpha) &= \sum_{(i, j)\in J_2}\alpha_{i, j}.\\
            \end{align*}
        \item Let $\kappa$ denote a map $S_{J, d}\rightarrow \ZZ$ defined as follows: 
             \[
                \kappa(\alpha) = \left\{
                    \begin{array}{ll}
                        0, & c_1(\alpha) = d,\\
                        2(d - c_1(\alpha)) - 1, & c_1(\alpha) < d.\\
                    \end{array}
                \right.
             \]
        \item For non-negative integers $d_0, d_1$, and $d_2$, let $S_{d_0, d_1, d_2}$ denote the following subset of $S_{J, d}$: 
            \[
                S_{d_0, d_1, d_2} = \{\alpha\in S_{J, d}\mid (c_0(\alpha), c_1(\alpha), c_2(\alpha)) = (d_0, d_1, d_2)\}.
            \]
        \item Let $\Kt$ denote the field of the puise\"{u}x functions over $\C$. 
    \end{itemize}
    First, we prove the following proposition. 
    By this proposition, to show the nonstable rationality of a very general hypersurface over $k$ is enough to show so over $\Kt$.

    \begin{proposition}\label{prop: Bertini-stably rational}
        With the notation above, the following statements follow:
        \begin{enumerate}
            \item[(a)] There exists a non-empty open subset $U_{X, d}$ of $\A^{|S_{J, d}|}_k$ such that the restriction 
            
            $\vartheta|_{\vartheta^{-1}(U_{X, d})}\colon \vartheta^{-1}(U_{X, d})\rightarrow U_{X, d}$ is smooth and projective. 
            \item[(b)] Let $K/k$ be a field extension. 
            We assume that $K$ is algebraically closed. 
            Let $x \in (U_{X, d})_K(K)$. 
            We assume that $(\mathscr{X}_{d})_{x}$ is not stably rational over $K$. 
            Then a very general hypersurface of degree $d$ in $X$ is not stably rational over $k$. 
        \end{enumerate}
    \end{proposition}
    \begin{proof}
         \item[(a)] By Bertini's theorem, the statement holds. 
        \item[(b)] For any $y\in U_{X, d}$, we fix a geometric point $\overline{y}$ of $y$. 
        Let $A$ denote the following subset of $U_{X, d}$. 
        \[
            A = \{y\in U_{X, d}\mid (\mathscr{X}_d)_{\overline{y}} \mathrm{\ is\ stably \ rational.}\}.
        \]
        Let $h\colon (U_{X, d})_K\rightarrow U_{X, d}$ be a canonical map. 
        Then $h(x)\notin A$  because $K$ is algebraically closed. 
        Thus, by \cite[Cor.4.1.2]{NO21}, $A$ is countable unions of strict closed subsets of $U_{X, d}$. 
        Because $k$ is an uncountable field, a very general hypersurface of degree $d$ in $X$ is not stably rational over $k$.         
    \end{proof}
    Pl\"{u}cker embedding is not a sch\"{o}n compactification of $\Gr^\circ_\C(2, n)$. 
    Instead, we substitute a compactification of $Z\times T_{N^\dagger}\cong \Gr^\circ_\C(2, n)$.  
    We describe the relationship of equations between hypersurfaces in $\Gr_\C(2, n)$ and hypersurfaces in $Z\times T_{N^\dagger}$ in the following proposition.

    \begin{proposition}\label{prop:grassman-mock}
        Let $D = (d_\alpha)_{\alpha\in S_{J, d}}\in \Kt^{|S_{J, d}|}$ be a vector, let $H_D$ be a hypersurface of degree $d$ defined by $\sum_{\alpha\in S_{J, d}}\mathrm{sign}(\alpha)d_\alpha\mathbb{W}^\alpha = 0$, let $X_D\subset \PP^{|J| - 1}_\Kt$ be the scheme theoretic intersection of $\mathrm{Pl}(\Gr_\Kt(2, n))\subset \PP^{|J| - 1}_\Kt$ and $H_D$, and let $g_D\in \Kt[Z]\otimes_\Kt \Kt[M^\dagger]$ be the following polynomial: 
        \[
            g_D = \sum_{\alpha\in S_{J, d}}d_\alpha \iota'^*(\chi^{\varpi_\alpha}).
        \]
        Let $Y_D$ be a closed subscheme of $(Z\times T_{N^\dagger})\times_{\Spec(k)}\Spec(\Kt)$ defined by $g_D = 0$.
        Then for a general $D\in \Kt^{|S_{J, d}|}$, $X_D$ and $Y_D$ are irreducible and birational. 
        \end{proposition}
        \begin{proof}  
           Let $\Theta$ denote the following composition of morphisms: 
           \[
                (Z\times T_{N^\dagger})\times_{\Spec(k)}\Spec(\Kt)\xrightarrow{\sim} \Spec(\Kt[u_0, v_0, \ldots, u_{n-3}, v_{n-3}]_F)\xrightarrow{\sim}\Gr^\circ_\Kt(2, n)\xrightarrow{\mathrm{Pl}}\PP^{|J| - 1}_\Kt,
           \]
           where the first morphism is induced by $\zeta^{-1}$. 
           We can check that $\Theta^*(\frac{W_{i, j}}{W_{0, 1}}) = s_{i, j}$ for any $(i, j)\in J$ be the definitions of $\xi^*$ and $\zeta$. 
           In particular, $\Theta^*(\frac{\mathbb{W}^\alpha}{W^d_{0, 1}}) = s_\alpha = \mathrm{sign}(\alpha)\iota'^*(\chi^{\varpi_\alpha})$ for any $\alpha\in S_{J, \alpha}$. 
           We recall that $\Gr^\circ_\Kt(2, n)$ is an open subscheme of $\Gr_\Kt(2, n)$. 
           Thus, for a general $D$, $X_D$ and $Y_D$ are irreducible and birational. 
        \end{proof} 
	Now, we proceed with the calculations using the discussions in Section 4 and Section 5. 
    The proof of the following proposition is quite complex, but it is an essential result in the proof of the main theorem:
	
    \begin{proposition}\label{prop: computation result} 
        Let $S$ denote $S_{J, d}$,  let $a = (a_\alpha)_{\alpha\in S}\in k^{|S|}$, let $l\in\ZZ_{>0}$, and let $F(a)$ be the following element in $k[Z'']$: 
        \[
            F(a) = \sum_{\alpha\in S}a_\alpha \iota''^*(\chi^{(u, l\kappa)(\alpha)}). 
        \]
        Let $H(a)$ be a hypersurface in $Z''$ defined by $F(a)$. 
        Then the following statements hold:
        \begin{enumerate}
            \item[(a)] The rational polyhedral convex fan $\Sigma(\Delta\times\Delta_!, (u, \kappa))$ is strongly convex. 
            \item[(b)] Let $\Sigma_l$ denote $\Sigma(\Delta\times\Delta_!, (u, l\kappa))$ and let $\Sigma_{l, +}$ denote the fan $\{\sigma\in \Sigma_l\mid \sigma\subset N'_\RR\times \RR_{\geq 0}\}$. 
            Then $(\Sigma_{l, +})_{\bdd}$ consists of the following 7 cones: 
            \begin{align*}
                \tau_0 &= \RR_{\geq 0}(e_t -2l\sum_{0\leq j\leq n-3}e^\dagger_j),\\
                \tau_1 &= \RR_{\geq 0}(e_t -l\sum_{0\leq j\leq n-3}e^\dagger_j),\\
                \tau_2 &= \RR_{\geq 0}(e_t +l\sum_{0\leq j\leq n-3}e^\dagger_j),\\
                \tau_3 &= \RR_{\geq 0}(e_t +2l\sum_{0\leq j\leq n-3}e^\dagger_j),\\
                \sigma_0 &= \tau_0 + \tau_1,\\                    \sigma_1 &= \tau_1 + \tau_2,\\
                \sigma_2 &= \tau_2 + \tau_3.
            \end{align*} 
            \item[(c)] The following equations hold: 
            \begin{align*}
                S^{\tau_0} &= \bigcup_{1\leq i\leq d}S_{0, d-i, i},\\
                S^{\tau_1} &= S_{0, d-1, 1} \cup S_{0, d, 0},\\
                S^{\tau_2} &= S_{0, d, 0} \cup S_{1, d-1, 0},\\
                S^{\tau_3} &= \bigcup_{1\leq i\leq d}S_{i, d-i, 0},\\
                S^{\sigma_0} &= S_{0, d-1, 1},\\
                S^{\sigma_1} &= S_{0, d, 0},\\
                S^{\sigma_2} &= S_{1, d-1, 0}.
            \end{align*}
            Moreover, these 7 cones are contained in $\Sigma(\Delta\times\Delta', (u, l\kappa), Z'')$.
            \item[(d)]
            A scheme $\overline{H(a)^{X(\Sigma_l)}}\cap O_{\sigma_1}$ is birational to a general hypersurface of degree $d$ in $\PP^{2n-5}_k$ for a general $(a_\alpha)_{\alpha\in S}\in k^{|S|}$. 
            \item[(e)] For a general $(a_\alpha)_{\alpha\in S}\in k^{|S|}$, both $\overline{H(a)^{X(\Sigma_l)}}\cap O_{\tau_1}$ and $\overline{H(a)^{X(\Sigma_l)}}\cap O_{\tau_2}$ are irreducible and rational. 
            \end{enumerate}
        \end{proposition}
        \begin{proof}
            We prove the proposition from (a) to (e) in order:
            \begin{enumerate}
                \item[(a)] 
                It is enough to show that $\Sigma((u, \kappa))$ is strongly convex. 
                We can compute $\Sigma((u, \kappa))$ by Lemma \ref{lem: compute normal fan}, and from now on, we use the notation in Lemma \ref{lem: compute normal fan}. 

                It is enough to show that $D((u, \kappa))$ is a full cone in $((M\oplus M^\dagger\oplus\ZZ^\vee)\oplus\ZZ)_\RR$. 
                Let $L$ denote the linear subspaces of $((M\oplus M^\dagger\oplus\ZZ^\vee)\oplus\ZZ)_\RR$ generated by $D((u, \kappa))$. 
                Let $\alpha_0\in S$ denote an element which $(\alpha_0)_{(0, 1)} = d$ holds. 
                Then $(\varpi_{\alpha_0}+\kappa(\alpha_0)\delta, 1)=((2d -1)\delta, 1)\in D((u, \kappa))$. 
                Similarly, for any $0\leq j\leq n - 3$, we can check that $(d\eta_{j}, 1)\in D((u, \kappa))$, $(d\eta_{-1}+d\eta_0, 1)\in D((u, \kappa))$. 
                Let $\beta_0\in S$ denote which $(\beta_0)_{(0, 1)} = 1$ and $(\beta_0)_{(1, 2)} = d-1$ hold. 
                Then $(\varpi_{\beta_0}+\kappa(\beta_0)\delta, 1)=((d-1)\eta_0+\delta, 1)\in D((u, \kappa))$. 
                
                Thus, $(\eta_0-\delta, 0) = (d\eta_0, 1) - ((d-1)\eta_0 + \delta, 1)\in L$, and $(d-1)(\eta_0-2\delta, 0) = ((d-1)\eta_0+\delta, 1)-((2d-1)\delta, 1)\in L$.  
                In particular, $(\eta_0, 0), (\delta, 0)\in L$. 
                Moreover, we can check that $(0, 1), (\eta_{-1}, 0), \ldots, (\eta_{n-3}, 0)\in L$, $\{(\omega_{i, i} - \omega_{0, 0})\}_{1\leq i\leq n-3}\subset L$, and $\{(\omega_{i, j} - \omega_{0, 0})\}_{0\leq i < j\leq n-3}\subset L$. 
                Therefore, $D((u, \kappa))$ is a full cone. 
                \item[(b)] We may assume that $l = 1$ by Proposition \ref{prop: modify by multiplication}(b). 
                Let $\tau\in (\Sigma_{1, +})_{\bdd}$. 
                By Lemma \ref{lem: compute normal fan}, there exists $c\in\mathcal{C}$ and $\tau'\prec C((u, \kappa), \pi^{-1}_\RR(\sigma_c)\times[0, \infty))$ such that $\pr_{1, \RR}(\tau') = \tau$ and $(0, 1)\notin\tau'$. 
                Let $\sigma'_c$ denote $\pi^{-1}_\RR(\sigma_c)\times[0, \infty)$ and let $\gamma\prec D((u, \kappa),\sigma'_c)$ be the dual face of $\tau'$, i.e. $\gamma = (\tau')^\perp\cap D((u, \kappa), \sigma'_c)$. 
                In particular, $\tau' = \gamma^\perp\cap C((u, \kappa), \sigma'_c)$. 

            Now, we will classify all such $\gamma$ (or such $\tau'$) from Step  1 to Step  9. 

            Step  1. In this step, we show that $(\delta, 0)\notin\gamma$ and there exists $\alpha\in S$, such $(\varpi_\alpha+\kappa(\alpha)\delta, 1)\in \gamma$. 
            Because $\tau\not\subset N'_\RR\times\{0_{(\ZZ)_\RR}\}$, $(\delta, 0)\notin \gamma$. 
            In particular, for any $\omega\in\sigma^\vee_c$ and $r>0$, $(\omega+r\delta, 0)\notin\gamma$ because $(\omega, 0), (\delta, 0)\in D((u, \kappa), \sigma'_c)$ and $\gamma\prec D((u, \kappa), \sigma'_c)$. 
            Because $\gamma\prec D((u, \kappa), \sigma'_c)$, $\gamma$ is generated by some generaters of $D((u, \kappa), \sigma'_c)$. 
            If $\gamma$ does not contain $(\varpi_\alpha+\kappa(\alpha)\delta, 1)$ for any $\alpha\in S$, then $\gamma\subset (M'\oplus\ZZ^\vee)_\RR\times\{0\}$. 
            This inclusion indicates $(0, 1)\in\tau'$, and it is a contradiction. 

            Step  2. Let $(i, j)\in J_2$. 
            In this step, we show that there exists $\omega_0\in \sigma^\vee_c\cap M$ such that either following equation holds in $M'$: 
            \begin{align*}
                 \varpi_{0, 1} + \varpi_{i, j} &= \varpi_{0, i} + \varpi_{1, j} + \omega_0,\\
                 \varpi_{0, 1} + \varpi_{i, j} &= \varpi_{0, j} + \varpi_{1, i} + \omega_0.
            \end{align*}
            We write down $c$ explicitly as $c = (V_1, V_2, \ldots, V_s)$. 
            Let $k_1, k_2,$ and $k_3$ be minimal integers such that $L_{i-2, i-2}\in V_{k_1}, L_{j-2, j-2}\in V_{k_2}, $ and $L_{i-2, j-2}\in V_{k_3}$. 
            Because $L_{i-2, j-2} = L_{i-2, i-2} - L_{j-2, j-2}$, $k_3\leq \max\{k_1, k_2\}$. 
            If $k_3\leq k_1$, then $\omega_{i-2, j-2} - \omega_{i-2, i-2}\in \sigma^\vee_c$. 
            Similary, if $k_3\leq k_2$, then $\omega_{i-2, j-2} - \omega_{j-2, j-2}\in \sigma^\vee_c$. 
            Therefore, by the definition of $\varpi_{\cdot, \cdot}$, there exists $\omega_0\in\sigma^\vee_c\cap M$ such that the either equation above holds. 

            Step  3. Let $d_0, d_1,$ and $d_2$ be integers and let $\alpha\in S_{d_0, d_1, d_2}$. 
            We assume that $d_0d_2 > 0$. 
            In this step, We show that $(\varpi_{\alpha} + \kappa(\alpha)\delta, 1)\notin \gamma$ holds. 
            Indeed, by the assumption of $d_0$ and $d_2$ and Step  2, there exists $\beta\in S_{(d_0 - 1, d_1 + 2, d_0 -1)}$, a positive integer $m$ and $\omega_0\in\sigma^\vee_c$ such that $(\varpi_\alpha + \kappa(\alpha)\delta, 1) = (\varpi_\beta + \kappa(\beta)\delta, 1) + (m\delta, 0) + (\omega_0, 0)$. 
            By the definition of $D((u, \kappa), \sigma'_c)$, $(\varpi_\beta + \kappa(\beta)\delta, 1), (m\delta, 0), (\omega_0, 0)\in D((u, \kappa), \sigma'_c)$. 
            Because $\gamma\prec D((u, \kappa), \sigma'_c)$, $(\delta, 0)\in\gamma$. 
            However, it is a contradiction to Step  1. 

            Step  4. Let $d_0, d_2$ be positive integers, let $\alpha\in S_{(d_0, d-d_0, 0)}$ and let $\alpha'\in S_{(0, d-d_2, d_2)}$. 
            In this step, we show that $\{(\varpi_\alpha + \kappa(\alpha)\delta, 1), (\varpi_{\alpha'} + \kappa(\alpha')\delta, 1)\}\not\subset \gamma$. 
            Indeed, there exist $\beta\in S_{(d_0 - 1, d - d_0 + 1, 0)}$, $\beta'\in S_{(0, d - d_2 + 1, d_2 - 1)}$, $\omega\in\sigma^\vee_c$, and $m\in\ZZ_{> 0}$ such that $\varpi_{\alpha} + \varpi_{\alpha'} = \varpi_{\beta} + \varpi_{\beta'} + \omega$ and $\kappa(\alpha) + \kappa(\alpha') = \kappa(\beta) + \kappa(\beta') + m$ holds by Step  2. 
            Thus, if $\{(\varpi_\alpha + \kappa(\alpha)\delta, 1), (\varpi_{\alpha'} + \kappa(\alpha')\delta, 1)\}\subset \gamma$, then $(m\delta, 0)\in \gamma$. 
            It is a contradiction to Step  1. 

            Step  5. Let $i > 1$ be an integer, let $\alpha\in S_{0, d, 0}$ and let $\alpha'\in S_{i, d - i, 0}$. 
            In this step, we show that $\{(\varpi_\alpha + \kappa(\alpha)\delta, 1), (\varpi_{\alpha'} + \kappa(\alpha')\delta, 1)\}\not\subset \gamma$. 
            Inded, there exist $\beta\in S_{(1, d - 1, 0)}$, $\beta'\in S_{(i - 1, d - i + 1, 0)}$, such that $\varpi_{\alpha} + \varpi_{\alpha'} = \varpi_{\beta} + \varpi_{\beta'}$ holds. 
            We remark that $\kappa(\alpha) + \kappa(\alpha') =\kappa(\beta) + \kappa(\beta') + 1$ holds. 
            Thus, if $\{(\varpi_\alpha + \kappa(\alpha)\delta, 1), (\varpi_{\alpha'} + \kappa(\alpha')\delta, 1)\}\subset \gamma$, we have that $(\delta, 0)\in \gamma$. 
            It is a contradiction to Step  1. 

            Step  6. Let $i > 1$ be an integer, let $\alpha\in S_{0, d, 0}$ and let $\alpha'\in S_{0, d - i, i}$. 
            In this step, we show that $\{(\varpi_\alpha + \kappa(\alpha)\delta, 1), (\varpi_{\alpha'} + \kappa(\alpha')\delta, 1)\}\not\subset \gamma$. 
            Indeed, we can check it as Step  5. 

            Step  7. Let $\tau'_0, \tau'_1, \tau'_2$, and $\tau'_3$ be cones in $((N'\oplus\ZZ)\oplus\ZZ)_\RR$ as follows: 
            \begin{align*}
                \tau'_0 &= \RR_{\geq 0}(e_t -2\sum_{0\leq j\leq n-3}e^\dagger_j, 2d + 1),\\
                \tau'_1 &= \RR_{\geq 0}(e_t -\sum_{0\leq j\leq n-3}e^\dagger_j, d),\\
                \tau'_2 &= \RR_{\geq 0}(e_t +\sum_{0\leq j\leq n-3}e^\dagger_j, -d),\\
                \tau'_3 &= \RR_{\geq 0}(e_t +2\sum_{0\leq j\leq n-3}e^\dagger_j, -2d + 1).\\
            \end{align*}
            In this step, we show that the cones above are rays of $C((u, \kappa), \sigma'_c)$. 
            Indeed, we can check that the cones above are contained in $C((u, \kappa), \sigma'_c)$. 
            Thus, It is enough to show that $(\tau'_j)^\perp \cap D((u, \kappa), \sigma'_c)$ is a facet of $D((u, \kappa), \sigma'_c)$.
            Let $\gamma_j$ denote $(\tau'_j)^\perp\cap D((u, \kappa), \sigma'_c)$ for $0\leq j\leq 3$. 
            We can check that $\{(\omega, 0)|\omega\in M_\RR\}\subset\langle\gamma_j\rangle$ for any $0\leq j\leq 3$ because of the strong convexity of $\sigma_c$ and the definition of $\tau'_j$. 
            For $0\leq j\leq 3$, let $S(j)$ denote $\{\alpha\in S\mid (\varpi_\alpha+\kappa(\alpha)\delta, 1)\in\gamma_j\}$. 
            We can check the following equations:
                \begin{align*}
                    S(0) &= \bigcup_{1\leq i\leq d}S_{0, d-i, i},\\
                    S(1) &= S_{0, d-1, 1} \cup S_{0, d, 0},\\
                    S(2) &= S_{0, d, 0} \cup S_{1, d-1, 0},\\
                    S(3) &= \bigcup_{1\leq i\leq d}S_{i, d-i, 0}.\\
                \end{align*}
            Thus, as the proof of (a), we can check that $\langle\gamma_j\rangle$ is a codimension $1$ linear subspace of $((M'\oplus\ZZ^\vee)\oplus\ZZ)_\RR$. 

            Step  8. We assume that $\tau'$ is a ray of $C((u, \kappa), \sigma'_c)$. 
            In this step, we show that there exists $0\leq j\leq 3$ such that $\tau'= \tau'_j$. 
            Let $S^*$ denote $\{\alpha\in S\mid (\varpi_\alpha+\kappa(\alpha)\delta, 1)\in\gamma\}$. 
            Then there exists $0\leq j\leq 3$ such that $S^*\subset S(j)$ by the argument from Step  3 to Step  7. 
            Moreover, we can check that $\{(\omega, 0)\mid \omega\in\sigma^\vee_c\}\subset \gamma_j$ for any $0\leq j\leq 3$. 
            Furthermore, $(\omega+r\delta, 0)\notin\gamma$ for any $\omega\in \sigma^\vee_c$ and $r > 0$ by Step  1. 
            Thus, $\gamma\subset\gamma_j$. 
            Because $\tau'$ is a ray of $C((u, \kappa), \sigma'_c)$, $\gamma$ is a facet of $D((u, \kappa), \sigma^\vee_c)$, and hence, $\gamma = \gamma_j$. 
            In particular, $\tau' = \tau'_j$. 

            Step  9. In this step, we classify $\tau'$. 
            By Step  8, all rays of $\tau'$ are conteind in $\{\tau'_0, \tau'_1, \tau'_2, \tau'_3\}$. 
            We have already known $S(j)$ for each $0\leq j\leq 3$. 
            Thus, by Step  1, there are 7 possible forms for $\tau'$, as follows: $\tau'_0, \tau'_1, \tau'_2, \tau'_3, \tau'_{0, 1} = \tau'_0 + \tau'_1, \tau'_{1, 2} = \tau'_1 + \tau'_2, \tau'_{2, 3} = \tau'_2 + \tau'_3$. 
            Let $\gamma_{i, j}$ denote $(\tau'_{i, j})^\perp\cap D((u, \kappa), \sigma'_c)$. 
            We can check that all $\langle\gamma_{i, j}\rangle$ is a codimension $2$ subspace of $((M'\oplus\ZZ^\vee)\oplus\ZZ)_\RR$. 
            Thus, all $\tau'_{i, j}$ are faces of $C((u, \kappa), \sigma'_c)$. 
            
            \item[(c)] By Step  7 in the proof of (b), we can check the first statement. 

            We will show the second statement. 
            Let $q$ denote the first projection $N\oplus N^\dagger\oplus\ZZ\rightarrow N$ and let $p\colon N\oplus N^\dagger\oplus\ZZ\rightarrow N^\dagger\oplus\ZZ$ be the second projection. 
            Let $\gamma$ be one of these 7 cones. 
            We can check that $q_\RR(\gamma) = \{0\}$. 
            Let $\gamma'$ denote a strongly convex cone $p_\RR(\gamma)$. 
            Then we can identify with $\overline{Z''^{X(\Sigma_l)}}\cap O_{\gamma}$ and $Z\times O_{\gamma'}$ by Proposition \ref{prop: easy-computation}. 
            In particular, $\overline{Z''^{X(\Sigma_l)}}\cap O_{\gamma}$ is irreducible. 
            Let $\iota_0$ denote the closed immersion $Z\times O_{\gamma'}\rightarrow T_N\times O_{\gamma'}$. 
            We can check that there exists $\omega\in M\oplus(\gamma'^\perp\cap(M^\dagger\oplus\ZZ^\vee))$ such that $\{x_{-1}\iota^*_0(\chi^\omega), \iota^*_0(\chi^\omega)\}\subset V(\gamma, 1, (u, l\kappa), \omega_\gamma)$, and hence, $d_{Z''}(\gamma, 1, (u, l\kappa)) \geq 2$. 
            \item[(d)] 
            We keep the notation in the proof of (c).
            Let $\sigma'_1$ denote $p_\RR(\sigma_1)$. 
            Thus, we can identify with $\overline{Z''^{X(\Sigma_l)}}\cap O_{\sigma_1}$ and $Z\times O_{\sigma'_1}$ by Proposition \ref{prop: easy-computation}. 
            On this identification, we will compute $(u, {l\kappa})^{\sigma_1}$ (See Definition \ref{def: orbit-polytope}).  
            Then we can check that $(\sigma'_1)^{\perp}\cap (M^\dagger\oplus\ZZ^\vee)$ is generated by $\eta_{-1}$ and $\{\eta_j - \eta_0\}_{1\leq j\leq n-3}$. 
            Let $\alpha_1\in S$ denote an element which $(\alpha_1)_{(1, 2)} = d$ holds. 
            Then $\alpha_1\in S^{\sigma_1}$. 
            Thus, $(u, l\kappa)(\alpha_1) = d\eta_0$ satisfies two conditions in Proposition \ref{prop: normal fan I}(b). 
            For $(i, j)\in J_1$, let $\varpi'_{i, j}$ denote the following elements in $N\oplus ((\sigma'_1)^{\perp}\cap (M^\dagger\oplus\ZZ^\vee))$:
            \[
                \varpi'_{i, j} = \left\{
                    \begin{array}{ll}
                        \eta_{-1}, & (i, j) = (0, 2),\\
                        \eta_{-1} + (\eta_{j - 2}- \eta_0) + (\omega_{j - 2, j - 2} - \omega_{0, 0}),& i = 0, j > 2,\\
                        0, & (i, j) = (1, 2),\\
                        (\eta_{j - 2} -\eta_0), & i = 1, j>2.
                    \end{array}
                \right.
            \]
            Let $u'$ denote the map $S^{\sigma_1}\rightarrow N\oplus ((\sigma'_1)^{\perp}\cap (M^\dagger\oplus\ZZ^\vee))$ such that it holds that $u'(\alpha) = \sum_{(i, j)\in J_1}{\alpha_{i, j}\varpi'_{i, j}}$
            $ = u(\alpha) - u(\alpha_1)$ for any $\alpha\in S^{\sigma_1}$ and let $\iota_1$ denote the closed immersion $Z\times O_{\sigma'_1}\rightarrow T_N\times O_{\sigma'_1}$. 
            We can check $(u, {l\kappa})^{\sigma_1}\sim u'$, and we identify with $u'$ and $(u, {l\kappa})^{\sigma_1}$. 
            Let $F'(a)$ denote $\sum_{\alpha\in S^{\sigma_1}}a_\alpha\iota_1^*(\chi^{u'(\alpha)})\in k[Z]\otimes k[(\sigma'_1)^\perp\cap (M^\dagger\oplus\ZZ^\vee)]$ and let $H'(a)$ denote the hypersurface in $Z\times O_{\sigma'_1}$ defined by $F'(a)$. 
            Then $\overline{H(a)^{X(\Sigma_l)}}\cap O_{\sigma_1} \cong H'(a)$ for a genera $(a_\alpha)_{\alpha\in S}\in k^{|S|}$ by Proposition \ref{prop: generic condition of flatness}(c) and (d).  
            
            On the other hand, $\{\iota^*_1(\chi^{\varpi'_{i, j}})\}_{(i, j)\in J_1\setminus\{(1, 2)\}}$ is a transcendence basis and a generator of the fraction field of $k[Z]\otimes_k k[(\sigma'_1)^{\perp}\cap (M^\dagger\oplus\ZZ^\vee)]$. 
            Moreover, $\Spec(k[Z]\otimes_k k[(\sigma'_1)^{\perp}\cap (M^\dagger\oplus\ZZ^\vee)])\rightarrow \Spec(k[\{\iota^*_1(\chi^{\varpi'_{i, j}})\}_{(i, j)\in J_1}]) \cong {\A^{|J_1| - 1}_k}$ is an open immersion. 
            Thus, by taking a more general element, $\overline{H(a)^{X(\Sigma_l)}}\cap O_{\sigma_1}$ is birational to a general hypersurface of degree $d$ in $\PP^{|J_1| - 1}_k$. 
            We remark that $|J_1| = 2n - 4$. 
            \item[(e)] We keep the notation in the proof of (c). 
            By the definition of $\tau_1$, $q_\RR(\tau_1) = \{0\}$. 
            Let $\tau'_1$ denote $p_\RR(\tau_1)$. 
            Thus, we can identify with $\overline{Z''^{X(\Sigma_l)}}\cap O_{\tau_1}$ and $Z\times O_{\tau'_1}$.   
            On this identification, we compute $(u, {l\kappa})^{\tau_1}$. 
            Then we can check that $(\tau'_1)^{\perp}\cap (M^\dagger\oplus\ZZ^\vee)$ is generated by $\eta_{-1}$, $\eta_0 + l\delta$, and $\{\eta_j - \eta_0\}_{1\leq j\leq n-3}$ and $\alpha_1\in S^{\tau_1}$. 
            Thus, $(u, l\kappa)(\alpha_1)$ satisfies two conditions in Proposition \ref{prop: normal fan I}(b). 
            For $(i, j)\in J_1\cup J_2$, let $\varpi''_{i, j}$ denote the following elements in $N\oplus ((\tau'_1)^{\perp}\cap (M^\dagger\oplus\ZZ^\vee))$:
            \[
                \varpi''_{i, j} = \left\{
                    \begin{array}{ll}
                        \eta_{-1}, & i = 0, j = 2,\\
                        \eta_{-1} + (\eta_{j - 2} - \eta_0) + (\omega_{j - 2, j - 2} - \omega_{0, 0}),& i = 0, j > 2,\\
                        0, & i = 1, j = 2,\\
                        (\eta_{j - 2} - \eta_{0}), & i = 1,\\
                        \eta_{-1} + (\eta_{j-2} - \eta_0) + (\omega_{0, j-2} - \omega_{0, 0}) + (\eta_0 + l\delta), & i = 2,\\
                        \eta_{-1} + (\eta_{j-2} - \eta_0) + (\eta_{i-2} - \eta_0) + (\omega_{i-2, j-2} - \omega_{0, 0}) + (\eta_0 + l\delta), & i > 2.
                    \end{array}
                \right.
            \]
            Let $u''$ denote the map $S^{\tau_1}\rightarrow N\oplus ((\tau'_1)^{\perp}\cap (M^\dagger\oplus\ZZ^\vee))$ such that for any $\alpha\in S^{\tau_1}$,  $u''(\alpha) = \sum_{(i, j)\in J_1\cup J_2}{\alpha_{i, j}\varpi''_{i, j}} = (u, l\kappa)(\alpha) - (u, l\kappa)(\alpha_1)$ , and let $\iota_2$ denote the closed immersion $Z\times O_{\tau'_1}\rightarrow T_N\times O_{\tau'_1}$. 
            Then we can check $(u, {l\kappa})^{\tau_1}\sim u''$, and we identify with $u''$ and $(u, {l\kappa})^{\tau_1}$. 
            Let $F''(a)$ denote $\sum_{\alpha\in S^{\tau_1}}a_\alpha\iota_2^*(\chi^{u''(\alpha)})\in k[Z]\otimes k[(\tau'_1)^\perp\cap (M^\dagger\oplus\ZZ^\vee)]$, and let $H''(a)$ denote the hypersurface in $Z\times O_{\tau'_1}$ defined by $F''(a)$. 
            Then $\overline{H(a)^{X(\Sigma_l)}}\cap O_{\tau_1} \cong H''(a)$ for a genera $(a_\alpha)_{\alpha\in S}\in k^{|S|}$ by Proposition \ref{prop: generic condition of flatness}(c) and (d).  
            
            We remark that $k[(\tau'_1)^{\perp}\cap (M^\dagger\oplus\ZZ^\vee)] \cong k[(\sigma'_1)^{\perp}\cap (M^\dagger\oplus\ZZ^\vee)]\otimes_k k[(x_0t^l)^{\pm}]$. 
            Moreover, we can check that the following three conditions hold:                 
            \begin{itemize}
                \item[(1)] The equation $S^{\tau_1} = S^{\sigma_0}\coprod S^{\sigma_1}$ by (c). 
                \item[(2)] For any $\alpha\in S^{\sigma_0}$, there exists $v(\alpha)\in (k[Z]\otimes_k k[(\sigma'_1)^{\perp}\cap (M^\dagger\oplus\ZZ^\vee)])^*$ such that $\iota^*_2(\chi^{u''(\alpha)}) = (x_0t^l) v(\alpha)$. 
                \item[(3)] For any $\beta\in S^{\sigma_1}$, there exists $w(\beta)\in (k[Z]\otimes_k k[(\sigma'_1)^{\perp}\cap (M^\dagger\oplus\ZZ^\vee)])^*$ such that $\iota^*_2(\chi^{u''(\beta)}) = w(\beta)$.
            \end{itemize}
            Thus, by taking a more general element, $\overline{H(a)^{X(\Sigma_l)}}\cap O_{\tau_1}$ is irreducible by Lemma \ref{lem: irreducible polynomial} and Lemma \ref{lem: no common divisor}. 
            Moreover, this is rational because $Z$ is rational.                  
            Similarly, we can check that the statement holds for $\overline{H(a)^{X(\Sigma_l)}}\cap O_{\tau_2}$. 
        \end{enumerate}       
        \end{proof}
        The following theorem is the main theorem of this paper. 
        
        \begin{theorem}\label{thm: d in Grassmannian}
            If a very general hypersurface of degree $d$ in $\PP_\C^{2n-5}$ is not stably rational, then a very general hypersurface of degree $d$ in $\Gr_\C(2, n)$ is not stably rational.
        \end{theorem}
        \begin{proof}
             We use the notation in Proposition \ref{prop: Bertini-stably rational}, Proposition \ref{prop:grassman-mock}, and Proposition \ref{prop: computation result}. 
            Let $k$ denote $\C$, let $X$ denote $\Gr_\C(2, n)$, let $\Sigma'_l$ denote $\Sigma(\Delta\times\Delta_!, (u, l\kappa), Z'')$, and let $\Sigma'_{l, +}$ denote $\{\tau\in \Sigma'_l\mid\tau\subset N'_\RR\times\RR_{\geq 0}\}$ for $l\in\ZZ_{>0}$. 
            We define the open subsets $U_0, U_1, U_2$ and $U_3$ of $\A^{|S|}_\C$ as follows: 
            \begin{enumerate} 
                \item[(0)] 
                By Proposition \ref{prop: modify by multiplication}(c), there exists a positive integer $l$ and a refinement $\Delta'$ of $\Sigma'_l$ such that $\Delta'_+ = \{\sigma\in\Delta'\mid \sigma\subset N'_\RR\times \RR_{\geq 0}\}$ is generically unimodular, specifically reduced, and compactly arranged. 
                Let $\Delta''$ denote $\Sigma'_{l, +}$. 
                By Proposition \ref{prop: computation result}(a), $\Delta''$ is strongly convex. 
                                
                Let $U_0$ denote a dense open subset of $\A^{|S|}_\C$ in Proposition \ref{prop: general condition of schon} for $(\Delta\times\Delta_!, (u, l\kappa), Z'')$.
                Let $a\in U_0(\C)$. 
                By Proposition \ref{prop: general condition of schon}, $\overline{H(a)^{X(\Sigma'_l)}}$ is proper over $k$. 
                In particular, the composition $\overline{H(a)^{X(\Sigma'_l)}}\rightarrow X(\Sigma'_l)\rightarrow X(\Delta_!) = \PP^1_k$ is proper, and hence, the morphism $\overline{H(a)^{X(\Delta'')}}\rightarrow \A^1_k$ is proper. 
                Similarly, $\overline{H(a)^{X(\Delta'_+)}}\rightarrow \A^1_k$ is also proper because $\Delta'_+$ is a refinement of $\Delta''$. 
                Moreover, the multiplication morphism $T_{N\oplus\ZZ}\times \overline{H(a)^{X(\Delta'_+)}}$ is smooth by Proposition \ref{prop: surjection}(c). 
                For the notation in Proposition \ref{prop: computation result} (b), let $\{E^{(j)}_{\tau_i}\}_{1\leq j\leq r_{\tau_i}}$ be all irreducible components of $\overline{H(a)^{X(\Delta'')}}\cap O_{\tau_i}$ for $0\leq i\leq 3$ and let $\{E^{(j)}_{\sigma_i}\}_{1\leq j\leq r_{\sigma_i}}$ be all irreducible components of $\overline{H(a)^{X(\Delta'')}}\cap O_{\sigma_i}$ for $0\leq i\leq 2$. 
                Note that each $r_{\tau_i}$ and $r_{\sigma_i}$ is depend on the choice of $a\in U_0(\C)$. 
                \item[(1)] Let $\lambda$ denote the automorphism of $\A^{|S|}_\Kt$ defined as $\lambda((c_\alpha)_{\alpha\in S}) = (t^{-l\kappa(\alpha)}c_\alpha)_{\alpha\in S}$ for $(c_\alpha)_{\alpha\in S}\in \Kt^{|S|}$. 
                Let $h$ denote a canonical map $\A^{|S|}_\Kt\rightarrow \A^{|S|}_\C$. 
                By Proposition \ref{prop:grassman-mock}, 
                there exists a dense open subset $V$ of $\A^{|S|}_\Kt$ such that $X_D$ and $Y_D$ are irreducible and birational for any $D\in V(\Kt)$. 
                Let $U_1$ denote an open subset $h\circ\lambda(h^{-1}(U_{X, d})\cap V)$ of $\A^{|S|}_\C$. 
                For $(a_\alpha)_{\alpha\in S}\in U_1(\C)$, let $D_0$ denote $(a_\alpha t^{l\kappa(\alpha)})_{\alpha\in S}\in \Kt^{|S|}$. 
                Then $D_0\in (U_{X, d})_\Kt(\Kt)$ and $D_0\in V(\Kt)$. 
                Moreover, we can check that $H(a)\times_{\mathbb{G}^1_{m, k}}\Spec(\Kt) \cong Y_{D_0}$. 
                Thus, $X_{D_0}$ and $Y_{D_0} \cong H(a)\times_{\mathbb{G}^1_{m, k}}\Spec(\Kt)$ are birational. 
                In particular, $H(a)\times_{\mathbb{G}^1_{m, k}}\Spec(\Kt)$ is irreducible.
                \item[(2)] Let $U_2$ denote an open subset of $\A^{|S|}_\C$ in Proposition \ref{prop: computation result}(d). 
                In particular, $r_{\sigma_1} = 1$ and $E^{(1)}_{\sigma_1}$ is birational to a general hypersurface of degree $d$ in $\PP^{2n-5}_{\C}$. 
                \item[(3)]
                Let $U_3$ denote an open subset of $\A^{|S|}_\C$ in Proposition \ref{prop: computation result}(e). 
                In particular, $r_{\tau_1} = r_{\tau_2} = 1$ and $\{E^{(1)}_{\tau_1}\}_{\mathrm{sb}} = \{E^{(1)}_{\tau_2}\}_{\mathrm{sb}} = \{\Spec(\C)\}_{\mathrm{sb}}$. 
            \end{enumerate}
            Let $U$ denote $U_0\cap U_1\cap U_2 \cap U_3$. 
            Let $(a_\alpha)_{\alpha\in S}\in U(\C)$. 
            Then by Proposition \ref{prop: model II}(e) and Proposition \ref{prop: computation result}(b) and (c), the following equation holds: 
            \begin{align*}
                \VOL(H(a)\times_{\mathbb{G}^1_{m, k}}\Spec(\Kt)) &= \sum_{i = 0, 3}\biggl(\sum_{1\leq j\leq r_{\tau_i}} \bigl\{E^{(j)}_{\tau_i}\bigr\}_{\mathrm{sb}}\biggr)-\sum_{i = 0, 2}\biggl(\sum_{1\leq j\leq r_{\sigma_i}} \bigl\{E^{(j)}_{\sigma_i}\bigr\}_{\mathrm{sb}}\biggr)\\
                &+ 2\{\Spec(\C)\}_{\mathrm{sb}}-\{E^{(1)}_{\sigma_1}\}_{\mathrm{sb}}.
            \end{align*}

            On the other hand, let $(b_\alpha)_{\alpha\in S^{\sigma_1}}\in \C^{|S^{\sigma_1}|}$ and let $(a'_\alpha)_{\alpha\in S}$ denote the following element in $\C^{|S|}$: 
            \[
                a'_\alpha = 
                \left\{
                    \begin{array}{ll}
                        b_\alpha, & 
                        \alpha\in S^{\sigma_1},\\
                        a_\alpha, & 
                        \alpha\notin S^{\sigma_1}.\\
                    \end{array}
                \right.
            \]
            Because $U\cap ((a_\alpha)_{\alpha\in S\setminus S^{\sigma_1}}\times \A^{|S^{\sigma_1}|}_\C)$ is a non-empty open subset of $(a_\alpha)_{\alpha\in S\setminus S^{\sigma_1}}\times \A^{|S^{\sigma_1}|}_\C$, $(a'_\alpha)_{\alpha\in S} \in U$ for a genera $(b_\alpha)_{\alpha\in S^{\sigma_1}}\in \C^{|S^{\sigma_1}|}$. 
            We assume that $(a'_\alpha)_{\alpha\in S} \in U(\C)$. 
            Let $\{F^{(j)}_{\tau_i}\}_{1\leq j\leq r'_{\tau_i}}$ be all irreducible components of $\overline{H(a')^{X(\Delta'')}}\cap O_{\tau_i}$ for $0\leq i\leq 3$, and let $\{F^{(j)}_{\sigma_i}\}_{1\leq j\leq r'_{\sigma_i}}$ be all irreducible components of $\overline{H(a')^{X(\Delta'')}}\cap O_{\sigma_i}$ for $0\leq i\leq 2$. 
            Let $\{f^{\tau_i}\}_{0\leq i\leq 3}$, $\{f^{\sigma_i}\}_{0\leq i\leq 2}$, $\{f'^{\tau_i}\}_{0\leq i\leq 3}$, and $\{f'^{\sigma_i}\}_{0\leq i\leq 2}$ denote the following elements: 
            \begin{itemize}
                \item $f^{\tau_i} = \sum_{\alpha\in S^{\tau_i}}a_\alpha(\iota^{\tau_i})^*(\chi^{(u, l\kappa)^{\tau_i}(\alpha)}) \in k[\overline{Z''^{X(\Delta'')}}\cap O_{\tau_i}]$,
                \item $f^{\sigma_i} = \sum_{\alpha\in S^{\sigma_i}}a_\alpha(\iota^{\sigma_i})^*(\chi^{(u, l\kappa)^{\sigma_i}(\alpha)})\in k[\overline{Z''^{X(\Delta'')}}\cap O_{\sigma_i}]$,
                \item $f'^{\tau_i} = \sum_{\alpha\in S^{\tau_i}}a'_\alpha(\iota^{\tau_i})^*(\chi^{(u, l\kappa)^{\tau_i}(\alpha)})\in k[\overline{Z''^{X(\Delta'')}}\cap O_{\tau_i}]$,
                \item $f'^{\sigma_i} = \sum_{\alpha\in S^{\sigma_i}}a'_\alpha(\iota^{\sigma_i})^*(\chi^{(u, l\kappa)^{\sigma_i}(\alpha)})\in k[\overline{Z''^{X(\Delta'')}}\cap O_{\sigma_i}]$,
            \end{itemize}
            where $\iota^{\tau_i}$ and $\iota^{\sigma_i}$ are  closed immersions $\overline{Z''^{X(\Delta'')}}\cap O_{\sigma_i}\rightarrow O_{\sigma_i}$ and $\overline{Z''^{X(\Delta'')}}\cap O_{\tau_i}\rightarrow O_{\tau_i}$, respectively.
            
            By Proposition \ref{prop: computation result}(c), $S^{\sigma_1}\cap S^{\tau_i} = \emptyset$ for any $i\in\{0, 3\}$ and $S^{\sigma_1}\cap S^{\sigma_i} = \emptyset$ for any $i\in\{0, 2\}$. 
            This shows that $f^{\tau_i} = f'^{\tau_i}$ for any $i = 0, 3$ and $f^{\sigma_i} = f'^{\sigma_i}$  for any $i = 0, 2$. 
            Thus, by Proposition \ref{prop: generic condition of flatness}(c) and (d), $r_{\tau_i} = r'_{\tau_i}$ for $i = 0, 3$ and if it is necessary, we can replace the index such that $E^{(j)}_{\tau_i} = F^{(j)}_{\tau_i}$ for any $1\leq j\leq r_{\tau_i}$. 
            Similarly, $r_{\sigma_i} = r'_{\sigma_i}$ for $i = 0, 2$ and if it is necessary, we can replace the index such that $E^{(j)}_{\sigma_i} = F^{(j)}_{\sigma_i}$ for any $1\leq j\leq r_{\sigma_i}$. 
            Therefore, the following equation holds: 
            \[
                \VOL(H(a)\times_{\mathbb{G}^1_{m, k}}\Spec(\Kt)) - \VOL(H(a')\times_{\mathbb{G}^1_{m, k}}\Spec(\Kt)) = \{E^{(1)}_{\sigma_1}\}_{\mathrm{sb}} - \{F^{(1)}_{\sigma_1}\}_{\mathrm{sb}}.
            \]
            By the assumption, a very general hypersurface of degree $d$ in $\PP^{2n-5}_\C$ is not stably rational. 
            Thus, by \cite[Corollary.4.2]{NO22}, a very general hypersurface of degree $d$ in $\PP^{2n-5}_\C$ is not stably birational to $E^{(1)}_{\sigma_1}$. 
            Because $\C$ is uncountable, there exists $(a'_\alpha)_{\alpha\in S}\in U(\C)$ such that $\{E^{(1)}_{\sigma_1}\}_{\mathrm{sb}} \neq\{F^{(1)}_{\sigma_1}\}_{\mathrm{sb}}$. 
            Thus, for such $(a'_\alpha)_{\alpha\in S}$, either of $H(a)\times_{\mathbb{G}^1_{m, k}}\Spec(\Kt)$ or $H(a')\times_{\mathbb{G}^1_{m, k}}\Spec(\Kt)$is not stably rational. 
            Because $(a_\alpha)_{\alpha\in S}, (a'_\alpha)_{\alpha\in S}\in U_1(\C)$, a very general hypersurface of degree $d$ in $\Gr_\C(2, n)$ is not stably rational by Proposition \ref{prop: Bertini-stably rational}(b). 
        \end{proof}
        The following corollary holds by \cite[Corollary 1.2]{S19} immediately. 
        \begin{corollary}\label{cor: log bound}
            If $n\geq 5$ and $d \geq 3 + \log_2(n-3)$, then a very general hypersurface of degree $d$ in $\Gr_\C(2, n)$ is not stably rational. 
        \end{corollary}
        \begin{proof}
            By \cite[Corollary 1.2]{S19}, a very general hypersurface of degree $d$ in $\PP^{2n-5}_\C$ is not stably rational. 
            Thus, the statement follows by Theorem \ref{thm: d in Grassmannian}. 
        \end{proof}
\section{Appendix}
In this section, we prove the lemmas needed in this article. 
The lemmas in the Appendix are all fundamental and well-known to experts, but they are frequently used in the main article. 
Therefore, they are compiled here as an appendix. 
We use the notation in Section 2.  
\subsection{Lemmas related to toric varieties}
In this subsection, we correct for lemmas related to toric varieties and convex geometry. 
\begin{lemma}\label{lem: action}
  Let $N$ be a lattice of finite rank, and let $\Delta$ be a strongly convex rational polyhedral fan in $N_\RR$.
  Then the following statements hold.
  \begin{itemize}
    \item[(a)] Let $\sigma\in\Delta$ and $q_*\colon T_N\rightarrow T_{N/\langle\sigma\rangle\cap N}$ be the quotient morphism associated with $q\colon N\rightarrow N/\langle\sigma\rangle\cap N$.
      Then the following diagram is a Cartesian product:
      \begin{equation*}
        \begin{tikzcd}
          T_N\times \overline{O_\sigma}\ar[r, "q_*\times\id"]\ar[d]&T_{N/\langle\sigma\rangle\cap N}\times\overline{O_\sigma}\ar[r]& \overline{O_\sigma}\ar[d]\\
          T_N\times X(\Delta) \ar[rr] & & X(\Delta),
        \end{tikzcd}
      \end{equation*}
      where horizontal morphisms are action morphisms and vertical morphisms are closed immersions.
    \item[(b)] Let $\pi\colon N'\rightarrow N$ be a morphism of lattices of finite rank.
      Let $\Delta'$ be a strongly convex rational polyhedral fan in $N'_\RR$.
      We assume that $\pi$ is compatible with the fans $\Delta'$ and $\Delta$.
      Then the following diagram is a Cartesian product:
      \begin{equation*}
        \begin{tikzcd}
          T_{N'}\times X(\Delta')\ar[rr]\ar[d, "\id\times\pi_*"] & & X(\Delta')\ar[d, "\pi_*"]\\
          T_{N'}\times X(\Delta) \ar[r, "\pi_*\times\id"] & T_N\times X(\Delta)\ar[r] & X(\Delta),
        \end{tikzcd}
      \end{equation*}
      where horizontal morphisms are action morphisms and vertical morphisms are toric morphisms.
    \item[(c)] Let $\sigma\in\Delta$ and $N_0$ be a sublattice of $N$ such that $N\cong N_0\oplus (\langle\sigma\rangle\cap N)$, let $p$ be the quotient morphism $p\colon N\rightarrow N/N_0$, and let $\sigma_0$ denote $p_\RR(\sigma)\subset (N/N_0)_\RR$.
    We remark that $\sigma_0$ is a strongly convex rational polyhedral cone in $(N/N_0)_\RR$.
    Then the following diagram is a Cartesian product:
    \begin{equation*}
        \begin{tikzcd}
          T_{N}\times X(\sigma)\ar[rr]\ar[d, "p_*\times\id"] & & X(\sigma)\ar[d, "p_*"]\\
          T_{N/N_0}\times X(\sigma) \ar[r, "\id\times p_*"] & T_{N/N_0}\times X(\sigma_0)\ar[r] & X(\sigma_0).
        \end{tikzcd}
      \end{equation*}
  \end{itemize}
\end{lemma}
\begin{proof}
  We prove the statements from (a) to (c) in order. 
  \begin{itemize}
    \item[(a)] Let $\overline{m}$ denote the action morphism $T_{N/\langle\sigma\rangle\cap N}\times \overline{O_\sigma}\rightarrow \overline{O_\sigma}$, let $m$ denote the action morphism $T_N\times X(\Delta)\rightarrow X(\Delta)$, let $\theta$ denote $\overline{m}\circ (q_*\times \id_{\overline{O_\sigma}})\colon T_N\times\overline{O_\sigma}\rightarrow \overline{O_\sigma}$, and let $\iota$ denote the closed immersion $\overline{O_\sigma}\hookrightarrow X(\Delta)$. 
    Then we can show that $(\pr_1, \theta)\colon T_N\times \overline{O_\sigma}\rightarrow T_N\times \overline{O_\sigma}$ and $(\pr_1, m)\colon T_N\times X(\Delta)\rightarrow T_N\times X(\Delta)$ are isomorphic, and the two small squares in the following diagram are Cartesian products:
    \begin{equation*}
      \begin{tikzcd}
        T_N\times \overline{O_\sigma}\ar[r, "{(\pr_1, \theta)}"]\ar[d, "\id\times \iota"]&T_N\times\overline{O_\sigma}\ar[r, "\pr_2"]\ar[d, "\id\times\iota"]& \overline{O_\sigma}\ar[d, "\iota"]\\
        T_N\times X(\Delta)\ar[r, "{(\pr_1, m)}"] & T_N\times X(\Delta)\ar[r, "\pr_2"] & X(\Delta).
      \end{tikzcd}
    \end{equation*}
    Thus, the statement holds. 
    \item[(b)] Let $m$ denote the action morphism $T_N\times X(\Delta)\rightarrow X(\Delta)$, let $m'$ denote the action morphism $T_{N'}\times X(\Delta')\rightarrow X(\Delta')$, and let $\theta$ denote $m\circ (\pi_*\times \id_{X(\Delta)})\colon T_{N'}\times X(\Delta)\rightarrow X(\Delta)$. 
    Then we can show that $(\pr_1, \theta)\colon T_{N'}\times X(\Delta)\rightarrow T_{N'}\times X(\Delta)$ and $(\pr_1, m')\colon T_{N'}\times X(\Delta)\rightarrow T_{N'}\times X(\Delta)$ are isomorphic, and the two small squares in the following diagram are Cartesian products:
    \begin{equation*}
      \begin{tikzcd}
        T_{N'}\times X(\Delta')\ar[r, "{(\pr_1, m')}"]\ar[d, "\id\times \pi_*"]&T_{N'}\times X(\Delta')\ar[r, "\pr_2"]\ar[d, "\id\times\pi_*"]& X(\Delta')\ar[d, "\pi_*"]\\
        T_{N'}\times X(\Delta)\ar[r, "{(\pr_1, \theta)}"] & T_{N'}\times X(\Delta)\ar[r, "\pr_2"] & X(\Delta).
      \end{tikzcd}
    \end{equation*}
    Thus, the statement holds. 
    \item[(c)] Let $m$ denote the action morphism $T_N\times X(\sigma)\rightarrow X(\sigma)$, let $m_0$ denote the action morphism $T_{N/N'}\times X(\sigma_0)\rightarrow X(\sigma_0)$, let $q$ denote the quotient morphism $N\rightarrow N/\langle\sigma\rangle\cap N$, let $\delta$ denote $(p, q)_*\colon T_N\rightarrow T_{N/N_0}\times T_{N/\langle\sigma\rangle\cap N}$, and let $\epsilon$ denote $(p, q)_*\colon X(\sigma)\rightarrow X(\sigma_0)\times T_{N/\langle\sigma\rangle\cap N}$. 
    We remark that $\delta$ and $\epsilon$ are isomorphic. 
    Let $\psi$ denote the automorphism of $T_{N/\langle\sigma\rangle\cap N}\times T_{N/\langle\sigma\rangle\cap N}$ defined by a permutation, let $\alpha$ denote $(\delta^{-1}, \epsilon^{-1})\circ(\id_{T_{N/N_0}\times X(\sigma_0)}\times\psi)\circ(\delta\times\epsilon)\colon T_N\times X(\sigma)\rightarrow T_N\times X(\sigma)$, and let $\beta$ denote $(\delta^{-1}\times\id_{X(\sigma_0)})\circ(\id_{T_{N/N_0}}\times \epsilon)\colon T_{N/N_0}\times X(\sigma)\rightarrow T_N\times X(\sigma_0)$. 
    Then we can check that $\alpha$ and $\beta$ are isomorphic, $m\circ\alpha = m$, $(\id_{T_N}\times p_*)\circ\alpha = \beta\circ(p_*\times\id_{X(\sigma)})$, $\id_{T_{N/N_0}}\times p_* = (p_*\times\id_{X(\sigma_0)})\circ\beta$, and the two small squares in the following diagram are Cartesian products by (b):
    \begin{equation*}
      \begin{tikzcd}
        T_{N}\times X(\sigma)\ar[r, "\alpha"]\ar[d, "p_*\times \id"]&T_{N}\times X(\sigma)\ar[r, "m"]\ar[d, "\id\times p_*"]& X(\sigma)\ar[d, "p_*"]\\
        T_{N/N_0}\times X(\sigma)\ar[r, "\beta"] & T_{N}\times X(\sigma_0)\ar[r, "m_0\circ(p_*\times\id)"] & X(\sigma_0).
      \end{tikzcd}
    \end{equation*}
    Thus, the statement holds. 
  \end{itemize}
\end{proof}
  The following lemma is referenced in \cite[Corollary of Theorem 22.5]{Mat86}. 
  This criteria of flatness is crucial for this article. 
  \begin{lemma}\label{lem: criterion of flatness}\cite[Corollary of Theorem 22.5]{Mat86}
    Let $\varphi\colon X\rightarrow Y$ be a flat and of finite type morphism of Noetherian schemes, let $x\in X$, let $y\in Y$ denote $\varphi(x)$, let $f_1, f_2, \ldots, f_r\in\Gamma(X, \OO_X)$, and let $Z$ be a closed subscheme of $X$ defined by the ideal generated by $\{f_i\}_{1\leq i\leq r}$. 
    We assume that $x\in Z$. 
    Let $\overline{f_i}\in \Gamma(X_y, \OO_{X_y})$ denote the restriction of $f_i$ to $X_y$ for each $1\leq i\leq r$.
    If $\overline{f_1}, \ldots, \overline{f_r}$ is a regular sequence of $\OO_{X_y, x}$, then $f_1, \ldots, f_r$ is a regular sequence of $\OO_{X, x}$, and the restriction morphism $\varphi|_Z$ is flat at $x$.   
  \end{lemma}
  \begin{proof}
    We may assume that $X$ and $Y$ are affine schemes. 
    Then the statement holds by \cite[Corollary of Theorem 22.5]{Mat86}. 
  \end{proof}
  The following proposition shows that tropical compactification has a toroidal structure. 
  The same claim has already been stated in \cite[Theorem 1.4]{Tev}, but here we provide a different proof.
  \begin{lemma}\label{lem: smooth to smooth}
    We keep the notation in Lemma \ref{lem: action} (c). 
    Let $Y\subset X(\sigma)$ denote a closed subscheme of $X(\sigma)$ and  let $m_Y \colon T_N\times Y\rightarrow X(\sigma)$ be the multiplication morphism. 
    We assume $m_Y$ is a smooth morphism. 
    Then $p_*|_{Y}\colon Y\rightarrow X(\sigma_0)$ is smooth at any point $y\in Y\cap O_\sigma$. 
  \end{lemma}
  \begin{proof}
    We keep the notation in the proof of Lemma \ref{lem: action} (c).  
    By the Lemma \ref{lem: action} (c), there exists the following Cartesian product:
      \begin{equation*}
        \begin{tikzcd}
          T_{N}\times Y \ar[r]\ar[d, "p_*\times\id_Y"] &T_N\times X(\sigma)\ar[r, "m"]\ar[d, "p_*\times \id"] & X(\sigma)\ar[d, "p_*"]\\
          T_{N/N_0}\times Y \ar[r] & T_{N/N_0}\times X(\sigma)\ar[r, "m_0\circ (\id\times p_*)"] & X(\sigma_0).
        \end{tikzcd}
      \end{equation*}
      We remark that the composition of the upper morphisms is $m_Y$.  
      Let $u$ denote $m_0\circ(\id\times p_*)|_{T_{N/N_0}\times Y}$. 
      Thus, $u$ is smooth, and hence $u^{-1}(O_{\sigma_0}) \cong T_{N/N_{0}}\times (Y\cap O_\sigma)$ is smooth over $k$. 
      Let $y\in Y\cap O_\sigma$ be a closed point, let $e\in T_{N/N_0}$ be an identity element, and let $x\in T_{N/N_0}\times Y$ denote a closed point $(e, y)$. 
      Let $r$ denote $\dim(T_{N/N_0})$ and let $\{\chi_1, \ldots, \chi_r\}\subset \Gamma(T_{N/N_0}, \OO_{T_{N/N_0}})$ be coordinate functions at $e\in T_{N/N_0}$. 
      Then $\{e\}\times Y\subset T_{N/N_0}\times Y$ is a closed subscheme of $T_{N/N_0}\times Y$ defined by the ideal generated by $\{\chi_i\otimes 1\}_{1\leq i\leq r}$. 
      Let $f_i\in\Gamma(T_{N/N_0}\times (Y\cap O_\sigma), \OO_{T_{N/N_0}\times (Y\cap O_\sigma)})$ denote the restriction of $\chi_i\otimes 1$ to $u^{-1}(O_{\sigma_0})$. 
      By the definition of $\{\chi_i\}_{1\leq i\leq r}$, the isomorphism $u^{-1}(O_{\sigma_0})\cong T_{N/N_0}\times (Y\cap O_\sigma)$, and the smoothness of $(Y\cap O_\sigma)$, $\{f_i\}_{1\leq i\leq r}$ is a regular sequence of $\OO_{u^{-1}(O_{\sigma_0}), x}$. 
      Thus, $u|_{e\times Y}$ is flat at $x$ by Lemma \ref{lem: criterion of flatness}. 
      By the definition of $m_0$, we can check that $u|_{e\times Y} = p_*|_Y$ on the identification of $e\times Y$ and $Y$, in particular, $p_*|_Y$ is flat at $y$. 
      Thus, $p_*|_Y$ is flat at any point $y'\in Y\cap O_\sigma$. 
      We have seen that $(p_*|_Y)^{-1}(O_{\sigma_0})\cong Y\cap O_\sigma$ is smooth over $k$, and thus $p_*|_Y$ is smooth at any point $y'\in Y\cap O_\sigma$. 
  \end{proof}
  The following lemma is basic, but it is important for the criterion of sch\"{o}n compactifications.  
  \begin{lemma}\label{lem: action smooth}
    Let $k$ be a field, let $T$ denote an algebraic torus $\mathbb{G}_{m, k}^n$, let $Y\subset T$ be a closed subscheme of $T$, and let $m\colon T\times Y\rightarrow T$ denote the restriction morphism of the multiplication of $T$. 
    We assume that $Y$ is smooth over $k$. 
    Then $m$ is smooth. 
  \end{lemma}
  \begin{proof}
      Let $\mu$ denote the multiplication map $T\times T\rightarrow T$ and let $\lambda\colon T\times T\rightarrow T\times T$ denote the morphism defined by $(\mu, \pr_2)$. 
      We can check that that $\lambda$ is isomorphic, $\lambda(T\times Y) = T\times Y$, and $m = \pr_1|_{T\times Y} \circ \lambda|_{T\times Y}$. 
      By the assumption, $\pr_1|_{T\times Y}\colon T\times Y\rightarrow T$ is smooth, hence $m$ is also smooth. 
  \end{proof}
The following elementary lemma is needed in Lemma \ref{lem: relative toroidal}.
\begin{lemma}\label{lem: relative decomposition}
  Let $N$ be a lattice of finite rank and $N'$ be a sublattice of $N\oplus\ZZ$. 
  We assume that $(N\oplus\ZZ)/N'$ is torsion free and there exists $v\in N$ such that $(v, 1)\in N'$. 
  Then there exists a sublattice $N''$ of $N$ such that $N'\oplus (N''\times\{0\}) = N\oplus\ZZ$ and $N/N''$ is torsion free. 
\end{lemma}
\begin{proof}
  Let $\iota$ denote a canonical inclusion $N\rightarrow N\oplus\ZZ$ such that $\iota(w) = (w, 0)$ for any $w\in N$. 
  Let $N_1$ be an inverse image of $N'$ under $\iota$. 
  We claim that $(N_1\times\{0\})\oplus \ZZ(v, 1) = N'$. 
  Indeed, for any $(w, a)\in N'$, we have $(av, a)\in \ZZ(v, 1)$ and $(w-av, 0)\in (N_1\times\{0\})$. 
  Moreover, $N/N_1$ is torsion free because $(N\oplus\ZZ)/N'$ is torsion free. 
  Thus, there exists a sublattice $N_2$ of $N$ such that $N_1\oplus N_2 = N$. 
  Therefore, $(N_1\times\{0\})\oplus(N_2\times\{0\})\oplus\ZZ(v, 1) = N\oplus\ZZ$, so let $N''$ denote $N_2$. 
\end{proof}

When we compute the stable birational volume, it is important to construct a strictly toroidal scheme. 
The following lemma indicates that the sufficient conditions for the reduced-ness of the fiber of the closed orbit can be given by the combinatorial conditions.
\begin{lemma}\label{lem: relative toroidal}
  Let $N$ be a lattice of finite rank and let $\sigma$ be a strongly convex rational polyhedral cone in $(N\oplus\ZZ)_\RR$. 
  We assume that $\sigma\subset N_\RR\times\RR_{\geq0}$ and $\sigma \not\subset N_\RR\times\{0\}$, and for every ray $\gamma\preceq\sigma$ with $\gamma\cap (N_\RR\times\{1\})\neq\emptyset$, we have $\gamma\cap (N\times\{1\})\neq\emptyset$. 
  Let $t\in k[\ZZ^\vee]$ be the torus invariant monomial associated with $1\in\ZZ^\vee$.
  Then the following statements follow:
  \begin{enumerate}
      \item[(a)] There exists a sublattice $N'$ of $N$ such that $N/N'$ is torsion free and the restriction $(\pi, \mathrm{id})|_{\langle\sigma\rangle\cap (N\oplus\ZZ)}\colon \langle\sigma\rangle\cap (N\oplus\ZZ) \rightarrow N/N'\oplus\ZZ$ is isomorphic, where $\pi\colon N\rightarrow N/N'$ be the quotient map and $(\pi, \mathrm{id}_\ZZ)\colon N\oplus\ZZ\rightarrow N/N'\oplus\ZZ$ is a direct product of morphisms $\pi$ and ${\mathrm{id}}_\ZZ$. 
      \item[(b)] For such $N'$ in (a), let $\sigma'$ denote $(\pi, \mathrm{id}_\ZZ)_\RR(\sigma)$. 
      We remark that $\sigma'$ is a strongly convex rational polyhedral cone in $(N/N'\oplus\ZZ)_\RR$. 
      Then the second projection $\pr_2\colon (N/N')\oplus\ZZ\rightarrow\ZZ$ is compatible with the cones $\sigma'$ and $[0, \infty)$. 
      Moreover, if $\gamma'\cap (N/N')_\RR\times\{1\}\neq\emptyset$, then $\gamma'\cap (N/N')\times\{1\}\neq\emptyset$ for any ray $\gamma'\preceq\sigma'$.   
      \item[(c)] Let $F$ denote $(\pr_2)_*^{-1}(0)$ which is the fiber of the toric morphism $(\pr_2)_*\colon X(\sigma')\rightarrow\A^1$ at a closed orbit of $\A^1_k$. 
      Then $F$ is reduced. 
      \item[(d)] 
      There exists a toric monoid $S$ and $\omega\in S$ such that $k[S]/(\chi^\omega)$ is reduced and $X(\sigma')\times_{\A^1_k}\Spec(\Rt)\cong \Spec(\Rt[S]/(t - \chi^\omega))$. 
  \end{enumerate}
\end{lemma}
\begin{proof}
  We prove the statements from (a) to (d) in order. 
  \begin{enumerate}
      \item[(a)] By the assumption of $\sigma$, there exists a ray $\gamma\preceq\sigma$ such that $\gamma\not\subset N_\RR\times \{0\}$. 
      Because $\gamma\cap (N\times\{1\})\neq \emptyset$, there exists $v\in N$ such that $(v, 1)\in\sigma$.  
      Thus, by Lemma \ref{lem: relative decomposition}, there exists a sublattice $N_0$ of $N$ such that $(\langle\sigma\rangle\cap(N\oplus\ZZ))\oplus (N_0\times\{0\}) = N\oplus\ZZ$. 
      Therefore, the statement holds. 
      \item[(b)] 
      Because $\sigma\subset N_\RR\times\RR_{\geq 0}$, we have $\sigma'\subset (N/N')_\RR\times \RR_{\geq 0}$. 
      Thus, the second projection $\pr_2\colon (N/N')\oplus\ZZ\rightarrow\ZZ$ is compatible with the cones $\sigma'$ and $[0, \infty)$. 
      Moreover, there is a one-to-one correspondence with rays of $\sigma$ and those of $\sigma'$. 
      Thus, the second statement holds by the assumption of $\sigma$ and $N'$. 
      
      \item[(c)]
      There is a one-to-one correspondence between irreducible components of $F$ and orbit closures of $X(\sigma')$ associated with a ray of $\sigma'$ such that it intersects $(N/N')_\RR\times\{1\}$. 
      Let $\gamma'\preceq\sigma'$ be a ray with $\gamma'\cap (N/N')_\RR\times\{1\}\neq\emptyset$ and let $\eta'\in X(\sigma)$ denote the generic point of $\overline{O_{\gamma'}}$. 
      Let $v'\in(N/N')\oplus\ZZ$ denote a minimal generator of $\gamma'$. 
      Then $v'$ is a unique torus invariant valuation of $X(\sigma')$ such that the valuation ring associated with $v'$ is $\OO_{X(\sigma'), \eta'}$ and the image of $v'$ is $\ZZ$. 
      Now, we show that $v'(t) = 1$. 
      Indeed, there exists $v''\in N/N'$ such that $v' = (v'', 1)$ by (b). 
      Because the second component of $v'$ is $1$, we have $v'(t) = 1$. 
      Thus, $\OO_{X(\sigma'), \eta'}/(t) = \OO_{F, \eta'}$ is a field, in particular, $\OO_{F, \eta'}$ is regular. 
      Because this follows for any generic points of irreducible components of $F$, $F$ has a property $(R_0)$. 
      Moreover, $X(\sigma')$ is a Cohen-Macaulay and an integral scheme. 
      Hence, $F$ is a Cohen-Macaulay scheme. 
      In particular, $F$ has a property $(S_1)$. 
      Therefore, $F$ is reduced. 
      \item[(d)] Let $S$ be the monoid associated with $X(\sigma')$, let $g\colon X(\sigma')\rightarrow \Spec(k[S])$ be a natural isomorphism, and let $\omega\in S$ be the element associated with $t$. 
      We remark that $\sigma'$ is a full cone in $(N/N'\oplus\ZZ)_{\RR}$ by the construction of $\sigma'$, and hence, $S$ is a toric monoid.  
      We regard $k[S]$ as a $k[t]$-algebra by a $k$-morphism $q\colon k[t]\rightarrow k[S]$ such that $q(t) = \chi^\omega$. 
      Then there exists the following commutative diagram:
      \begin{equation*}
          \begin{tikzcd}
              X(\sigma')\ar[r, "g"]\ar[d, "(\pr_2)_*"]& \Spec(k[S])\ar[ld, "q^*"]\\
              \Spec(k[t])& ,
          \end{tikzcd}
      \end{equation*}
      where $q^*$ is a morphism induced by $q$. 

      Let $r$ denote $q\otimes\mathrm{id}_{k[S]}\colon k[t][S]\rightarrow k[S]$. 
      We regard $k[t][S]$ as a $k[t]$-algebra by a $k$-morphism $s\colon k[t]\rightarrow k[t][S]$ such that $s(t) = t$. 
      We remark that $r$ is a $k[t]$-morphism. 
      Let $r_0$ denote the quotient map $k[t][S]/(t - \chi^\omega)\rightarrow k[S]$ induced by $r$ and let $s_0$ denote a morphism $k[t]\rightarrow k[t][S]/(t - \chi^\omega)$ induced by $s$. 
      Because of the definition of $r$, $r_0$ is an isomorphism. 
      Thus, there exists the following commutative diagram:
      \begin{equation*}
          \begin{tikzcd}
              X(\sigma')\ar[r, "g"]\ar[d, "(\pr_2)_*"]& \Spec(k[S])\ar[ld, "q^*"]\ar[d, "r^*_0"]\\
              \Spec(k[t])& \Spec(k[t][S]/(t-\chi^\omega))\ar[l, "s^*_0"],
          \end{tikzcd}
      \end{equation*}
      where $r^*_0$ and $s^*_0$ are morphisms induced by $r_0$ and $s_0$. 
      Thus, $X(\sigma')\times_{\A^1_k}\Spec(\Rt)\cong \Spec(\Rt[S]/(t - \chi^\omega))$, and $k[S]/(\chi^\omega)$ is isomorphic to a global section ring of the fiber of $F$, in particular, it is reduced by (c). 
  \end{enumerate}
\end{proof}
The following lemma is needed for the explicit calculation of the normal fan of the polytope. 
  \begin{lemma}\label{lem: compute normal fan}
      Let $S$ be a finite set, let $N$ be a lattice of finite rank, let $M$ be the dual lattice of $N$, let $u\colon S\rightarrow M$ be a map, and let $P(u)$ be a convex closure of $u(S)$ in $M_\RR$, $\Sigma(u)$ be the normal fan of $P(u)$. 
      Let $D(u)$ denote the rational polyhedral cone in $(M\oplus\ZZ)_\RR$ generated by $\{(u(i), 1)\mid i\in S\}$, let $C(u)\subset (N\oplus\ZZ)_\RR$ denote the dual cone of $D(u)$, let $\pi\colon N\oplus\ZZ\rightarrow N$ be the first projection, let $\Delta$ be a rational polyhedral convex cone in $N_\RR$, and let $\Sigma(\Delta, u)$ denote the set $\{\sigma_1\cap\sigma_2\mid \sigma_1\in \Delta, \sigma_2\in\Sigma(u)\}$. 
      Then the following statements hold: 
      \begin{itemize}
          \item[(a)] Let $\sigma\preceq C(u)$ be a face. 
          If $(0, 1)\notin \sigma$, then $\pi_\RR(\sigma)\in \Sigma(u)$. 
          \item[(b)] Let $\sigma'\in \Sigma(u)$. 
          Then there uniquely exists a face $\sigma\preceq C(u)$ such that $(0, 1)\notin \sigma$ and $\sigma' = \pi_\RR(\sigma)$. 
          \item[(c)] The set $\Sigma(\Delta, u)$ is a rational polyhedral convex fan in $N_\RR$ and a refinement of $\Delta$. 
          \item[(d)] Let $\tau\in\Delta$, let $D(\tau, u)$ denote the rational polyhedral cone in $(M\oplus\ZZ)_\RR$ generated by $\{(u(i), 1)\mid i\in S\}\cup\{(\omega, 0)\mid \omega\in \tau^\vee\}$, let $C(\tau, u)\subset (N\oplus\ZZ)_\RR$ denote the dual cone of $D(\tau, u)$, and let $\gamma\preceq C(\tau, u)$. 
          If $(0, 1)\notin\gamma$, then $\pi_\RR(\gamma)\in \Sigma(\Delta, u)$ and $\pi_\RR(\gamma)\subset\tau$. 
          \item[(e)] We keep the notation in (d). 
          Let $\gamma'\in \Sigma(\Delta, u)$. 
          If $\gamma'\subset \tau$, then there uniquely exists a face $\gamma\preceq C(\tau, u)$ such that $(0, 1)\notin\gamma$ and $\pi_\RR(\gamma) = \gamma'$. 
      \end{itemize}
  \end{lemma}
  \begin{proof}
      We prove the statement from (a) to (e) in order. 
      \begin{itemize}
          \item[(a)] Let $\sigma^*$ denote the face $\sigma^\perp\cap D(u)\preceq D(u)$. 
          By the assumption, $\sigma^* \neq \{0\}$. 
          On the identification of $D(u)\cap (M_\RR\times \{1\})$ and $P(u)$, $\sigma^*\cap (M_\RR\times \{1\})$ is a face of $P(u)$. 
          Let $Q(\sigma)$ denote this face. 
          Then we can check that $\{v\in N_\RR\mid \langle v, \omega_1-\omega_2\rangle\geq 0, \forall \omega_1\in P(u), \forall\omega_2\in Q(\sigma)\} = \pi_\RR(\sigma)$. 
          The left-hand side is an element in $\Sigma(u)$. 
          Thus, the statement holds. 
          \item[(b)] For $\sigma'\in \Sigma(u)$, there uniquely exists a face $Q\preceq P(u)$ such that $\sigma' = \{v\in N_\RR\mid \langle v, \omega_1-\omega_2\rangle\geq 0, \forall \omega_1\in P(u), \forall\omega_2\in Q\}$. 
          Let $\tau$ denote the subset of $(M\oplus\ZZ)_\RR$ defined as $\tau = \{(c\omega, c)\mid \omega\in Q, c\in\RR_{\geq 0}\}$. 
          Because $Q$ is a face of $P(u)$, $\tau$ is also a face of $D(u)$. 
          Let $\sigma$ denote $\tau^\perp\cap C(u)$. 
          Because $Q\neq \emptyset$, $(0, 1)\notin\sigma$. 
          Then we can check that $\sigma' = \pi_\RR(\sigma)$. 
          The uniqueness is held by the condition that $(0, 1)\notin\sigma$ and the property of face. 
          \item[(c)] Because $\Sigma(u)$ is a rational polyhedral convex fan with $\supp(\Sigma(u)) = N_\RR$ and $\Delta$ is a rational polyhedral convex fan, $\Sigma(\Delta, u)$ is also a rational polyhedral convex fan with $\supp(\Sigma(\Delta, u)) = \supp(\Delta)$. 
          By the definition of $\Sigma(\Delta, u)$, it is obvious that $\Sigma(\Delta, u)$ is a refinement of $\Delta$. 
          \item[(d)] By the definition of $C(\tau, u)$, $C(\tau, u) = C(u)\cap \pi^{-1}_\RR(\tau)$. 
          Thus, there exists faces $\sigma\preceq C(u)$ and $\tau'\preceq \tau$ such that $\gamma = \sigma\cap \pi^{-1}_\RR(\tau')$. 
          By the assumption of $\gamma$, $(0, 1)\notin\sigma$. 
          Thus, $\pi_\RR(\gamma) = \pi_\RR(\sigma)\cap \tau'\in\Sigma(\Delta, u)$ by (a). 
          \item[(e)] By (b), there exists $\sigma\preceq C(u)$ and $\tau'\in\Delta$ such that $(0, 1)\notin\sigma$ and $\gamma' = \pi_\RR(\sigma)\cap \tau'$. 
          Because $\gamma'\subset\tau$, $\gamma' = \pi_\RR(\sigma)\cap (\tau'\cap \tau)$ and hence, we may assume $\tau'$ is a face of $\tau$. 
          Then $\sigma\cap \pi^{-1}_\RR(\tau')$ is a face of $C(\tau, u)$ and $(0, 1)\notin\sigma\cap \pi^{-1}_\RR(\tau')$. 
          Let $\gamma$ denote $\sigma\cap \pi^{-1}_\RR(\tau')$, then $\pi_\RR(\gamma) = \gamma'$. 
          The uniqueness holds as in the proof of (b). 
      \end{itemize}
  \end{proof}
\subsection{Lemmas related to the scheme theoretic image}
In this subsection, we correct lemmas related to the scheme theoretic image. 
\begin{lemma}\label{lem: open-image}
  Let $f\colon X\rightarrow Y$ be a quasi-compact morphism of schemes, and $Z$ be the scheme theoretic image of $f$, and $V$ be an open subscheme of $Y$. 
  Then $Z\cap V$ is the scheme theoretic image of $f|_{f^{-1}(V)}\colon f^{-1}(V)\rightarrow V$. 
\end{lemma}
\begin{proof}
  Because $f$ is quasi-compact, $\mathscr{I} = \ker(\OO_Y\rightarrow f_*\OO_X)$ is the quasi-coherent ideal sheaf associated with the closed subscheme $Z$ of $Y$. 
  Because $\mathscr{I}|_V = \ker(\OO_{V}\rightarrow (f|_{f^{-1}(V)})_*\OO_{f^{-1}(V)})$ and $\mathscr{I}|_V$ is the quasi-coherent ideal sheaf associated with the closed subscheme $Z\cap V$ of $V$, $Z\cap V$ is the scheme theoretic image of $f|_{f^{-1}(V)}$. 
\end{proof}
\begin{lemma}\label{lem:covering-image}
  Let $f\colon X\rightarrow Y$ be a quasi-compact morphism of schemes, $Z$ be a closed subscheme of $Y$, and $\{V_i\}_{i\in I}$ be an open covering of $Y$. 
  Then the following statements are equivalent. 
  \begin{itemize}
      \item[(a)] A closed subscheme $Z$ of $Y$ is the scheme theoretic image of $f$. 
      \item[(b)] For any $i\in I$, a closed subscheme $Z\cap V_i$ of $V_i$ is the scheme theoretic image of $f|_{f^{-1}(V_i)}$. 
  \end{itemize}
\end{lemma}
\begin{proof}
  When the statement (a) holds, the statement (b) holds by Lemma \ref{lem: open-image}. 
  Then we assume the statement (b). 
  Let $\mathscr{I}_Z\subset \OO_Y$ be the quasi-coherent ideal sheaf associated with $Z$ of $Y$. 
  By the assumption, for any $i\in I$, $\mathscr{I}_Z|_{V_i} = \ker(\OO_{V_i}\rightarrow (f|_{f^{-1}(V_i)})_*\OO_{f^{-1}(V_i)})$.
  Thus, $\mathscr{I}_Z = \ker(\OO_Y\rightarrow f_*\OO_X)$ because $\{V_i\}_{i\in I}$ is an open covering of $Y$. 
  Therefore, $Z$ is the scheme theoretic image of $f$. 
\end{proof}
\begin{lemma}\label{lem: flat-closure lemma}
    Let $S$ be a scheme, let $X$ be a scheme over $S$, and let $Y$ be a closed subscheme of $X$. 
    Let $\varphi\colon S'\rightarrow S$ be a quasi-compact morphism of schemes, let $X'$ denote $X\times_S S'$, and let $Y'$ denote $Y\times_S S'$. 
    We assume that $Y\rightarrow S$ is a flat morphism and the scheme theoretic image of $\varphi$ is $S$. 
    Then the scheme theoretic image of $Y'\rightarrow X$ is $Y$.
    \begin{equation*}
        \begin{tikzcd}
            Y'\ar[r]\ar[d]&X'\ar[r]\ar[d]&S'\ar[d, "\varphi"]\\
            Y\ar[r]&X\ar[r]&S
        \end{tikzcd}
    \end{equation*}
  \end{lemma}
  \begin{proof}
      By Lemma \ref{lem:covering-image}, we may assume $S$ and $X$ are affine. 
      Let $A$ and $B$ be rings such that $S = \Spec(A)$ and $X = \Spec(B)$, let $I$ be an ideal of $B$ such that $Y = \Spec(B/I)$, and let $\{\Spec(C_r)\}_{1\leq r\leq n}$ be a finite affine covering of $S'$. 
      By the assumption, the ring morphism $A\rightarrow \bigoplus_{1\leq r\leq n}C_r$ induced by $S'\rightarrow S$ is injective. 
      Because $B/I$ is flat over $A$, $B/I \rightarrow \bigoplus_{1\leq r\leq n}C_r\otimes_A B/I$ is also injective. 
      Let $f$ denote the ring morphism of the latter one, and let $g\colon B\rightarrow \bigoplus_{1\leq r\leq n}C_r\otimes_A B/I$ denote a ring morphism defined by the composition of $f$ and the quotient morphism $B\rightarrow B/I$. 
      Then the ideal associated with the scheme theoretic image of $Y'\rightarrow X$ is $\mathrm{Ker}(g)$, and hence $\mathrm{Ker}(g) = I$. 
      Thus, the statement holds. 
  \end{proof}
  \begin{lemma}\label{lem: base change and closure}
    Let $f\colon Y\rightarrow X$ be a quasi-compact morphism of schemes, let $X'$ be a scheme over $X$, let $Y'$ denote $Y\times_X X'$, let $f'$ denote the morphism $Y'\rightarrow X'$, let $Z$ denote the scheme theoretic image of $f$, and let $Z'$ denote the scheme theoretic image of $f'$. 
    If $X'$ is flat over $X$, then $Z'$ is $Z\times_X X'$. 
    \begin{equation*}
        \begin{tikzcd}
            Y'\ar[r, "f'"]\ar[d]&X'\ar[d]\\
            Y\ar[r, "f"]&X
        \end{tikzcd}
    \end{equation*}
  \end{lemma}
  \begin{proof}
      By Lemma \ref{lem:covering-image}, we may assume $X$ and $X'$ are affine. 
      Let $A$ and $B$ be rings such that $X = \Spec(A)$ and $X' = \Spec(B)$, and let $\{\Spec(C_r)\}_{1\leq r\leq n}$ be a finite affine covering of $Y$. 
      Let $f$ denote the ring morphism $A\rightarrow \bigoplus_{1\leq r\leq n}C_r$ induced by $Y\rightarrow X$. 
      Then the ideal associated with $Z$ is $\mathrm{Ker}(f)$. 
      Because $B$ is flat over $A$, the kernel of $f\otimes \id_B$ is $\mathrm{Ker}(f)B$. 
      Thus, the scheme theoretic image of $Y'\rightarrow X'$ is $Z\times_X X'$. 
  \end{proof}
  \begin{remark}
      We will provide a counterexample in which Lemma \ref{lem: base change and closure} fails to hold without assuming flatness. 
      Let $X$ denote $\A^2_\C$, let $X'\rightarrow X$ denote the blowing-up of $X$ at $(0, 0)$, $W$ denote a curve passing through $(0, 0)$, $Y$ denote $W\setminus\{(0, 0)\}$, and $f\colon Y\rightarrow X$ denote an immersion. 
      Then $Y'$ coincides with $\hat{W}\setminus E$, where $\hat{W}$ is the strict transform of $Z$ and $E$ is the exceptional divisor. 
      Thus, it holds that $Z = W$ and $Z' = \hat{W}$. 
      However, $Z\times_X X'$ is not irreducible. 
      
  \end{remark}
\subsection{Lemmas related to a scheme over valuation ring}
In this subsection, we correct for lemmas related to a scheme over $\Rt$. 
The definitions for $\Rt$ and $\Kt$ are given in Section 2.
For a positive integer $l$, $R_l$ denote $k[[t^{1/l}]]$ and $K_l$ denote $k((t^{1/l}))$. 

The following lemma is referenced in \cite[Remark. 4.6]{G13}; however, for the sake of
thoroughness, a detailed proof is provided here.
\begin{lemma}\label{lem:Gubler}
  \cite[Remark. 4.6]{G13}
      Let $\mathfrak{X}$ be a flat scheme over $\Spec(\Rt)$, let $X$ denote a base change of $\mathfrak{X}\rightarrow \Spec(\Rt)$ to $\Spec(\Kt)$, and let $Y$ be a closed subscheme of $X$. 
      Then there exists a unique closed subscheme $\mathfrak{Y}$ of $\mathfrak{X}$ such that the following conditions hold:
      \begin{enumerate}
          \item[(1)] A scheme $\mathfrak{Y}$ is flat over $\Spec(\Rt)$. 
          \item[(2)] A base change of $\mathfrak{Y}\rightarrow \Spec(\Rt)$ to $\Spec(\Kt)$ is equal to $Y$ as a closed subscheme of $X$. 
      \end{enumerate}
      In fact, such $\mathfrak{Y}$ is the scheme theoretic closure of $Y$ in $\mathfrak{X}$. 
  \end{lemma}
  \begin{proof}
      We remark that the open immersion $X\hookrightarrow \mathfrak{X}$ is quasi-compact because $\Spec(\Kt)\rightarrow\Spec(\Rt)$ is quasi-compact. 
      Thus, the composition $Y\hookrightarrow X\rightarrow \mathfrak{X}$ is quasi-compact too.
      We already know that the statement holds in the case when $\mathfrak{X}$ is an affine scheme by \cite[Proposition. 4.4]{G13}. 

      First, we show the existence of $\mathfrak{Y}$. 
      Let $\mathfrak{Y}_0$ be the scheme theoretic closure of $Y$ in $\mathfrak{X}$. 
      Thus, for any affine open subset $U\subset \mathfrak{X}$, the scheme theoretic closure of $U\cap Y$ in $U$ is $U\cap\mathfrak{Y}_0$ by Lemma \ref{lem: open-image}. 
      For such $U$, $U$ is flat over $\Spec(\Rt)$, so that $(U\cap\mathfrak{Y}_0)_{\Kt} = U\cap Y$ and $U\cap\mathfrak{Y}_0$ is flat over $\Spec(\Rt)$ by \cite[Proposition. 4.4]{G13}. 
      By considering one affine open covering of $\mathfrak{X}$, $\mathfrak{Y}_0$ is flat over $\Spec(\Rt)$ and the generic fiber of $\mathfrak{Y}_0$ is equal to $Y$ as a closed subscheme of $X$. 

      Second, we show the uniqueness of $\mathfrak{Y}$. 
      Let $\mathfrak{Y}_1$ be a closed subscheme of $\mathfrak{X}$ such that $\mathfrak{Y}_1$ is flat over $\Spec(\Rt)$ and the generic fiber of $\mathfrak{Y}_1$ is equal to $Y$ as a closed subscheme of $X$. 
      Let $U$ be an affine open subset of $\mathfrak{X}$. 
      Then $U\cap \mathfrak{Y}_1$ is flat over $\Spec(\Rt)$ and the generic fiber of $U\cap \mathfrak{Y}_1$ is equal to $U\cap Y$ as a closed subscheme of $U\times_\Rt \Kt$.
      Because $U$ is flat over $\Spec(\Rt)$, $U\cap \mathfrak{Y}_1$ is the scheme theoretic closure of $U\cap Y$ in $U$ 
      by \cite[Proposition. 4.4]{G13}. 
      Thus, by considering one affine open covering of $\mathfrak{X}$, $\mathfrak{Y}_1$ is the scheme theoretic closure of $Y$ in $\mathfrak{X}$ by Lemma \ref{lem:covering-image}.
  \end{proof}
  The following lemma is well-known and referenced in \cite[\href{https://stacks.math.columbia.edu/tag/055C}{Tag 055C}]{SP}.  
  \begin{lemma}\label{lem: DVR-connected}
  \cite[\href{https://stacks.math.columbia.edu/tag/055C}{Tag 055C}]{SP}
      Let $S$ be a DVR, let $u$ be its uniformizer, let $Q$ be the fraction field of $S$, and let $\kappa$ be the residue field of $S$. 
      Let $B$ be a flat $S$-algebra and let $x\in B\otimes_S Q$ be an idempotent element. 
      We remark that extension morphism $B\rightarrow B\otimes_S Q$ is injective, and $u$ is not a zero divisor of $B$ because $B$ is flat over $S$. 
      We assume that $B\otimes_S \kappa$ is reduced. 
      Then we have $x\in B$.
  \end{lemma}
  \begin{proof}
      Let $a\in B$ and let $n\in\ZZ_{\geq 0}$ be a non-negative integer such that $x = \frac{a}{u^n}$ in $B\otimes_S Q$. 
      We can take $n$ minimal. 
      Because $x$ is an idempotent element, $\frac{a}{u^n} = \frac{a^2}{u^{2n}}$  in $B\otimes_S Q$. 
      Then $a^2 = au^n$  in $B$ because $u$ is not a zero divisor of $B$. 
      If $n \geq 1$, this shows that $a + (u)B$ is a nilpotent element in $B\otimes_S \kappa$. 
      However, by the assumption, $a\in (u)B$, so that this is a contradiction to the minimality of $n$. 
      Thus, we can take $n$ as $0$, so that $x\in B$.  
  \end{proof}
  For a general valuation ring, we do not know whether Lemma \ref{lem: DVR-connected} holds. 
  However, for our application, we can assume $S = \Rt = \cup_{n>0} k[[t^{1/n}]]$ and $B$ is flat and of finite presentation over $S$, then Lemma \ref{lem: DVR-connected} holds. 
  Indeed, this claim is proven in Lemma \ref{lem: stacks exchange}. 
  
  The following lemma serves as a preparation for the proof of Lemma \ref{lem: stacks exchange}. 
  \begin{lemma}\label{lem:R-of finite presentation}
      Let $A$ be a flat and of finite presentation ring over $\Rt$ and let $x\in A\otimes_\Rt \Kt$. 
      Then there exists a positive integer $l$ and a sub $R_l$-algebra $C$ of $A$ such that
      \begin{itemize}
          \item A $R_l$-algebra $C$ is flat over $R_l$.
          \item All morphisms in the following diagram are injective:
          \begin{equation*}
              \begin{tikzcd}
                  C\ar[r]\ar[d, hook] & C\otimes_{R_l} K_l\ar[d]\\
                  A\ar[r]&A\otimes_{\Rt} \Kt,
              \end{tikzcd}
          \end{equation*}
          where the right one is induced by the base change of $R_l$-morphism $C\hookrightarrow A$ to $K_l$. 
          \item A subring $C\otimes_{R_l} K_l$ of $A\otimes_{\Rt} \Kt$ contains $x$. 
          \item The inclusion $C\hookrightarrow A$ induce the isomorphism $C\otimes_{R_l}k \cong A\otimes_\Rt{k}$. 
      \end{itemize}
  \end{lemma}
  \begin{proof}
      By the assumption, there exist $N\in\ZZ_{> 0}$ and $f_1, \ldots, f_r\in \Rt[X_1, \ldots, X_N]$ such that $\Rt[X_1, \ldots, X_N]/(f_1, \ldots, f_r)\cong A$ as an $\Rt$-algebra. 
      Moreover, there exists $g\in \Kt[X_1, \ldots, X_N]$ such that $g + (f_1, \ldots, f_r) = x$ by the identification with $\Kt[X_1, \ldots, X_N]/(f_1, \ldots, f_r)$ and $A\otimes_{\Rt} \Kt$. 
      Thus, by the definition of $\Rt$ and $\Kt$, there exists $l\in\ZZ_{>0}$ such that $f_1, \ldots, f_r\in R_l[X_1, \ldots, X_N]$ and $g\in K_l[X_1, \ldots, X_N]$. 
      
      Let $C$ denote a $R_l$-algebra $R_l[X_1, \ldots, X_N]/(f_1, \ldots, f_r)$. 
      We remark that $C\otimes_{R_l} \Rt\cong \Rt[X_1, \ldots, X_N]/(f_1, \ldots, f_r)$ as an $\Rt$-algebra, and $\Rt\otimes_{R_l}K_l = \Rt[t^{-\frac{1}{l}}] = \Kt$. 

      First, we show that $C$ is flat over $R_l$. 
      Because $R_l$ is a DVR, it is enough to show that $C$ is a torsion-free $R_l$-module. 
      Let $f\colon C\rightarrow C$ denote an endomorphism of $C$ defined by the scalar product of $t^\frac{1}{l}$. 
      Then it is enough to show that $f$ is injective. 
      Because $\Rt$ is a faithfully flat $R_l$-module, it is enough to show that $f\otimes\id_{\Rt}$, which is induced by $f$, is injective. 
      By the assumption, $\Rt[X_1, \ldots, X_N]/(f_1, \ldots, f_r)$ is flat over $\Rt$. 
      Thus, $f\otimes\id_{\Rt}$ is injective. 

      Second, the extension morphism $C\rightarrow \Rt[X_1, \ldots, X_N]/(f_1, \ldots, f_r)$ is injective because $C$ is flat over $R_l$. 
      Thus, we can regard $C$ as an $R_l$-sub algebra of $A$ on the identification with $\Rt[X_1, \ldots, X_N]/(f_1, \ldots, f_r)$ and $A$. 
      We can easily check that this inclusion $C\hookrightarrow A$ induces an isomorphism $C\otimes_{R_l}k\cong A\otimes_{\Rt}k$ because both rings are isomorphic to $k[X_1, \ldots, X_N]/(\overline{f_1}, \ldots, \overline{f_r})$, where $\overline{f_i}\in k[X_1, \ldots, X_N]$ is a polynomial whose each coefficient is replaced by constant terms of coefficients of $f_i$. 
      
      Third, in the diagram above, we show that all morphisms are injective. 
      Indeed, the left one is injective by the argument above. 
      The upper one is injective because $C$ is flat over $R_l$. 
      The right one is injective because $K_l$ is flat over $R_l$. 
      At this point, we remark that $A\otimes_\Rt \Kt = A\otimes_\Rt \Rt\otimes_{R_l}K_l = A\otimes_{R_l}K_l$. 
      The lower one is injective because $A$ is flat over $\Rt$. 

      Finally, we show that $x\in C\otimes_{R_l} K_l$. 
      Indeed, $g + (f_1, \ldots, f_r)\in C\otimes_{R_l} K_l$ and this element go to $x$ along the right morphism in the diagram above. 
  \end{proof}
  \begin{lemma}\label{lem: stacks exchange}
      Let $\eta$ denote a generic point of $\Spec(\Rt)$. 
      Let $X$ be a flat and locally of finite presentation scheme over $\Spec(\Rt)$. 
      For an open subscheme $U\subset X$, let $U_\eta$ denote the generic fiber of $U$. 
      We assume that $X$ is a connected scheme and the closed fiber $X_k$ is reduced.
      Then the following statements follow:
      \begin{enumerate}
          \item[(a)] Let $U$ be an open subscheme of $X$. 
          Then the following ring morphism induced by an open immersion $U_\eta\rightarrow U$ is injective: 
          \[
              \Gamma(U, \OO_{X})\rightarrow \Gamma(U_\eta, \OO_{X}).
          \]
          \item[(b)] A scheme $X_\eta$ is connected. 
      \end{enumerate}
  \end{lemma}
  \begin{proof}
      We prove the statements from (a) to (b). 
      \begin{enumerate}
          \item[(a)] When $U$ is an affine open subscheme of $X$, the statement holds. 
          Indeed, let $A$ be a ring such that $\Spec(A) = U$. 
          Then $A$ is flat over $\Rt$. 
          Thus, an extension map $A\rightarrow A\otimes_{\Rt}\Kt$ is injective. 
          In the general case, $U$ can be covered by open affine subschemes of $X$. 
          Because $\OO_X$ is a sheaf over $X$, the statement holds. 
          \item[(b)] We show that the statement holds by a contradiction. 
          We assume that $X_\eta$ is not connected. 
          Then there exists an idempotent element $e\in\Gamma(X_\eta, \OO_X)$ such that $e\neq 0$ and $e\neq1$. 
          Let $U$ be an open affine subscheme of $X$ and let $A$ be a ring such that $U = \Spec(A)$ and $A$ is of finite presentation over $\Rt$. 
          We remark that $A$ is flat over $\Rt$ and $U_\eta = \Spec(A[\frac{1}{t}])$, so that $e|_{U_\eta}\in A\otimes_{\Rt} \Kt$. 
          Then for $A$ and $e|_{U_\eta}\in A\otimes_{\Rt} \Kt$, there exists $l\in\ZZ_{>0}$ and a sub $R_l$-algebra $C$ of $A$ such that $l$ and $C$ satisfy the conditions in Lemma \ref{lem:R-of finite presentation}. 
          Thus, $e|_{U_\eta}\in C\otimes_{R_l}K_l$ and $C\otimes_{R_l}K_l\rightarrow A\otimes_{\Rt}\Kt$ is injective. 
          Hence, $e|_{U_\eta}$ is an idempotent element of $C\otimes_{R_l}K_l$. 
          Moreover, $C$ is a flat over a DVR $R_l$ and $C\otimes_{R_l}k\cong A\otimes_{\Rt} k$ is reduced by the assumption. 
          Thus, by Lemma \ref{lem: DVR-connected}, we have $e|_{U_\eta}\in C$ and $e|_{U_\eta}\in A$. 
          This shows that $e|_{U_\eta}\in\Gamma(U, \OO_X)$ for any affine open subscheme $U$ of $X$. 
          Hence, by (a), we can check that $e\in\Gamma(X, \OO_X)$ and $e$ is an idempotent element in $\Gamma(X, \OO_X)$. 
          The element $e\in\Gamma(X, \OO_X)$ is a non-trivial idempotent element of $\Gamma(X, \OO_X)$, but it is a contradiction to the assumption that $X$ is connected. 
      \end{enumerate}
  \end{proof}
  \subsection{Lemmas related to a commutative algebra}
In this subsection, we correct for lemmas related to commutative algebra. 

The following lemma claims that general members in the linear system, which is generated by units, are non-empty if the dimension of the linear system is greater than $1$.   
  \begin{lemma}\label{lem: dim 2 lemma}
    Let $R$ be an integral domain of finite type over an algebraically closed field $k$, let $\chi_1, \ldots, \chi_r$ be units in $R$, let $V$ denote the $k$-linear subspace of $R$ generated by $\{\chi_1, \ldots, \chi_r\}$, let $f$ denote $\sum_{1\leq l\leq r}\chi_l t_l\in R[t_1, \ldots, t_n]$, let $X$ denote $\Spec(R)$, and let $V(f)$ denote the closed subscheme of $X\times_k \A^r_k$ defined by $f$. 
    We assume that $\dim_k(V)\geq 2$. 
    Then the following statements hold:
    \begin{itemize}
      \item[(a)] The restriction map $\pr_2|_{V(f)}\colon V(f)\rightarrow \A^r_k$ is dominant. 
      \item[(b)] There exists an open subset $W\subset \A^r_k$ such that $\sum_{1\leq l\leq r}a_l\chi_l\notin R^*$ 
      for any $a = (a_l)_{1\leq l\leq r}\in W(k)$. 
    \end{itemize}
  \end{lemma}
  \begin{proof}
    We prove the statements from (a) to (b). 
    \begin{itemize}
      \item[(a)] 
      We show that the statement holds by contradiction. 
      We assume that $\pr_2|_{V(f)}$ is not dominant. 
      By the definition of $f$, $\pr_1|_{V(f)}\colon V(f)\rightarrow X$ is a vector bundle of rank $r - 1$. 
      In particular, $V(f)$ is irreducible and $\dim(V(f)) = \dim(X) + r - 1$. 
      Let $Z$ be the scheme theoretic image of $\pr_2|_{V(f)}$. 
      By the argument above, $Z$ is integral, $\dim(Z) = r - 1$, and $V(f) = X\times Z$. 
      Let $x\in X(k)$. 
      Then $Z$ is a closed subscheme of $\A^r_k$ defined by $\sum_{1\leq l\leq r}\chi_l(x)t_l$ by considering $V(f)\cap {\{x\}\times \A^r_k}$. 
      Thus, there exists $g\in R^*$ such that $\chi_l = \chi_l(x)g$ for any $1\leq l\leq r$. 
      However, this is a contradiction to the assumption of $\dim_k(V)$.     
      \item[(b)] By Chevalley's Theorem, the image of $\pr_2|_{V(f)}$ contains a dense open subset $W$ of $\A^r_k$. 
      Thus, $\sum_{1\leq l\leq r}a_l\chi_l\notin R^*$ for any $a \in W(k)$. 
    \end{itemize}
  \end{proof}
  The following lemma is a generalization of the fact related to the prime ideal of a polynomial ring over a UFD. 
  \begin{lemma}\label{lem: irreducible polynomial}
      Let $R$ be a Noetherian regular integral ring, let $a, b\in R\setminus\{0\}$, and let $P\subset R[t]$ be an ideal generated by $at+b$. 
      We assume that $\mathrm{div}(a)$ and $\mathrm{div}(b)$ have no common prime divisors of $\Spec(R)$. 
      Then $P$ is a prime ideal. 
  \end{lemma}
  \begin{proof}
      Let $\mathfrak{q}$ be a prime ideal of $R[t]$ and let $\mathfrak{p}$ denote a prime ideal $\mathfrak{q}\cap R$ of $R$. 
      By the assumption, the ring $R_{\mathfrak{p}}[t]$ is a UFD and $P\cdot R_{\mathfrak{p}}[t]$ is a prime ideal of $R_{\mathfrak{p}}[t]$. 
      In particular, $P\cdot R[t]_{\mathfrak{q}}$ is also a prime ideal or a trivial ideal of $R[t]_{\mathfrak{q}}$, and hence, $P$ is radical. 
      Let $P = \cap \mathfrak{q}_i$ be a minimal prime decomposition of $P$ and let $\mathfrak{p}_i$ denote  $\mathfrak{q}_i\cap R$ for each $i$. 
      By Krull's principal ideal theorem, the height of all $\mathfrak{q}_i$ is $1$, and the height of $\mathfrak{p}_i$ is $0$ or $1$. 
      If the height of $\mathfrak{p}_i$ is $1$, then $\mathfrak{q}_i = \mathfrak{p}_i[t]$ and $a, b\in \mathfrak{p}_i$, but it is a contradiction to the assumption. 
      Hence, all $\mathfrak{p}_i$ is $0$. 
      This indicates that all $\mathfrak{q}_i$ are on the generic fiber of $\Spec(R[t])\rightarrow \Spec(R)$. 
      Because the degree of $at+b$ is $1$, $P$ is a prime ideal.  
  \end{proof}
  The following lemma shows that general pairs in the two linear systems, which are generated by units, are co-prime. 
  \begin{lemma}\label{lem: no common divisor}
      Let $X$ be a smooth integral affine scheme of finite type over $k$, let $A$ be a global section ring of $X$, and let $\chi_1, \ldots, \chi_r, \chi'_1, \ldots, \chi'_s$ be units in $A$. 
      
      For $(a, b) = (a_1, \ldots, a_r, b_1, \ldots, b_s)\in \A^{r+s}_k(k)$, let $f$ denote $\sum_{1\leq i\leq r}a_i\chi_i$ and let $g$ denote $\sum_{1\leq j\leq s}b_j\chi'_j$. 
      Then there exists a nonempty open subset $V$ of $\A^{r+s}_k$ such that $\mathrm{div}(f)$ and $\mathrm{div}(g)$ have no common prime divisors of $X$ for any $(a, b)\in V(k)$. 
  \end{lemma}
  \begin{proof}
      Let $B$ be a polynomial ring $A[x_1, \ldots, x_r, y_1, \ldots y_s]$ over $A$, let $F$ denote $\sum_{1\leq i\leq r}x_i\chi_i\in B$, let $G$ denote $\sum_{1\leq j\leq s}y_j\chi'_j\in B$, let $Y$ denote $\Spec(B)$, let $Z$ denote the closed subscheme of $Y$ associated with the ideal generated by $F$ and $G$, and let $p\colon Y\rightarrow X$ and $q\colon Y\rightarrow \A^{r+s}_k$ be projections. 
      Then we can check that $p|_Z\colon Z\rightarrow X$ is a vector bundle of rank $r+s-2$ over an integral scheme $X$. 
      By the generic flatness, there exists a dense open subset $V$ of $\A^{r+s}_k$, such that $q|_{q^{-1}(V)\cap Z}\colon q^{-1}(V)\cap Z\rightarrow V$ is flat. 
      Let $(a, b)\in V(k)$ and let $W$ be a fiber of $q|_{q^{-1}(V)\cap Z}$ at $(a, b)$.  
      If $W = \emptyset$, $f$ and $g$ generate a trivial ideal of $A$, and hence, $\mathrm{div}(f)$ and $\mathrm{div}(g)$ have no common component. 
      In the other case, the dimension of the irreducible component of $W$ is $\dim(X) - 2$, and hence, $\mathrm{div}(f)$ and $\mathrm{div}(g)$ also have no common component. 
  \end{proof}
\bibliographystyle{amsplain}
\bibliography{yoshino-bib}
\end{document}